\documentclass[a4paper, 11pt]{article}

\usepackage{graphicx}
\usepackage[english]{babel}
\usepackage{subfigure}
\usepackage{amsmath,amsfonts,amssymb,mathabx,color}

\usepackage{amsthm}
\usepackage{mathrsfs}
\usepackage{multirow}
\usepackage{color}
\usepackage{bm}
\usepackage{url}
\usepackage{hyperref}

\newtheorem{theorem}{Theorem}[section]
\newtheorem{definition}[theorem]{Definition} 
\newtheorem{proposition}[theorem]{Proposition}
\newtheorem{lemma}[theorem]{Lemma}
\newtheorem{remark}[theorem]{Remark}
\newtheorem{notations}[theorem]{Notation}
\numberwithin{equation}{section}

\newcommand{\Enorm}[1]{\|#1\|\raisebox{-0.5ex}{\scalebox{0.5}{$E$}}}

\newcommand{\R}{\mathbb{R}}
\newcommand{\C}{\mathbb{C}}
\newcommand{\Z}{\mathbb{Z}}
\newcommand{\N}{\mathbb{N}}
\newcommand{\sphere}{\mathbb{S}}
\newcommand{\T}{\mathbb{T}}
\newcommand{\F}{\mathcal{F}}
\newcommand{\E}{\mathcal{E}}
\newcommand{\CC}{\mathcal{C}}
\newcommand{\B}{\mathcal{B}}
\newcommand{\set}{\mathcal{S}}

\newcommand{\tD}{\tilde D}

\newcommand{\bw}{\bar \omega}
\newcommand{\bW}{\bar W}
\newcommand{\bO}{\bar \Omega}
\newcommand{\bu}{\bar u}
 
\newcommand{\rednorm}[1]{\|#1\|_{\XXred}} 
\newcommand{\norm}[1]{\|#1\|_{\XX}}
\newcommand{\acts}{.}
\newcommand{\X}{X}
\newcommand{\XX}{\mathcal{X}}
\newcommand{\XXred}{\mathcal{X}^{\text{\textup{red}}}}
\newcommand{\ered}{\text{\textup{red}}}
\newcommand{\XXsym}{\mathcal{X}^{\text{\textup{sym}}}}
\newcommand{\Xsym}{\X^{\text{\textup{sym}}}}
\newcommand{\Xred}{\X^{\text{\textup{red}}}}

\newcommand{\J}{J}

\newcommand{\JJred}{\mathcal{J}^{\text{\textup{red}}}}
\newcommand{\Jdom}{\J^{\text{\textup{dom}}}}
\newcommand{\Jtriv}{\J^{\text{\textup{triv}}}}

\newcommand{\Jsym}{\J^{\text{\textup{sym}}}}
\newcommand{\Jred}{\J^{\text{\textup{red}}}}
\newcommand{\Jdag}{\J^{\dagger}}

\newcommand{\Cred}{C^{\text{\textup{red}}}}
\renewcommand{\j}{{j}}
\renewcommand{\c}{\omega}
\renewcommand{\b}{\phi}

\newcommand{\bb}{\bar{\varphi}}
\newcommand{\tc}{\widetilde{\omega}}
\newcommand{\talpha}{\widetilde{\alpha}}
\newcommand{\tg}{\widetilde{g}}
\newcommand{\Fred}{F^{\text{\textup{red}}}}
\newcommand{\Psired}{\Psi^{\text{\textup{red}}}}
\newcommand{\FFred}{\mathcal{F}^{\text{\textup{red}}}}
\newcommand{\e}{\mathrm{e}}
\newcommand{\weight}{\xi}
\newcommand{\symweight}{\xi^\text{\textup{s}}}

\newcommand{\tn}{\tilde{n}}
\newcommand{\bydef}{\stackrel{\mbox{\tiny\textnormal{\raisebox{0ex}[0ex][0ex]{def}}}}{=}}
\newcommand{\ta}{\tilde{a}}
\newcommand{\tC}{\tilde{C}}

\newcommand{\hG}{\widehat{G}}
\newcommand{\tgamma}{\widetilde{\gamma}}
\newcommand{\RR}{\mathcal{R}}
\newcommand{\Av}{\mathcal{A}}

\newcommand{\phase}{\leftmoon}
\renewcommand{\div}{\text{\textup{div}}}

\newcommand{\sol}{\text{\textup{sol}}}
\newcommand{\NN}[1]{\mathcal{N}_{#1}} 

\newcommand{\rhs}{f^\omega}
\newcommand{\approximate}{\text{\textup{approx}}}
\newcommand{\Adag}{\widehat{A}}
\newcommand{\Ared}{A_{\text{\textup{red}}}}
\newcommand{\Bred}{B_{\text{\textup{red}}}}
\newcommand{\Tred}{T_{\text{\textup{red}}}}
\newcommand{\Areddag}{\widehat{A}_{\text{\textup{red}}}}
\newcommand{\tildeset}{\widetilde{\set}}
\newcommand{\NNred}[1]{\mathcal{N}_{#1}^{\text{\textup{red}}}} 
\newcommand{\Scode}{\texttt{S}}

\begin{document}

\title{
Spontaneous periodic orbits in the Navier-Stokes flow
}

\author{
Jan Bouwe van den Berg 
\thanks{
Department of Mathematics, 
VU Amsterdam, 
1081 HV Amsterdam, 
The Netherlands, {\tt janbouwe@few.vu.nl};
partially supported by NWO-VICI grant 639033109.
}
\and 
Maxime Breden 
\thanks{
Faculty of Mathematics, 
Technical University of Munich, 
85748 Garching bei M\"unchen, 
Germany, {\tt maxime.breden@tum.de}; partially supported by a Lichtenberg Professorship grant of the VolkswagenStiftung awarded to C. Kuehn.
}
\and
Jean-Philippe Lessard 
\thanks{
Department of Mathematics and Statistics, 
McGill University, 
805 Sherbrooke St W, 
Montreal, QC, H3A 0B9, 
Canada, {\tt jp.lessard@mcgill.ca}; supported by NSERC. 
}
\and 
Lennaert van Veen
\thanks{
Faculty of Science, 
University of Ontario Institute of Technology, 
Oshawa, ON L1H 7K4, 
Canada, {\tt lennaert.vanveen@uoit.ca}; supported by NSERC.
}
}


\maketitle

\vspace{-.5cm}

\begin{abstract}
In this paper, a general method to obtain constructive proofs of existence of periodic orbits in the forced autonomous Navier-Stokes equations on the three-torus is proposed. After introducing a zero finding problem posed on a Banach space of geometrically decaying Fourier coefficients, a Newton-Kantorovich theorem is applied to obtain the (computer-assisted) proofs of existence. The required analytic estimates to verify the contractibility of the operator are presented in full generality and symmetries from the model are used to reduce the size of the problem to be solved. As applications, we present proofs of existence of spontaneous periodic orbits in the Navier-Stokes equations  with Taylor-Green forcing.
\end{abstract}

\begin{center}
{\bf \small Keywords} \\ \vspace{.05cm}
{ \small Navier-Stokes equations $\cdot$ periodic orbits $\cdot$ symmetry breaking $\cdot$ computer-assisted proofs }
\end{center}

\begin{center}
{\bf \small Mathematics Subject Classification (2010)}  \\ \vspace{.05cm}
{\small 35Q30 $\cdot$ 35B06 $\cdot$ 35B10 $\cdot$ 35B36 $\cdot$ 65G20 $\cdot$ 76D17 } 
\end{center}


\section{Introduction}
\label{s:introduction}

The Navier-Stokes equations for a fluid of constant density $\rho$ can be expressed as
\begin{equation}
\label{eq:NS}
\left\{
\begin{aligned}
\partial_t u + (u\cdot \nabla)u - \nu\Delta u + \nabla p &= f \\
\nabla \cdot u &= 0 ,
\end{aligned}
\right.
\end{equation}
where $u=u(x,t)$ is the velocity, $p(x,t)=P(x,t)/\rho$ is the pressure scaled by the density, $\nu$ is the kinematic viscosity and $f=f(x,t)$ is an external forcing term. 
These equations can be considered on compact or unbounded domains, complemented by boundary and initial conditions. The first equation, which expresses momentum balance, has a quadratically nonlinear advection term. While the presence of the nonlinearity generally obstructs obtaining closed form solutions, there are some notable exceptions.
In parallel shear flows, the advection term vanishes identically and analytic solutions are available. Examples include Hagen-Poiseuille flow in pipes~\cite{Sutera1993} and Taylor-Couette flow between co-axial cylinders \cite{Taylor1923}. In Beltrami flows, the nonlinearity takes the form of a gradient and can be absorbed in the pressure. An example of an explicit solution with this property is the ABC flow \cite{Dombre1986}.
Explicit 
solutions with non-trivial nonlinearities are rare, but some are known. For instance, exact vortical solutions include Burgers' vortex on~$\mathbb{R}^3$ \cite{Burgers1948} and the periodic vortex array of Taylor and Green \cite{Taylor1937}. However, it has been known for a long time that, even in fluids with strong viscous damping, more complicated, time-periodic motions can occur. A classical experiment is that of a fluid flowing past a stationary cylinder. Experiments started by von K\'arm\'an at the beginning of the twentieth century, and carried on by his students, showed that, at a well-defined flow rate, the motion in the wake of the cylinder becomes time-periodic, as alternating clockwise and counter clockwise vortices travel downstream~\cite{Kovasznay1949}. 

While there is little hope of writing down explicit solutions that describe such oscillatory behaviour, a number of authors have attempted to at least establish the existence of time-periodic solutions. The earliest contribution was likely the work of James Serrin. In 1959, he published two papers on the existence and stability of certain solutions to the Navier-Stokes equations in the limit of large viscosity. In the first one, he established the existence of globally stable equilibrium solutions by finding bounds for the nonlinear and forcing terms, and by showing that a certain energy decays \cite{Serrin1959a}.
In the second one, he considered large viscosity and gave a criterion for the existence of periodic solutions on a three-dimensional bounded domain subject to time-periodic boundary data and body forces~\cite{Serrin1959b}.

Many authors followed Serrin in studying the periodically forced non-autonomous Navier-Stokes system dominated by viscosity. Kaniel and Shinbrot \cite{Kaniel1967} considered bounded domains with fixed boundaries and showed the existence of periodic strong solutions for small time-periodic forcing $f$. Without making any assumption about the size of $f$, Takeshita \cite{MR0264254} showed the same result as Kaniel and Shinbrot. Some time later, Teramoto \cite{MR725967} proved the existence of time-periodic solutions for domains with slowly moving boundaries. Then, Maremonti \cite{Maremonti1991} and Kozono and Nakao \cite{Kozono1996} extended the results from bounded domains to $\mathbb{R}^3$. The latter made use of the $L^p$ theory of the Stokes operator rather than the energy method. A similar result, relying on a milder condition on the forcing function, was derived by Kato \cite{Kato1997}. Other extensions were those to inhomogeneous boundary conditions on compact domains by Farwig and Okabe \cite{Farwig2010} and to the case of a rotating fluid in two dimensions by Hsia {\sl et al.} \cite{Hsia2017}. The latter paper also contains a fairly extensive list of references of which only a fraction is discussed here. 

Thus, our understanding of periodic flows in response to time-periodic forcing is rather advanced. The same cannot be said about \emph{spontaneous} periodic motions, which we refer to as being periodic flows driven by a time-\emph{independent} forcing. In other words, spontaneous periodic motions are periodic orbits in the autonomous Navier-Stokes equations. The regular vortex shedding in the wake of a cylinder, for instance, arises in the absence of a body force and \emph{as a consequence of} the nonlinearity of the Navier-Stokes equation, not by virtue of the advection being dominated by viscous damping. In an attempt to address the difficulties in studying spontaneous motions, the present paper proposes a general (computer-assisted) approach to prove existence of time-periodic Navier-Stokes flows on the three-torus for given time-independent forcing terms $f=f(x)$.

The novelty of our paper is threefold. Foremost, it provides the first computer-assisted proof of existence of spontaneous periodic Navier-Stokes flows. Second, it introduces general analytic bounds applicable to prove existence of three-dimensional time-periodic solutions for any time-independent forcing term. Third, it uses the symmetries present in Navier-Stokes to significantly reduce the size of the problem to work with. 

A few comments on the symmetries are in order. Our approach permits one to take advantage of symmetries of the forcing $f$, and in particular of those (subgroup of) symmetries that are also obeyed by the examined solution. 
We allow general time-independent forcings with zero spatial average (so that periodic solutions are not ruled out a priori, 
see Equation~\eqref{e:spatialaverageforcing}). Any time-periodic solution thus spontaneously breaks the shift symmetry in time,
and other symmetries of~$f$ may also be broken by the solution. However, the bigger the symmetry group of the solution is, the more we can reduce the computational cost (in terms of time and, especially, memory). 

While all of the analysis is performed in full generality on the 3-torus, the solutions we present in Theorem~\ref{thm:NS_result} below are two-dimensional (in space) time-periodic solutions. Indeed, they are homogeneous in one spatial variable and can thus be interpreted as solutions on the 2-torus.
The only reason for this reduction is that the physical memory requirements for a three-dimensional solution are, for now, prohibitive in our current implementation.
To be precise, we consider the so-called Taylor-Green forcing
\begin{equation}
\label{eq:TG_intro}
  f=f(x)=\begin{pmatrix} \hfill 2 \sin x_1 \cos x_2 \\ - 2 \cos x_1 \sin x_2 \\0
  \end{pmatrix},
\end{equation}
which corresponds to counter rotating vortex columns. Clearly, this forcing allows one to restrict to the first two spatial variables.
It is expected that some periodic solutions in fact break the 2D symmetry of the forcing~\eqref{eq:TG_intro}, see also Section~\ref{s:application}.
While we aim to investigate such solutions in future work,
the solutions obtained in the current paper respect the 2D symmetry: they are independent of $x_3$ and the third component of the velocity vanishes.
We call such a solution an (essentially) 2D solution. In addition, the solutions that we find here are invariant under a symmetry group with 16 elements, see Section~\ref{s:application} for details. This allows us to reduce the number of independent Fourier modes on which we perform the computational analysis by a factor 16, which represents considerable savings in memory requirements. Finally, due to the shift-invariance of the torus, in general it may be appropriate to look for solutions which are shift-periodic, but such complications do not arise when studying 2D solutions for the forcing~\eqref{eq:TG_intro}. 

Before we state a representative sample result, we note that determining the period of the solution is part of the problem. Hence the frequency $\bar{\Omega}$ of a numerical approximate solution~$(\bar{u},\bar{p})$ only approximates the true frequency $\Omega$. The solution of~\eqref{eq:NS} will therefore be close to a slightly time-dilated version $(\bar{u}_\theta,\bar{p}_\theta)$ of
the numerical data, where $\theta$ is the dilation factor, see Remark~\ref{rem:comparison_of_norms}.
As outlined below in more detail, we use a Newton operator to show that, under computable conditions, there is a solution to~\eqref{eq:NS} near the numerically obtained approximation~$(\bar{u},\bar{p})$, where the error is bounded explicitly. 
As an example, we prove the following result. 
\begin{theorem} \label{thm:NS_result}
Consider \eqref{eq:NS} defined on the three-torus $\T^3$ (with size length $L=2\pi$) and consider the time-independent forcing term~\eqref{eq:TG_intro}.
Let $\nu = 0.265$ and $(\bar u,\bar p)$ be the numerical solution whose Fourier coefficients and time frequency $\bar \Omega$ are given in the file \verb+dataorbit2.mat+ and can be downloaded at \cite{navierstokescode} (and whose vorticity is represented in Figure~\ref{2DPO_intro}). 
Let $r_{\sol}^\Omega=2.2491 \cdot 10^{-6}$, $r_{\sol}^u=2.2491 \cdot 10^{-6}$, and $r_{\sol}^p=5.6486 \cdot 10^{-5}$.
There exists a $\frac{2 \pi}{\Omega}$-periodic solution $(u,p)$ of \eqref{eq:NS} with $|\Omega-\bar \Omega|\le r_{\sol}^\Omega$ and such that
\[
\|u-\bar{u}_{\Omega/\bar{\Omega}}\|_{\CC^0} \le r_{\sol}^u
\quad \text{and} \quad 
\|p-\bar{p}_{\Omega/\bar{\Omega}}\|_{\CC^0} \le r_{\sol}^p.
\]
\end{theorem}

We point out that the $\CC^0$-norm is only used here to get a simple statement. A more general version of Theorem~\ref{thm:NS_result}, with a stronger norm which is the one actually used in the analysis, is presented in Theorem~\ref{thm:result2} in Section~\ref{s:application}.

\begin{figure}[t]
\begin{center}
\hspace{2cm}
\includegraphics[width=0.8\textwidth]{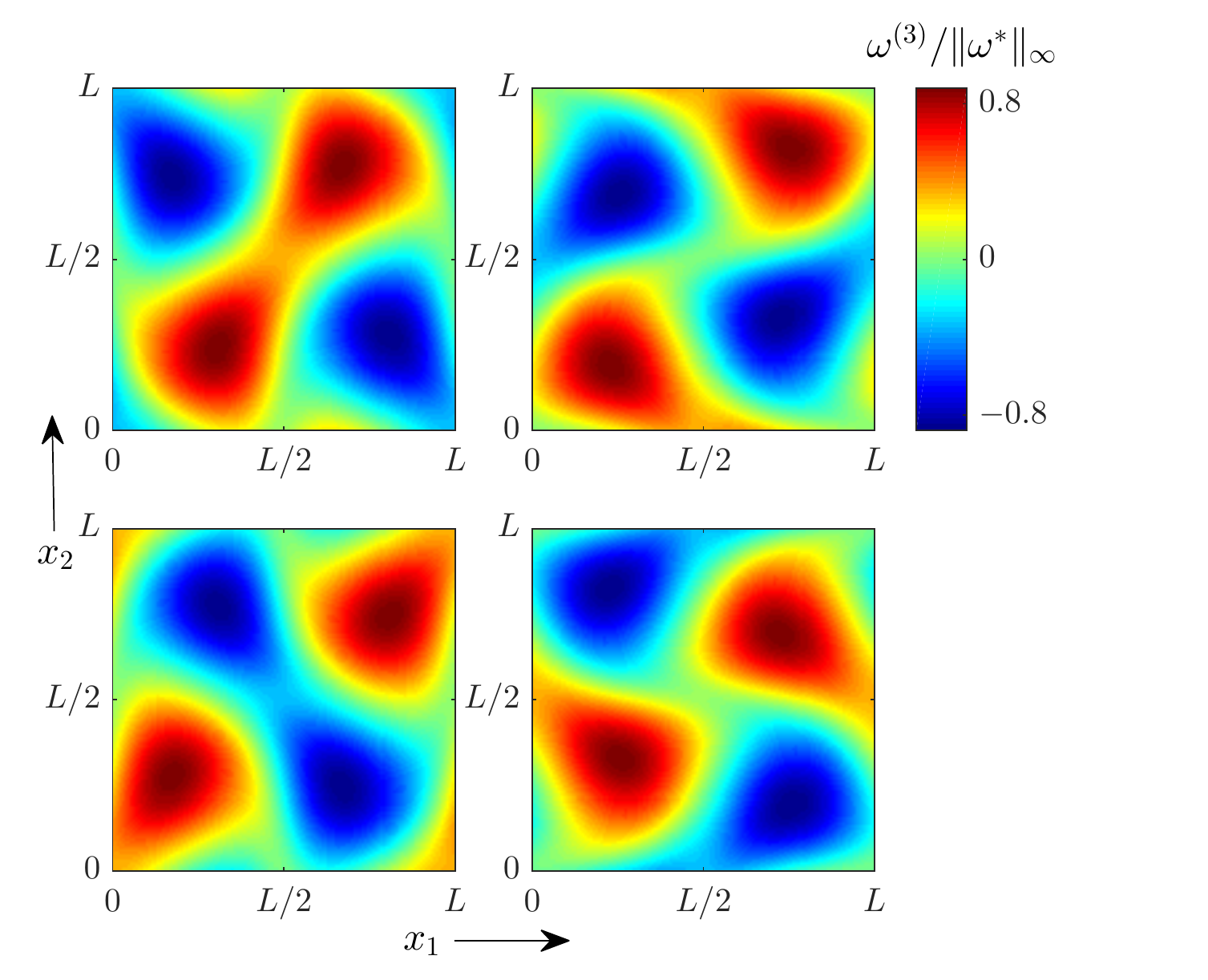}
\end{center}
\vspace{-.5cm}
\caption{The third component of $\bw = \nabla \times \bar u$ of the spontaneous periodic flow obtained in Theorem~\ref{thm:NS_result}, normalized by the amplitude of the equilibrium solution defined in \eqref{visc_eq}. Snap shots are depicted at times $0$, $\frac{\pi}{2\bO}$, $\frac{\pi}{\bO}$ and $\frac{3\pi}{2\bO}$.}
\label{2DPO_intro}
\end{figure}

It is important to recognize that in the last forty years, important open problems were settled with computer-assisted proofs: the universality of the 
Feigenbaum constant \cite{feigenbaum}, the four-colour theorem \cite{fourcolor}, the existence of the strange attractor in the Lorenz system \cite{lorenz} (i.e.\ Smale's 14th problem) and Kepler's densest sphere packing problem \cite{kepler}. We refer the interested to reader to the expository works \cite{MR3444942,gomez_survez,MR1420838,jay_konstantin_survey,MR1849323,Plu01,MR2652784,MR2807595} 
and the references therein, for a more complete overview of the field of rigorously verified numerics. Let us however mention some results related to the present work. In \cite{Watanabe1999}, Watanabe proposes an approach to obtain computer-assisted proofs of existence of stationary solution in the Navier-Stokes equation, which then lead to the rigorous computation of equilibria in a three-dimensional thermal convection problem \cite{Kim2009} and in Kolmogorov flow, i.e. flow with periodic boundary conditions and a constant body force with a simple structure \cite{Watanabe2009,Watanabe2016}. Independently, Heywood {\sl et al.} \cite{Heywood1999} established fixed-point theorems for steady, two-dimensional Kolmogorov flows. Their results fall short of a proof of existence only because of the presence of round-off error, which Watanabe avoided by using interval arithmetic. The rigorous computation of time-dependent solutions to the autonomous Navier-Stokes equation has so far been out of reach. We note that computer-assisted proofs for periodic orbits, along lines similar to the current paper, have been obtained for the Kuramoto-Shivashinsky PDE \cite{AriKoc10,FigLla17,GamLes17,Zgl04} and the ill-posed Boussinesq equation \cite{CasGamLes18}. 
 
Our strategy begins by identifying a problem of the form $\F(W)=0$ posed on a Banach algebra of geometrically decaying Fourier coefficients, whose solutions yield the time-periodic orbits. This zero finding problem is derived by applying the curl operator to \eqref{eq:NS} and solving for the periodic orbits in the vorticity equation. Expressing a periodic orbit using a space-time Fourier series and plugging the series in the vorticity equation yields the infinite dimensional nonlinear problem $\F(W)=0$, where $W$ corresponds to the sequence of Fourier coefficients of the vorticity $\omega = \nabla \times u$. The details of the derivation of the map $\F$ are given in Section~\ref{sec:zero finding problem}. A proof that the solutions of $\F=0$ correspond to time-periodic Navier-Stokes flows is presented in Lemma~\ref{lem:eq_NS_vorticity}. The next step is to consider a finite dimensional projection of $\F$ and to numerically obtain an approximation $\bW$ of a zero of $\F$, that is $\F(\bW) \approx 0$. 
Next, we turn the problem $\F(W)=0$ into an equivalent fixed point problem of the form $T(W) = W - D\F(\bW)^{-1} \F(W)$. We then set out to prove that $T$ is a contraction on a neighborhood of $\bW$. The advantage is that instead of trying to prove \emph{equalities} in the formulation $\F(W)=0$, contractivity involves \emph{inequalities} only. The proof then proceeds by a Newton-Kantorovich type argument (see Theorem~\ref{th:radii_pol} and Theorem~\ref{th:radii_pol_sym}) to find a ball centered at $\bW$ on which the map $T$ is a contraction mapping. Having done the hard work in the analysis of reducing the problem to finitely many explicit inequalities, one therefore resorts to interval arithmetic computer calculations for this final step of the proof. 

The paper is organized as follows. In Section~\ref{s:setup}, we introduce the
rigorous computational approach and the zero finding problem $\F(W)=0$, as well
as the Banach space in which we solve for the zeros of $\F$. The
Newton-Kantorovich type theorem is presented in Theorem~\ref{th:radii_pol}. In
Section~\ref{s:estimates}, we introduce the general bounds necessary to verify
the hypotheses of Theorem~\ref{th:radii_pol}. Then in
Section~\ref{s:symmetries} we describe how the symmetries of the model can be
used to simplify solving the zero-finding problem, by reducing significantly
its size. Using this reduction based on symmetries a modified
Newton-Kantorovich theorem is proved (Theorem~\ref{th:radii_pol_sym}) and the
associated symmetry-adapted estimates are derived in
Section~\ref{s:symmetrybounds}. Sample results are then presented in Section~\ref{s:application}.
All the estimates obtained in this paper
culminate in Theorems~\ref{thm:result1} and~\ref{thm:result2}, which allows us to validate periodic solutions $\omega$ of the vorticity equation, with explicit error bounds. In the Appendix we describe how to recover errors bounds for the associated velocity $u$ and pressure $p$ that solve the Navier-Stokes
equations.

\section{The rigorous computational approach}
\label{s:setup}

This section is devoted to the presentation of the framework that is needed to study periodic solutions of~\eqref{eq:NS} by computer-assisted means. We first derive a suitable $\F=0$ problem in Section~\ref{sec:zero finding problem} and introduce the proper Banach spaces to study that problem in Section~\ref{sec:notations}. Well chosen approximations of $D\F$ and $D\F^{-1}$ are then introduced in Section~\ref{s:AAdag}, and used in Section~\ref{s:radpol} to state Theorem~\ref{th:radii_pol}, which provides us with sufficient conditions for the existence of non trivial zeros of $\F$.

\subsection{The zero finding problem in Fourier space} \label{sec:zero finding problem}

In this section, we introduce the zero finding problem $\F=0$ that we are going to work with. We start by some (somewhat algebraic) manipulations and then explain in Lemma~\ref{lem:eq_NS_vorticity} how the zero finding problem is related to the original Navier-Stokes equations.

We consider the 3D incompressible Navier-Stokes equations \eqref{eq:NS} on the three-torus $\T^3$ and look for time-periodic solutions. As mentioned in the introduction, both the numerical and the theoretical part of our work are based on Fourier series, for which we will use the following notations. For $n=(n_1,n_2,n_3,n_4)\in\Z^4$, we write $n=(\tn,n_4)$ where $\tn=(n_1,n_2,n_3)$ and $\tn^2 \bydef n_1^2+n_2^2+n_3^2$. If $u:\T^3\times\R\to\R^3$ is a time-periodic function (that is periodic in the fourth variable), we denote by $\left(u_n\right)_{n\in \Z^4}\in \left(\C^3\right)^{\Z^4}$ its Fourier coefficients:
\begin{equation*}
\label{eq:function_coeffs_u}
u(x,t)=\sum_{n\in\Z^4} u_n e^{i(\tn\cdot x + n_4\Omega t)},
\end{equation*}
where $\Omega$ is the a-priori unknown angular frequency. 
\begin{remark}
In this paper, we are only concerned with smooth (that is analytic) periodic functions. Therefore, we can identify a function with its sequence of Fourier coefficients, and to make the notations lighter we use the same symbol to denote both of them. It should be clear from context whether $u$  (and similarly later for $\omega$, $f$, $\rhs$, etc.) denotes a periodic function or a sequence of Fourier coefficients.
\end{remark}
For $1\leq l\leq 3$, we use $u^{(l)}\in\C^{\Z^4}$ to denote the Fourier sequence of the $l$-th component of~$u$. For any Fourier sequence $a=(a_n)\in\C^{\Z^4}$ and any $1\leq l\leq 3$ we define the sequence $D_l a$ corresponding to the partial derivative of $a$ with respect to $x_l$ (up to a factor $i$): 
\begin{equation*}
\left(D_l a\right)_n \bydef n_l a_n \qquad \text{for all } n\in\Z^4.
\end{equation*}
For any Fourier sequence $a=(a_n)$ and $b=(b_n)$ in $\C^{\Z^4}$ we define their convolution product as
\begin{equation*}
(a\ast b)_n \bydef \sum_{k\in\Z^4}a_k b_{n-k}.
\end{equation*}
Finally, given $a,b\in\left(\C^3\right)^{\Z^4}$ we define
\begin{equation*}
\left[\left(a \star \tD\right) b\right]^{(l)} \bydef \sum_{m=1}^3 a^{(m)} \ast \left(D_m b^{(l)}\right) \qquad \text{for all } 1\leq l\leq 3,
\end{equation*}
which is the $l$-th component in Fourier space of $(a\cdot\nabla)b$, again up to a factor $i$. We will frequently use the following lemma.
\begin{lemma}
\label{lem:trick_derivative}
Let $a,b\in \left(\C^3\right)^{\Z^4}$ be such that $\sum_{m=1}^3 D_m a^{(m)}=0$. Then, for all $l\in\{1,2,3\}$,
\begin{equation*}
\left[\left(a \star \tD\right) b\right]^{(l)} = \sum_{m=1}^3 D_m \left[a^{(m)}\ast b^{(l)}\right].
\end{equation*}
\end{lemma}
\begin{proof}
This is a consequence of the product rule:
\begin{align*}
\left[\left(a \star \tD\right) b\right]^{(l)} &= \sum_{m=1}^3 \left[a^{(m)} \ast D_m b^{(l)}\right] \\
&= \sum_{m=1}^3 \left(D_m\left[a^{(m)} \ast b^{(l)}\right] - \left[D_m a^{(m)} \ast b^{(l)}\right]\right) \\
&= \sum_{m=1}^3 D_m\left[a^{(m)} \ast b^{(l)}\right]. \qedhere
\end{align*}
\end{proof}

We are now almost ready to set up our $\F=0$ problem, but instead of looking directly at periodic solutions of~\eqref{eq:NS} we are going to work with the vorticity equation. Namely, we consider the vorticity $\omega = \nabla \times u$ and look for the equation it satisfies. Using 
\begin{equation*}
(u\cdot \nabla)u = \nabla \left(\frac{|u|^2}{2}\right)-u\times \omega,
\end{equation*}
we get
\begin{align}
\label{eq:curl_of_nonlin}
\nabla \times \left((u\cdot \nabla)u\right) &= \nabla \times \left(\omega\times u\right) \nonumber\\
&=  \left(u\cdot\nabla\right)\omega - \left(\omega\cdot\nabla\right)u + \omega\left(\nabla\cdot u\right) - u\left(\nabla\cdot\omega\right),
\end{align}
and since both $u$ and $\omega$ are divergence free we end up with
\begin{equation}
\label{eq:curl_simplified}
\nabla \times \left((u\cdot \nabla)u\right) = \left(u\cdot\nabla\right)\omega - \left(\omega\cdot\nabla\right)u. \end{equation}
The vorticity equation is then given by
\begin{equation}
\label{eq:vorticity}
\partial_t \omega + \left(u\cdot\nabla\right)\omega - \left(\omega\cdot\nabla\right)u - \nu\Delta \omega = \rhs \quad \text{on } \T^3\times\R,
\end{equation}
where $\rhs \bydef \nabla\times f$.

We are going to solve for the Fourier coefficients of $\omega$ satisfying the vorticity equation~\eqref{eq:vorticity}. More precisely, our unknowns are the Fourier coefficients $\left(\omega_n\right)_{n\in\Z^4}$ and the angular frequency $\Omega$. However, in~\eqref{eq:vorticity} the unknowns $\omega$ and $u$ both appear. To obtain an equation depending on the vorticity $\omega$ only, we need to express $u$ in term of $\omega$ by solving
\begin{equation*}
\left\{
\begin{aligned}
&\nabla\times u=\omega \\
&\nabla\cdot u=0.
\end{aligned}
\right.
\end{equation*}
Applying the curl operator to the first equation, and using that $\nabla\cdot u=0$, we get
\begin{equation*}
-\Delta u = \nabla\times \omega,
\end{equation*}
and so, formally,
\begin{equation}
\label{e:u_from_omega}
u=(-\Delta)^{-1}\nabla\times\omega.
\end{equation}
\begin{remark}\label{r:zerospatialaverage}
Expression \eqref{e:u_from_omega} is not completely well defined because the Laplacian has, in general, a non-zero kernel. In particular, the space average velocity 
\begin{equation*}
\int_{\T^3} u(x,t) dx = \sum_{n_4\in\Z} u_{0,n_4}e^{in_4\Omega t}
\end{equation*}
cannot be recovered from the vorticity $\omega$. However, going back to~\eqref{eq:NS} we have
\begin{equation}\label{e:spatialaverageforcing}
\frac{d}{dt}\int_{\T^3} u(x,t) dx = \int_{\T^3} f(x) dx.
\end{equation}
In this work, we consider a time independent forcing with spatial average equal to zero, therefore the space average velocity is a conserved quantity, which can always be assumed to be zero by Galilean invariance.
\end{remark}

To give a well defined version of~\eqref{e:u_from_omega}, we go through Fourier space and introduce
\begin{equation}\label{e:defM}
M_n \bydef 
\begin{cases}
\displaystyle
\frac{i}{\tn^2} \begin{pmatrix}
0 & -n_3 & n_2\\
n_3 & 0 & -n_1\\
-n_2 & n_1 & 0\\
\end{pmatrix},
& \tn\neq 0, \\
0, &\tn= 0,
\end{cases}
\end{equation}
and
\begin{equation*}
M\omega \bydef \left(M_n\omega_n\right)_{n\in\Z^4},
\end{equation*}
As mentioned previously, we can assume the space average velocity to be zero, which is why we define $(M\omega)_{0,n_4}$ to be zero for all $n_4\in\Z$. 
\begin{remark}
The above computations and construction of $M$ can be summarized in the following equivalence
\begin{equation*}
\label{rem:equivalence_u_omega}
  \left\{
  \begin{array}{l}
	   \nabla \times u = \omega \\
	   \nabla \cdot u =0 \\
	   \int_{\T^3} u=0
   \end{array}
   \right.
   \qquad\Longleftrightarrow\qquad
  \left\{
  \begin{array}{l}
   u = M \omega \\
   \nabla \cdot \omega =0.
   \end{array}
   \right.   
\end{equation*}
\end{remark}
Going back to~\eqref{eq:vorticity} and replacing $u$ by $M\omega$, we obtain the equation
\begin{equation}
\label{eq:F_physical_space}
\partial_t \omega + \left(M\omega\cdot\nabla\right)\omega - \left(\omega\cdot\nabla\right)M\omega - \nu\Delta \omega = \rhs \quad \text{on } \T^3\times\R,
\end{equation}
which is the one we are going to work with. More precisely, we first define \begin{equation}\label{eq:W}
W= \begin{pmatrix}
\Omega \\
\left(\omega_{n}\right)_{n\in\Z^4\setminus\{0\}}
\end{pmatrix},
\end{equation}
which corresponds to all the unknowns we are solving for. Notice that $\omega_0$ is not part of the unknowns, as it can always be taken equal to $0$. In the sequel, to simplify the presentation we introduce $\Z^4_*=\Z^4\setminus\{0\}$, and always identify $\left(\omega_{n}\right)_{n\in\Z^4_*}$ with $\left(\omega_{n}\right)_{n\in\Z^4}$ where $\omega_0=0$. In particular, the representation in physical space associated to $W$ is given by
\begin{equation}
\label{eq:function_coeffs_omega}
\omega(x,t)= \sum_{n\in\Z^4_*} \omega_n e^{i(\tn\cdot x + n_4\Omega t)}
= \sum_{n\in\Z^4} \omega_n e^{i(\tn\cdot x + n_4\Omega t)}.
\end{equation}
We then define $F=\left(F_n\right)_{n\in\Z^4_*}$ by
\begin{equation}\label{e:defFn}
F_n(W) \bydef
i\Omega n_4\omega_n + i\left[\left(M\omega\star \tD\right) \omega\right]_n -i\left[\left(\omega \star \tD\right) M\omega\right]_n +\nu\tn^2\omega_n - \rhs_n ,
\end{equation}
for all $n \in \Z^4_*$, and aim to show the existence of a nontrivial zero of $F$.

For later use (see Section~\ref{s:Z1bound} and Section~\ref{s:symboundZ1}), we also introduce the notation $\Psi=\left(\Psi_n\right)_{n\in\Z^4_*}$ for the nonlinear terms in~\eqref{e:defFn}:
\begin{equation}\label{e:defPsi}
\Psi_n(\omega) \bydef
i\left[\left(M\omega\star \tD\right) \omega\right]_n -i\left[\left(\omega \star \tD\right) M\omega\right]_n  \qquad \text{for all } n \in \Z^4_* .
\end{equation}

\begin{lemma}
\label{lem:eq_NS_vorticity}
Let $W\in\R\times(\C^3)^{\Z^4_*}$ be such that the corresponding function $\omega$ is analytic. Assume that $F(W)=0$ and $\nabla\cdot\omega=0$. Assume also that $f$ does not depend on time and has space average zero. Define $u=M\omega$. Then there exists a pressure function $p: \T^3\times\R \to \R$ such that $(u,p)$ is a $\frac{2\pi}{\Omega}$-periodic  solution of~\eqref{eq:NS}.
\end{lemma}
\begin{remark}
\label{rem:eq_NS_vorticity}
Let us mention that the analyticity condition could be considerably weakened: the lemma would hold for any smoothness assumption allowing to justify all the taken derivatives and the switches between functional and Fourier representation. In our case analyticity simply happens to be the most natural assumption, because of the space of Fourier coefficients we end up using, see Section~\ref{sec:notations} and Remark~\ref{rem:analyticity}.
\end{remark}
\begin{proof}
First, a straightforward computation gives that
\begin{equation*}
D_1\left(M\omega\right)^{(1)} + D_2\left(M\omega\right)^{(2)} + D_3\left(M\omega\right)^{(3)} =0,
\end{equation*}
for any $\omega$, which amounts to saying that $\nabla\cdot M\omega=0$, therefore $\nabla\cdot u=0$. 

The next step is to prove that $\nabla\times u=\omega$ (that is~$\nabla\times M\omega=\omega$). Using the definition of $M_n$ and the fact that $\nabla\cdot\omega=0$, another straightforward computation in Fourier space shows that $\left(\nabla\times M\omega\right)_n=\omega_n$ for all $\tn\neq 0$.
To prove that $\left(\nabla\times M\omega\right)_n=\omega_n$ for all $\tn= 0$,
first observe that $\left(\nabla\times M\omega\right)_n =0$ for $\tn= 0$.
Next, since $\omega$ and $M\omega$ are divergence free, we can use Lemma~\ref{lem:trick_derivative} on both nonlinear terms of $F$, and get that 
\begin{equation*}
\left[\left(M\omega\star \tD\right) \omega\right]_{0,n_4} =0 = \left[\left(\omega \star \tD\right) M\omega\right]_{0,n_4}
\qquad \text{for all } n_4\in\Z.
\end{equation*}
Therefore, $F_{0,n_4}(W)=0$ implies that $\omega_{0,n_4}=0$ for $n_4\neq 0$, which concludes the proof that $\nabla\times u=\omega$. 

Finally, we define the periodic function
\begin{equation}\label{e:defPhi}
\Phi \bydef \partial_t u + (u\cdot \nabla)u - \nu\Delta u - f.
\end{equation}
Using $\nabla\times u = \omega$ we get that, for all $n\neq 0$, $\left(\nabla\times\Phi\right)_n=F_n(W)$, and since $(\nabla\times\Phi)_0$ vanishes as well, we conclude that $\nabla\times\Phi=0$. Recalling that $u=M\omega$ is divergence free, Lemma~\ref{lem:trick_derivative} yields that
\begin{align*}
\left[\left(M\omega\star\tD\right) M\omega\right]^{(l)}_{0,n_4} =0, \qquad \text{for all } n_4\in\Z.
\end{align*}
Also using that $\left(M\omega\right)_{0,n_4}=0$ and that $f_{0,n_4}=0$ (since $f$ does not depend on time and has average zero), we get that $\Phi_{0,n_4}=0$ for all $n_4\in\Z$. We have shown that
\begin{equation*}
 \left\{
 \begin{aligned}
 &\nabla\times\Phi=0 \\
 &\Phi_{n}=0,\qquad\text{for all } \tn=0.
 \end{aligned}
 \right.
 \end{equation*} 
Therefore (see Lemma~\ref{lem:exist_grad}) there exists a $p$ such that $\Phi=-\nabla p$, that is
\begin{equation*}
\partial_t u + (u\cdot \nabla)u - \nu\Delta u + \nabla p = f.
\end{equation*}
Since we had already shown that $u$ is divergence free, this completes the proof.
\end{proof}

Finally, in view of the contraction argument that we are going to apply, we add a phase condition in order to isolate the solution. 
Assume that we have an approximate periodic orbit given by $\hat\omega$. 
A common choice for the phase condition (see e.g.~\cite{kuznetsov2013elements}) is to require that the orbit $\omega$ satisfies
\begin{equation*}
\int\limits_0^{\frac{2\pi}{\Omega}}\int\limits_{\T^3} \omega(x,t) \cdot \partial_t\hat\omega(x,t)~dx dt =0.
\end{equation*}
Hence, we define the phase condition
\begin{equation}
\label{eq:phase_condition}
F_{\phase}(W)=i\sum_{l=1}^3\sum_{n\in\Z^4_*}\omega_n^{(l)}n_4\left(\hat\omega_n^{(l)}\right)^*,
\end{equation}
where the superscript ${}^*$ denotes complex conjugation, and we assumed that the reference orbit given by $\hat{\omega}$ is real-valued, that is $\hat\omega_{-n}^{(l)}=\left(\hat\omega_n^{(l)}\right)^*$.

We now consider the enlarged problem
\begin{equation}\label{eq:fullFnosymmetry}
\F= \begin{pmatrix}
F_{\phase} \\
\left(F_{n}\right)_{n\in\Z^4_*}
\end{pmatrix}.
\end{equation}
In view of the similarity between~\eqref{eq:W} and~\eqref{eq:fullFnosymmetry},
we will abuse notation and refer to the set of elements of such variables via
\[
  \{ Q_n \}_{n \in \{\phase,\Z^4_* \}} \,,
\]
where $Q_{\phase} \in \C$ and $Q_n \in \C^3$ for $n\in\Z^4_*$,
rather than introducing notation for the projections onto different components.

\subsection{Banach spaces, norms, index sets}
\label{sec:notations}

In this section we introduce the Banach spaces on which we are going to study $\F$, as well as some additional notations that are going to be used throughout this paper.

For $\eta\geq 1$, we denote by $\ell^1_\eta(\C)$ the subspace of all sequences $a\in\C^{\Z^4_*}$ such that
\begin{equation*}
\left\Vert a\right\Vert_{\ell^1_{\eta}} \bydef \sum_{n\in\Z^4_*}\vert a_n\vert \eta^{ \left\vert n\right\vert_1}<\infty,
\end{equation*}
with $\left\vert n\right\vert_1=\sum_{j=1}^4 \vert n_j\vert$. We also introduce the subspaces $\ell^1_{\eta,-2,-1}(\C)$, $\ell^1_{\eta,-1,-1}(\C)$ and $\ell^1_{\eta,-1,0}(\C)$ of $\C^{\Z^4_*}$, associated to the norms
\begin{equation*}
\left\Vert a\right\Vert_{\ell^1_{\eta,-2,-1}} \bydef \sum_{n\in\Z^4_*}\vert a_n\vert \frac{\eta^{ \left\vert n\right\vert_1}}{\max(\vert\tn\vert_\infty^2,\vert n_4\vert)} , \qquad
\left\Vert a\right\Vert_{\ell^1_{\eta,-1,-1}} \bydef \sum_{n\in\Z^4_*}\vert a_n\vert \frac{\eta^{ \left\vert n\right\vert_1}}{\vert n\vert_\infty}
\end{equation*}
and
\begin{equation*}
\left\Vert a\right\Vert_{\ell^1_{\eta,-1,0}} \bydef \sum_{n\in\Z^4_*}\vert a_n\vert \frac{\eta^{ \left\vert n\right\vert_1}}{\vert \tn\vert_\infty},
\end{equation*}
with $\vert n\vert_\infty=\max_{1\leq j\leq 4} \vert n_j\vert$ and $\vert \tn\vert_\infty=\max_{1\leq j\leq 3} \vert n_j\vert$.

The main space in which we are going to work is the Banach space $\XX=\C\times \left(\ell^1_{\eta}(\C)\right)^3$ with the norm
\begin{equation*}
\left\Vert W \right\Vert_{\XX}=\vert \Omega\vert +\sum_{1\leq l\leq 3} \Vert \omega^{(l)} \Vert_{\ell^1_{\eta}}.
\end{equation*}
\begin{remark}\label{r:unitweight}
Since $\Omega$ and $\omega$ are incommensurable, $\Omega$ and $\omega$ do not live in the same space it is prudent to introduce an extra weight in the norm for the $\vert \Omega\vert$ term. Depending on the application at hand this may indeed be necessary, but for the results in the current paper setting this weight to unity suffices.
\end{remark}
\begin{remark}
\label{rem:analyticity}
Notice that, as soon as $\Omega\in\R$ and $\eta>1$, the function $\omega$ associated to an element $W$ of $\XX$ via~\eqref{eq:function_coeffs_omega} is analytic.
\end{remark}
Similarly, we introduce the Banach spaces $\XX_{-2,-1}=\C\times \left(\ell^1_{\eta,-2,-1}(\C)\right)^3$ and $\XX_{-1,-1}=\C\times \left(\ell^1_{\eta,-1,-1}(\C)\right)^3$, respectively endowed with the norms
\begin{equation*}
\left\Vert W \right\Vert_{\XX_{-2,-1}}=\vert \Omega\vert +\sum_{1\leq l\leq 3} \Vert \omega^{(l)} \Vert_{\ell^1_{\eta,-2,-1}}
\end{equation*}
and
\begin{equation*}
\left\Vert W \right\Vert_{\XX_{-1,-1}}=\vert \Omega\vert +\sum_{1\leq l\leq 3} \Vert \omega^{(l)} \Vert_{\ell^1_{\eta,-1,-1}}.
\end{equation*}
Notice that $\F$ defined in \eqref{eq:fullFnosymmetry} maps $\XX$ into $\XX_{-2,-1}$. We are also going to consider subspaces of divergence free sequences of $\XX$ and $\XX_{-2,-1}$, namely:
\begin{equation*}
\XX^{\div} \bydef \left\{W\in \XX : \sum_{m=1}^3 D_m\omega^{(m)}=0\right\},
\quad \XX_{-2,-1}^{\div} \bydef \left\{W\in \XX_{-2,-1} : \sum_{m=1}^3 D_m\omega^{(m)}=0\right\}.
\end{equation*}
Notice that $\omega \mapsto \sum_{m=1}^3 D_m\omega^{(m)}$ is a bounded linear map from $\left(\ell^1_\eta\right)^3$ to $\ell^1_{\eta,-1,0}$, therefore $\XX^{\div}$ is a closed subspace of $\XX$ and thus still a Banach space, with the norm $\left\Vert\cdot\right\Vert_{\XX}$. Similarly, $\XX_{-2,-1}^{\div}$ is  a Banach space, with the norm $\left\Vert\cdot\right\Vert_{\XX_{-2,-1}}$.

We recall that, in Lemma~\ref{lem:eq_NS_vorticity}, we need the zero of $F$ to be divergence free in order to prove that it corresponds to a solution of Navier-Stokes equation~\eqref{eq:NS}. Therefore, when we later prove the existence of a zero of $\F$ in $\XX$, it is crucial to be able to show that this zero is actually in~$\XX^{\div}$. To do so in Theorem~\ref{th:radii_pol}, we will make use of the following observation.
\begin{lemma}
\label{lem:X_div}
$\F$ maps $\XX^{\div}$ to $\XX_{-2,-1}^{\div}$.
\end{lemma}
\begin{proof}
We need to show that, for all $\omega$ satisfying $\nabla\cdot\omega=0$, we have $\nabla\cdot F(W)=0$. The only term in $F(W)$ that is not obviously divergence free is the nonlinear term, but since $\nabla\cdot \omega=0$ by assumption and $\nabla\cdot M\omega=0$ by construction of $M$, we can proceed as in~\eqref{eq:curl_of_nonlin}-\eqref{eq:curl_simplified} to rewrite the nonlinear term as a curl, therefore concluding that it is indeed divergence free.
\end{proof}

Similarly, we want to obtain zero of $\F$ that corresponds to a real-valued function. To ensure this, the following notation and observation are needed.

\begin{definition}
\label{def:complex_conj}
We define the complex conjugation symmetry $\gamma_*$, acting on $\C\times\left(\C^3\right)^{\Z^4_*}$, by
\begin{equation*}
(\gamma_*Q)_n = \begin{cases}
Q_{\phase}^*  & \text{if } n =\phase \\   
    Q_{-n}^*  & \text{if } n \in \Z^4_*,
\end{cases}
\end{equation*}
where ${}^*$ denotes complex conjugation, which is to be understood component-wise when applied to $Q_n\in\C^3$. We still use the symbol $\gamma_*$ to denote the complex conjugation symmetry acting on~$\left(\C^3\right)^{\Z^4_*}$.
\end{definition}

Notice that, thanks to the factor $i$ in the definition~\eqref{eq:phase_condition} of $F_\phase$, $\F$ is $\gamma_*$-equivariant. More precisely, one has the following statement, which will be used for Theorem~\ref{th:radii_pol}.
\begin{lemma}
\label{lem:FequivariantS}
Assume that $\hat \omega$ used in the phase condition~\eqref{eq:phase_condition} is such that $\gamma_*\hat{\omega}=\hat{\omega}$. Then
\begin{equation}\label{e:Fisgammastarequivariant}
	\F(\gamma_* W) = \gamma_* \F(W) \qquad\text{for all } W\in\XX .
\end{equation}
\end{lemma}

We end this section with a last set of notations that will be useful to describe linear operators.
\begin{notations}
\label{not:bloc}
To work with a linear operator $B:\XX\to\XX$, it is convenient to introduce the \emph{block notation}
\renewcommand\arraystretch{1.3}
\begin{equation*}
B=\left(\begin{array}{c|c|c|c}
B^{(\phase,\phase)} & B^{(\phase,1)} & B^{(\phase,2)} & B^{(\phase,3)}\\
\hline 
B^{(1,\phase)} & B^{(1,1)} & B^{(1,2)} & B^{(1,3)}\\
\hline 
B^{(2,\phase)} & B^{(2,1)} & B^{(2,2)} & B^{(2,3)}\\
\hline 
B^{(3,\phase)} & B^{(3,1)} & B^{(3,2)} & B^{(3,3)}
\end{array}\right),
\end{equation*}
\renewcommand\arraystretch{1}
where
\begin{itemize}
\item $B^{(l,m)}$ is a linear operator from $\ell^1_\eta$ to $\ell^1_\eta$, for all $1\leq l,m\leq 3$,
\item $B^{(\phase,m)}$ is a linear operator from $\ell^1_\eta$ to $\C$, for all $1\leq m\leq 3$,
\item $B^{(l,\phase)}$ is a linear operator from $\C$ to $\ell^1_\eta$, for all $1\leq l\leq 3$,
\item $B^{(\phase,\phase)}$ is a linear operator from $\C$ to $\C$.
\end{itemize}
For all $1\leq l,m\leq 3$, we write $B^{(l,m)}=\left(B^{(l,m)}_{k,n}\right)_{k,n\in\Z^4_*}$, so that, for all $a\in\ell^1_\eta$ and all $k\in\Z^4_*$,
\begin{equation*}
\left(B^{(l,m)}a\right)_k = \sum_{n\in\Z^4_*}B^{(l,m)}_{k,n}a_n.
\end{equation*}
Similarly, $B^{(\phase,m)}=\left(B^{(\phase,m)}_{n}\right)_{n\in\Z^4_*}$ and $B^{(l,\phase)}=\left(B^{(l,\phase)}_{k}\right)_{k\in\Z^4_*}$, so that, for all $a\in\ell^1_\eta$, all $\Omega\in\C$ and all $k\in\Z^4_*$,
\begin{equation*}
B^{(\phase,m)}a = \sum_{n\in\Z^4_*}B^{(\phase,m)}_{n}a_n \quad\text{and}\quad \left(B^{(l,\phase)}\Omega\right)_k = B^{(l,\phase)}_{k}\Omega.
\end{equation*}
We also use 
\begin{equation*}
B^{(.,\phase)}=\begin{pmatrix} B^{(\phase,\phase)} \\ B^{(1,\phase)} \\ B^{(2,\phase)} \\ B^{(3,\phase)} \end{pmatrix} \quad\text{and}\quad B^{(.,m)}_{.,n}=\begin{pmatrix} B^{(\phase,m)}_n \\ \left(B^{(1,m)}_{k,n}\right)_{k\in\Z^4_*} \\ \left(B^{(2,m)}_{k,n}\right)_{k\in\Z^4_*} \\ \left(B^{(3,m)}_{k,n}\right)_{k\in\Z^4_*} \end{pmatrix}
\end{equation*}
to denote all the columns of $B$. Similar notations will be used when $B$ is a linear operator from $\XX$ to $\XX_{-2,-1}$, or vice versa.
\end{notations}

\subsection{The operators \boldmath$\Adag$\unboldmath~and \boldmath$A$\unboldmath}
\label{s:AAdag}

We now assume that we have an approximate zero $\bW=(\bO,\bw)\in\XX$ of $\F$. In practice $\bW$ will only a have finite number of non zero coefficients (see~\eqref{e:defSsol}), but this is not crucial for the moment.  

As mentioned in Section~\ref{s:introduction}, our strategy to obtain the existence of a zero of $\F$ in a neighborhood of $\bW$ is to show that a Newton-like operator of the form
\begin{equation*}
T:W\mapsto W-D\F(\bW)^{-1}\F(W)
\end{equation*}
is a contraction in a neighborhood of $\bW$. To prove this directly, we would need to derive a (computable) estimate of $\left\Vert D\F(\bW)^{-1} \right\Vert_{B(\XX_{-2,-1},\XX)}$ (see Remark~\ref{rem:th_radii_pol}), which would be a quite formidable task. We circumvent this difficulty by working with an approximate inverse $A$ of~$D\F(\bW)^{-1}$, which is defined via a finite dimensional numerical approximation, and for which estimates are much easier to obtain. More precisely, we are going to consider a finite dimensional projection of $D\F(\bW)$, invert it numerically and then use the numerical inverse to construct $A$. To define $A$ precisely, it will be convenient to use the following notation.
\begin{definition}
\label{def:ball_coef}
Let $\nu$ be the kinematic viscosity used in~\eqref{eq:NS} and $\bO$ be the time-frequency of the numerical solution $\bW$. For all $n\in\Z^4$ and $n\in\N$, we define 
\begin{equation*}
\mu(n) \bydef \left\vert \nu\tn^2+i\bO n_4 \right\vert, 
\end{equation*}
and the set:
\begin{equation*}
\E(N)\bydef \left\lbrace n\in\Z^4_* : \mu(n) \leq N\right\rbrace.
\end{equation*}
\end{definition}
We then fix $N^\dag\in\N\setminus\{0\}$ (to be chosen later) and consider the set $\E^\dag\bydef\E(N^\dag)$ corresponding to the subset of indices that are used in
the definitions of $\Adag$ and $A$ below. 
The reasoning for choosing such a set $\E^\dag$ will become apparent later, when we have to control the dominant
linear part of $\F$ (see for instance Section~\ref{s:Z1bound}). This choice is also
quite natural from a spectral viewpoint. Indeed $\E^\dag$ is the set of all indices corresponding to 
eigenvalues of the heat operator $\partial_t - \nu\Delta$ with modulus less than
$N^\dag$ (see Lemma~\ref{lem:C}).

\begin{definition}
We define the subspace $\XX^\dag$ of $\XX$ as 
\begin{equation*}
\XX^\dag =\left\{ W\in \XX :  \omega_n=0,\ \forall~n\notin\E^\dag \right\}.
\end{equation*}
For $a\in\C^{\Z^4_*}$, we define
\begin{equation*}
\Pi^\dag a = \left(a_n\right)_{n\in \E^\dag}.
\end{equation*}
For $W=(\Omega,\omega)\in\XX$, this notation naturally extends to
\begin{equation*}
\Pi^\dag\omega = \begin{pmatrix} \Pi^\dag\omega^{(1)} \\ \Pi^\dag\omega^{(2)} \\ \Pi^\dag\omega^{(3)} \end{pmatrix} \quad\text{and}\quad \Pi^\dag W=\begin{pmatrix} \Omega \\ \Pi^\dag\omega \end{pmatrix}.
\end{equation*}
In the sequel, we identify the finite dimensional vector $\Pi^\dag W$ with its natural injection into $\XX^\dag$, and therefore interpret $\Pi^\dag$ as the canonical projection from $\XX$ to $\XX^\dag$. These notations also naturally extend to $\XX_{-1,-2}$.
\end{definition}
We are now ready to introduce an approximation $\Adag$ of $D\F(\bW)$ and then an approximation $A$ of $D\F(\bW)^{-1}$.
The bounded linear operator $\Adag:\XX\to \XX_{-2,-1}$ is defined by
\begin{equation*}
\left\{
\begin{aligned}
&\Adag \Pi^\dag W = \Pi^\dag D\F(\bW)|_{\XX^\dag} \Pi^\dag W , \\
&\left(\Adag\left(I-\Pi^\dag\right)W\right)_n = \left(\nu\tn^2+i\bar \Omega n_4\right)\omega_n,\qquad \text{for 
} n\notin\E^\dag.
\end{aligned}
\right.
\end{equation*}
Notice that $\Adag$ leaves both subspaces $\XX^\dag$ and $(I-\Pi^\dag)\XX$ invariant, and that it acts diagonally on $(I-\Pi^\dag)\XX$.
 
Next, we introduce $A^{(N^\dag)}$, an approximate inverse of $\Pi^\dag D\F(\bW)|_{\XX^\dag}$ that is computed numerically
(interpreting it as a finite matrix), and the bounded linear operator $A:\XX_{-2,-1}\to\XX$ defined by
\begin{equation*}
\left\{
\begin{aligned}
&A \Pi^\dag W = A^{(N^\dag)} \Pi^\dag W \\
&\left(A\left(I-\Pi^\dag\right)W\right)_n = \lambda_n\omega_n,\qquad  \text{for }  n\notin\E^\dag,
\end{aligned}
\right.
\end{equation*}
where 
\begin{equation}\label{e:deflambda}
 \lambda_n \bydef \frac{1}{\nu\tn^2+i\bar\Omega n_4}.
\end{equation}
Notice that $A$ also leaves both subspaces $\XX^\dag$ and $(I-\Pi^\dag)\XX$ invariant, and that it acts diagonally on $(I-\Pi^\dag)\XX$.

\subsection{A posteriori validation framework}
\label{s:radpol}

We are now ready to give sufficient conditions for the a posteriori validation of the solution $\bW$, that is conditions under which the existence of a zero of $\F$ in a neighborhood of $\bW$ is guaranteed. This is the content of the following theorem.

\begin{theorem}
\label{th:radii_pol}
Let $\eta>1$. With the notations of the previous sections, assume there exist $\bW\in\XX$ and non-negative constants $Y_0$, $Z_0$, $Z_1$ and $Z_2$ such that
\begin{align}
\left\Vert A\F(\bW) \right\Vert_{\XX} &\leq Y_0 \label{def:Y_0}\\
\left\Vert I-A\Adag \right\Vert_{B(\XX,\XX)} &\leq Z_0 \label{def:Z_0}\\
\left\Vert A\left(D\F(\bW)-\Adag\right) \right\Vert_{B(\XX,\XX)} &\leq Z_1 \label{def:Z_1}\\
\left\Vert A(D\F(W)-D\F(\bW)) \right\Vert_{B(\XX,\XX)} &\leq Z_2 \left\Vert W-\bW \right\Vert_{\XX}, \quad \text{for all } W\in\XX. \label{def:Z_2}
\end{align}
Assume also that
\begin{itemize}
\item the forcing term $f$ is time independent and has space average zero;
\item $\bW$ is in $\XX^{\div}$;
\item $\hat\omega$ (used to define the phase condition~\eqref{eq:phase_condition}) and $\bW$ are such that $\gamma_*\hat\omega=\hat\omega$ and $\gamma_*\bW=\bW$.
\end{itemize}
If 
\begin{equation}
\label{hyp:cond_pol}
Z_0+Z_1<1 \qquad \text{and}\qquad 2Y_0Z_2<\left(1-(Z_0+Z_1)\right)^2,
\end{equation}
then, for all $r\in[r_{\min},r_{\max})$ there exists a unique $\tilde W =(\tilde \Omega, \tilde \omega) \in \B_{\XX}(\bW,r)$ such that $\F(\tilde W)=0$, where $\B_{\XX}(\bW,r)$ is the closed ball of ${\XX}$, centered at $\bW$ and of radius $r$, and
\begin{equation*}
r_{\min} \bydef \frac{1-(Z_0+Z_1)-\sqrt{\left(1-(Z_0+Z_1)\right)^2-2Y_0Z_2}}{Z_2},\qquad r_{\max} \bydef \frac{1-(Z_0+Z_1)}{Z_2}.
\end{equation*}
Besides, this unique $\tilde W$ also lies in $\XX^{\div}$. Finally, defining $u=M\tilde\omega$, there exists a pressure function~$p$ such that $(u,p)$ is a $\frac{2\pi}{\tilde \Omega}$-periodic, real valued and analytic solution of Navier-Stokes equations~\eqref{eq:NS}.
\end{theorem}
\begin{proof}
First notice that we have 
\begin{equation*}
\left\Vert I-A\Adag \right\Vert_{B(\XX,\XX)} \leq Z_0 <1,
\end{equation*}
hence $A\Adag$ is a bounded linear and invertible operator from $\XX$ to itself by a standard Neumann series argument. Besides, since $A$ and $\Adag$ have diagonal tails that are the exact inverse of one another, the above bound gives that the finite part of $A$ is invertible and thus $A^{-1}$ is a bounded linear operator from $\XX_{-2,-1}$ to $\XX$. We also have that
\begin{equation*}
\left\Vert I-AD\F(\bW) \right\Vert_{B(\XX,\XX)} \leq Z_0+Z_1 <1,
\end{equation*}
and therefore $Q \bydef AD\F(\bW)$ is a bounded linear and invertible operator from $\XX$ to itself. Furthermore, $\left\Vert Q^{-1}\right\Vert_{B(\XX,\XX)} \leq (1-(Z_0 + Z_1))^{-1}$. Hence, we finally have that
\begin{equation*}
D\F(\bW)=A^{-1}Q
\end{equation*}
is a bounded linear and invertible operator from $\XX_{-2,-1}$ to $\XX$.

We can thus consider
\begin{equation} \label{def:T}
T(W) \bydef W - D\F(\bW)^{-1}\F(W)
\end{equation}
which maps $\XX$ to itself. We are now going to show that $T$ is a contraction on $\B_{\XX}(\bW,r)$ for all $r\in[r_{\min},r_{\max})$. It is going to be helpful to introduce the polynomial
\begin{equation*}
P(r)\bydef \frac{1}{1-(Z_0+Z_1)}\frac{1}{2}Z_2 r^2 -r +\frac{1}{1-(Z_0+Z_1)}Y_0.
\end{equation*}
Notice that by~\eqref{hyp:cond_pol}, the quadratic polynomial $P$ has two positive roots, $r_{\min}$ being the smallest, and that $P(r_{\max})<0$ since $r_{\max}$ is the apex of $P$. We estimate, for $r>0$ and $W\in\B_{\XX}(\bW,r)$,
\begin{align*}
\left\Vert T(W)-\bW\right\Vert_{\XX} &\leq \left\Vert T(W)-T(\bW)\right\Vert_{\XX} + \left\Vert T(\bW)-\bW\right\Vert_{\XX} \\
&\leq \int_0^1\left\Vert DT(\bW+t(W-\bW))\right\Vert_{B(\XX,\XX)}dt \left\Vert W-\bW\right\Vert_{\XX} 
 + \left\Vert D\F(\bW)^{-1}\F(\bW)\right\Vert_{\XX} \\
&\leq \left\Vert Q^{-1}\right\Vert_{B(\XX,\XX)} 
\big( \int_0^1\left\Vert A\left(D\F(\bW+t(W-\bW))-D\F(\bW)\right)\right\Vert_{B(\XX,\XX)}dt \left\Vert W-\bW\right\Vert_{\XX} \\
& \qquad \qquad \qquad \qquad  + \left\Vert A\F(\bW)\right\Vert_{\XX} \big) \\
&\leq \left\Vert Q^{-1}\right\Vert_{B(\XX,\XX)} \left(Z_2 r^2 \int_0^1t dt + Y_0 \right) \\
&\leq \frac{1}{1-(Z_0+Z_1)}\left(\frac{1}{2}Z_2 r^2 + Y_0\right).
\end{align*}
For all $r\in[r_{\min},r_{\max}]$, we have $P(r)\leq 0$ and hence $\left\Vert T(W)-\bW\right\Vert_{\XX} \leq r$, that is $T$ maps $\B_{\XX}(\bW,r)$ into itself.

Furthermore, for $r>0$ and $W\in\B_{\XX}(\bW,r)$ we also have
\begin{align*}
\left\Vert DT(W)\right\Vert_{B(\XX,\XX)} &= \left\Vert I-D\F(\bW)^{-1}D\F(W)\right\Vert_{B(\XX,\XX)} \nonumber \\
&\leq \left\Vert Q^{-1}\right\Vert_{B(\XX,\XX)} \left\Vert A\bigl(D\F(\bW)-D\F(W)\bigr)\right\Vert_{B(\XX,\XX)} \nonumber \\
&\leq \frac{Z_2r}{1-(Z_0+Z_1)}. \label{e:2nd-rhs}
\end{align*}
Hence, from the definition of $r_{\max}$ it follows that  $\left\Vert DT(W)\right\Vert_{B(\XX,\XX)} < 1$ for any $r\in [0,r_{\max})$.

Thus we infer that $T$ is a contraction on $\B_{\XX}(\bW,r)$ for any $r\in[r_{\min},r_{\max})$. Banach's fixed point Theorem then yields the existence of a unique fixed point $\tilde W = (\tilde \Omega, \tilde \omega)$ of $T$ in $\B_{\XX}(\bW,r)$, for all $r\in[r_{\min},r_{\max})$, which corresponds to a unique zero of $\F$ since $D\F(\bW)^{-1}$ is injective.

In order to prove that $\tilde{\omega}$ is divergence free, we then make use of the fact that $\F(\XX^{\div})\subset \XX_{-2,-1}^{\div}$ (Lemma~\ref{lem:X_div}). Since $\XX_{-2,-1}^{\div}$ is a closed subspace of $\XX_{-2,-1}$, this implies that for all $W\in \XX^{\div}$, $D\F(W)\left(\XX^{\div}\right)\subset\XX_{-2,-1}^{\div}$. In particular, since $D\F(\bW)$ is invertible and $\bW\in \XX^{\div}$, we have that $D\F(\bW)^{-1}\left(\XX_{-2,-1}^{\div}\right)\subset\XX^{\div}$, and hence that $T(\XX^{\div})\subset \XX^{\div}$. Therefore, $T$ is also a contraction on $\B_{\XX^{\div}}(\bW,r)\bydef\B_{\XX}(\bW,r)\cap \XX^{\div}$ and we now get the existence of a unique fixed point $\tilde W$ of $T$ in $\B_{\XX^{\div}}(\bW,r)$ (which must be the same as the one obtained previously by uniqueness).

Finally, concerning real-valuedness, by Lemma~\ref{lem:FequivariantS} we have $\F(\gamma_*\tilde W)=\gamma_*\F(\tilde W)=0$, and since $\B_{\XX}(\bW,r)$ is invariant under $\gamma_*$, this ``new'' zero $\gamma_*\tilde W$ of $\F$ also belongs to $\B_{\XX}$. By uniqueness we must have $\gamma_*\tilde W=\tilde W$, which means that that $\Omega\in\R$ and that the functions associated to $\tilde\omega$ and $u$ are real-valued.

The function associated to $\tilde\omega$ is therefore divergence free, analytic since $\eta>1$, and $\Omega\in\R$. By Lemma~\ref{lem:eq_NS_vorticity}, $u$ then solves Navier-Stokes equations.
\end{proof}

\begin{remark}
\label{rem:th_radii_pol}
Many similar versions of this theorem have been used in the last decades, in a
posteriori error analysis and computer-assisted proofs (see for
instance~\cite{CalRap97,Plu01,Yamamoto1998,AriKocTer05,DayLesMis07}). One possible approach (which
is maybe the most natural one) to show that the operator $T$ defined
in~\eqref{def:T} is a contraction from a small ball around $\bW$ into itself,
is to directly estimate
\begin{equation*}
\left\Vert \F(\bW) \right\Vert_{\XX_{-2,-1}}, ~~ \left\Vert D\F(\bW)^{-1} \right\Vert_{B(\XX_{-2,-1},\XX)} ~~ \text{and}
~~ \sup_{W\in\B(\bW,r)}\left\Vert D\F(W)-D\F(\bW) \right\Vert_{B(\XX,\XX_{-2,-1})}
\end{equation*}
instead of~\eqref{def:Y_0}-\eqref{def:Z_2}. The main difficulty of this
approach in practice is to obtain a bound for the inverse $D\F(\bW)^{-1}$,
which can sometimes be done using eigenvalue enclosing techniques
(see~\cite{Plu01} and the references therein). Another possibility is to
replace $D\F(\bW)^{-1}$ by an approximate inverse $A$ of $D\F(\bW)$, and to
study the fixed point operator
\begin{equation*}
T_{\approximate} \bydef I-A\F
\end{equation*}
instead of $T$. Estimating the quantities in~\eqref{def:Y_0}-\eqref{def:Z_2} is
then a good way to prove that $T_{\approximate}$ is a contraction from a small
ball around $\bW$ into itself (see for instance~\cite{Yamamoto1998,DayLesMis07}). The
price to pay for this approach is that one has to actually compute (partially
numerically) a good enough approximate inverse $A$ of $D\F(\bW)$, but
estimating the norm of $A$ then becomes very straightforward, compared to
having to work with the exact inverse $D\F(\bW)^{-1}$.

What Theorem~\ref{th:radii_pol} shows is that, when an approximate inverse $A$
is used and when the estimates~\eqref{def:Y_0}-\eqref{def:Z_2} are good enough
to show that $T_{\approximate}$ is a contraction from $\B(\bW,r)$ into itself
for some $r>0$, then the same is actually true for $T$. This observation may
seem inconsequential, as what we really care about is the existence of a zero
of $\F$ near $\bW$, rather than which fixed point operator was used to prove this
fact (since both give the same error bound with this approach). However, it
turns out to be very advantageous in this work. Indeed, to show that a zero $W$
of $\F$ corresponds to a solution of Navier-Stokes equations, we need to know a
priori that $\omega$ is divergence free, see Lemma~\ref{lem:eq_NS_vorticity}.
Since $\F$ (and thus $D\F(\bW)^{-1}$) preserves the divergence free property,
so does $T$, and hence we were able to obtain \emph{for free} that our fixed
point is divergence free (by making sure that the numerical approximate solution $\bW$ itself is divergence free). To obtain the same result with $T_{\approximate}$, we
would have to make sure that the approximate inverse $A$ also preserves
\emph{exactly} the divergence free property. For the approximate inverse $A$
described in Section~\ref{s:AAdag} this could likely be achieved, albeit at a considerable computational cost. 
However, as mentioned in the introduction, in practice we work with symmetry
reduced variables (see Section~\ref{s:symmetries}). Therefore we use a
symmetrically reduced version of $A$ instead of $A$ itself. Making sure that
this \emph{symmetrically reduced} version of $A$ preserves the divergence free
property without directly having access to $A$ is a formidable task, which we
are able to avoid by working with $T$ instead of $T_{\approximate}$.
\end{remark}

\begin{remark}
\label{rem:comparison_of_norms}
We note that the weighted $\ell^1$-norm controls the $\CC^0$-norm. Therefore, when the assumptions of Theorem~\ref{th:radii_pol} are satisfied, we have the following explicit error control between the \emph{exact} vorticity $\tilde \omega$ and the \emph{approximate} vorticity $\bw$:
\begin{equation*}
\sup\limits_{\substack{t\in\R \\ x\in\T^3}} \sum_{m=1}^3\left\vert \tilde\omega^{(m)}(x,t)-\bar{\omega}^{(m)}(x,\theta t) \right\vert \leq \left\Vert \tilde{W} - \bW \right\Vert_\XX \leq r_{\min},
\end{equation*}
where $\theta=\tilde\Omega / \bar\Omega$ accounts for the time dilation that occurs when the approximate period is not exactly equal to the true period. We also point out that, for $\eta>1$, we could also obtain explicit error estimates on derivatives of the vorticity. In Lemma~\ref{lem:error_estimates} corresponding error bounds on the velocity field and the pressure are presented.
\end{remark}

While Theorem~\ref{th:radii_pol} is the cornerstone of our approach, the main difficulty still lies ahead of us: we need to derive and implement explicit bounds satisfying~\eqref{def:Y_0}-\eqref{def:Z_2}, that are sharp enough for~\eqref{hyp:cond_pol} to hold. These bounds are obtained in Section~\ref{s:estimates}. However, the computations required to evaluate them are quite prohibitive due to the high dimension of the problem. Therefore, in Section~\ref{s:symmetries} we make use of the symmetries of the solution to reduce the amount of computation needed, and update the bounds obtained in Section~\ref{s:estimates} accordingly, by showing that they are in a certain precise sense \emph{compatible} with the symmetries. The implementation in a MATLAB code of the bounds in the symmetric setting can be found at~\cite{navierstokescode}, and the results are discussed in Section~\ref{s:application}.

\section{Estimates without using symmetries}
\label{s:estimates}

In this section, we derive bounds $Y_0$, $Z_0$, $Z_1$ and $Z_2$ satisfying~\eqref{def:Y_0}-\eqref{def:Z_2}. We first list a few auxiliary lemmas that are going to be used several times, and then devote one subsection to each bound. We assume throughout that the approximate solution $\bW$ only has a finite number of non-zero modes. More precisely, writing $\bW=(\bO,\bw) \in \C\times\C^{\Z^4_*}$, we assume there exists a finite set of indices $\set^{\sol}\subset \Z^4_*$ such that 
\begin{equation}\label{e:defSsol}
  \bw_n=0 \qquad\text{for all }n\notin\set^{\sol}.
\end{equation}
We also assume that $\hat\omega$ used in the phase condition~\eqref{eq:phase_condition} is chosen so that $\hat{\omega}_n=0$ for all $n\notin \set^\sol$, that $\N^\dag\in N\setminus\{0\}$ is fixed, and recall that $\E^\dag$ is defined in Section~\ref{s:AAdag}.

\subsection{Uniform estimates involving \boldmath$\lambda_n$\unboldmath}
\label{sec:uniform_estimates}

We regroup here some straightforward lemmas that are going to facilitate estimates involving the tail part of $A$, in Sections~\ref{s:Z1bound} and~\ref{s:Z2bound}. Recall that $\E(N)$ was introduced in Definition~\ref{def:ball_coef} and $\lambda_n$ in~\eqref{e:deflambda}.

\begin{lemma}
\label{lem:C}
For all $N\in\N\setminus\{0\}$
\begin{equation*}
\sup\limits_{n\notin\E(N)} \left\vert\lambda_n \right\vert \leq \frac{1}{N}.
\end{equation*}
\end{lemma}
\begin{proof}
Simply notice that $\mu(n)= \frac{1}{\left\vert\lambda_n\right\vert}$.
\end{proof}
\begin{lemma}
\label{lem:C123}
For all $N\in\N\setminus\{0\}$ and all $m\in\{1,2,3\}$,
\begin{equation*}
\sup\limits_{n\notin\E(N)} \left\vert\lambda_n \right\vert\vert n_m\vert  \leq \frac{1}{\sqrt{\nu N}}.
\end{equation*}
\end{lemma}
\begin{proof}
We note that
\begin{equation*}
\vert n_m\vert \leq \sqrt{\frac{\mu(n)}{\nu}},
\end{equation*}
hence
\begin{equation*}
\left\vert\lambda_n \right\vert \vert n_m\vert \leq \frac{1}{\sqrt{\nu \mu(n)}}. \qedhere
\end{equation*}
\end{proof}

\begin{lemma}
\label{lem:C123_bis}
For all $N\in\N\setminus\{0\}$
\begin{equation*}
\sup\limits_{n\notin\E(N)} \left\vert\lambda_n \right\vert\sum_{m=1}^3 \vert n_m\vert   \leq \frac{\sqrt 3}{\sqrt{\nu N}}.
\end{equation*}
Besides, for all $p\in\{1,2,3\}$, we also have
\begin{equation*}
\sup\limits_{n\notin\E(N)} \left\vert\lambda_n \right\vert\sum_{\substack{1\leq m\leq 3 \\ m\neq p}}\vert n_m\vert \leq \frac{\sqrt 2}{\sqrt{\nu N}}.
\end{equation*}
\end{lemma}
\begin{proof}
Just notice that 
\begin{align*}
\sum_{m=1}^3 \vert n_m\vert &\leq \sqrt{3\sum_{m=1}^3  n_m^2} 
\leq \sqrt{\frac{3\mu(n)}{\nu}},
\end{align*}
where the first inequality follows from the Cauchy-Schwarz inequality. Therefore
\begin{equation*}
\left\vert\lambda_n \right\vert \sum_{m=1}^3 \vert n_m\vert \leq \frac{\sqrt 3}{\sqrt{\nu \mu(n)}}.
\end{equation*}
The second estimate is obtained similarly, using that
\begin{equation*}
\sum_{\substack{1\leq m\leq 3 \\ m\neq p}}\vert n_m\vert \leq \sqrt{2\sum_{\substack{1\leq m\leq 3 \\ m\neq p}} n_m^2} \leq \sqrt{2\sum_{m=1}^3  n_m^2}. \qedhere
\end{equation*}
\end{proof}

\begin{lemma}
\label{lem:C_inf}
For all $N\in\N\setminus\{0\}$
\begin{equation*}
\sup\limits_{n\notin\E(N)} \left\vert\lambda_n \right\vert \vert n\vert_\infty \leq \max\left(\frac{1}{\sqrt{\nu N}},\frac{1}{\bO}\right).
\end{equation*}
\end{lemma}
\begin{proof}
Let $n\notin\E(N)$. If $\vert n\vert_\infty=\vert n_4\vert$, then clearly
\begin{equation*}
\left\vert\lambda_n \right\vert \vert n\vert_\infty \leq \frac{1}{\bO}.
\end{equation*}
Otherwise, $\vert n\vert_\infty=\vert n_m\vert$ for $m\in\{1,2,3\}$ and we conclude by Lemma~\ref{lem:C123}.
\end{proof}

\subsection{\boldmath$Y_0$\unboldmath~bound for \boldmath$A\F(\bar W)$\unboldmath}
\label{s:Y0bound}

We start by giving a computable bound for $\left\Vert  A\F(\bar W)\right\Vert_{\XX}$.
 
\begin{proposition}
\label{prop:Y0}
Assume that $\rhs_n=0$ for $n \notin \set^\sol+\set^\sol \bydef \{n_1+n_2\ |\ n_1,n_2\in \set^\sol\}$, and that $A$ is defined as in Section~\ref{s:AAdag}. Then~\eqref{def:Y_0} is satisfied with
\begin{align*}
Y_0 &\bydef \left\Vert  A^{(N^\dag)} \Pi^\dag\F(\bW)\right\Vert_{\XX} 
+ \sum_{n\in (\set^\sol+\set^\sol)\setminus \E^\dag} \sum_{m=1}^3 \vert\lambda_n\vert \left\vert F^{(m)}_n(\bW) \right\vert \eta^{\vert n\vert_1} \\
&= \left\vert \left( A^{(N^\dag)} \Pi^\dag \mathcal{F}(\bW) \right)_{\phase} \right\vert   + \sum_{n\in \E^\dag} \sum_{m=1}^3 \left\vert \left(  A^{(N^\dag)} \Pi^\dag \mathcal{F}(\bW) \right)^{(m)}_n \right\vert \eta^{\vert n\vert_1} \\
& \quad
+ \sum_{n\in (\set^\sol+\set^\sol)\setminus \E^\dag} \sum_{m=1}^3 \vert\lambda_n\vert \left\vert F^{(m)}_n(\bW) \right\vert \eta^{\vert n\vert_1}.
\end{align*}
\end{proposition}
\begin{proof}
The only thing to notice is that the approximate solution $\bW$ only has a finite number of non-zero modes, hence the same is true for $\F(\bar W)$. More precisely, since $F$ is quadratic we infer that $F_n(\bW)=0$ for all $n\notin \set^\sol+\set^\sol$. Recalling the definition of $A$ and using the splitting
\begin{align*}
\left\Vert  A\F(\bar W)\right\Vert_{\XX} &= \left\Vert  \Pi^\dag A\F(\bar W)\right\Vert_{\XX} + \left\Vert  (I -\Pi^\dag)A \F(\bar W)\right\Vert_{\XX} \\
&= \left\Vert A \Pi^\dag\F(\bar W)\right\Vert_{\XX} + \left\Vert A (I -\Pi^\dag)\F(\bar W)\right\Vert_{\XX},
\end{align*}
then directly yields $Y_0$.
\end{proof}

\subsection{\boldmath${Z_0}$\unboldmath~bound for \boldmath${I-A\Adag}$\unboldmath}
\label{s:Z0bound}

We now give a computable bound for $\left\Vert  I-A\Adag\right\Vert_{B(\XX,\XX)}$.
\begin{proposition}
\label{prop:Z0}
Assume that $\Adag$ and $A$ are defined as in Section~\ref{s:AAdag}. Then~\eqref{def:Z_0} is satisfied with
\begin{equation*}
Z_0\bydef\left\Vert  I-A^{(N^\dag)} D\F^\dag|_{\XX^\dag}(\bW)\right\Vert_{B(\XX,\XX)}.
\end{equation*}
\end{proposition}
\begin{proof}
We recall that we defined $\Adag$ and $A$ in such a way that their \emph{tails} are \emph{exact} inverses of each other, that is
\begin{equation*}
\left(I-\Pi^\dag\right)\left(I-A\Adag\right)=0,\quad \text{and}\quad \left(I-A\Adag\right)\left(I-\Pi^\dag\right)=0.
\end{equation*}
Therefore, the only non-zero part of $I-A\Adag$ is the \emph{finite} part $\Pi^\dag \left(I-A\Adag\right) \Pi^\dag$, which yields~$Z_0$.
\end{proof}
Besides, for a linear operator $B:\XX\to\XX$, the operator norm of $B$ is nothing but the supremum of the norm of each of its column, with a weight, since we use a weighted $\ell^1$ norm on~$\XX$:
\begin{align*}
\left\Vert B\right\Vert_{B(\XX,\XX)} = \max\left[\left\Vert B^{(.,\phase)} \right\Vert_{\XX},\max\limits_{1\leq m\leq 3} \sup\limits_{n\in\Z^4_*} \frac{1}{\eta^{\vert n\vert_1}}\left\Vert B^{(.,m)}_{.,n}\right\Vert_{\XX}\right].
\end{align*}
Therefore, if $B$ only has a finite number of non-zero \emph{columns} $B^{(.,m)}_{.,n}$, and if each of these columns only has a finite number of non-zero terms, (which is the case for the operator involved in $Z_0$) we can evaluate such a norm on a computer. 

\subsection{\boldmath${Z_1}$\unboldmath~bound for \boldmath${A\bigl(D\F(\bW)-\Adag\bigr)}$\unboldmath}
\label{s:Z1bound}

In this section, we give a computable bound for $\bigl\Vert  A\bigl(D\F(\bW)-\Adag\bigr)\bigr\Vert_{B(\XX,\XX)}$. Because of the way $\Adag$ and $A$ are defined, and since we only consider here the derivative of $\F$ at a numerical solution~$\bW$ (which only has a finite number of nonzero coefficients), each column of $A\bigl(D\F(\bW)-\Adag\bigr)$ also only has a finite number of nonzero coefficients. Therefore, we can numerically evaluate the norm of any finite number of columns of $A\bigl(D\F(\bW)-\Adag\bigr)$. Our strategy to obtain the $Z_1$ bound is thus to compute the norm of a finite (but large enough) number of columns of $A\bigl(D\F(\bW)-\Adag\bigr)$, and to then get an analytic estimate for the remaining columns. 
To describe for which columns we compute the norm explicitly, we introduce the following set of indices.
\begin{definition}
\label{def:tildeSbis}
Let $\tilde N\in\N$. We define the set $\tildeset^{\sol}(\tilde N)$ by
\begin{equation}\label{e:deftildeSsol}
\tildeset^{\sol}(\tilde N) \bydef \E(\tilde N)+ \set^{\sol},
\end{equation}
\end{definition}
\begin{remark}
\label{rem:deftildeSbis}
In the sequel, we will have to estimate quantities such as
\begin{equation}
\label{e:bound_lambda}
\sup_{n\notin \set,\, k\in\set^\sol} \vert \lambda_{n+k}\vert,
\end{equation}
for a given set $\set$. By definition~\eqref{e:deftildeSsol} of $\tildeset^{\sol}(\tilde N)$, and assuming $-\set^\sol=\set^\sol$, we have
\begin{equation*}
\left(\tildeset^{\sol}(\tilde N)\right)^c + \set^\sol \subset \left(\E(\tilde N)\right)^c,
\end{equation*}
hence for $\set=\tildeset^{\sol}(\tilde N)$ we can bound~\eqref{e:bound_lambda} by
\[
  \sup_{n \notin \tildeset^{\sol}(\tilde N) ,\, k\in\set^\sol} \vert \lambda_{n+k}\vert \leq  \sup_{n \notin \E(\tilde N)} \vert \lambda_{n}\vert  ,
\]
and we can thus use the estimates of Section~\ref{sec:uniform_estimates} to control such terms.
\end{remark}
\begin{remark}\label{r:NdagNtilde}
Whereas the computational parameter $N^\dag$ determines the size of the ``computational/finite'' part $A^{(N^\dag)}$ of the linear operator $A$, the computational parameter $\tilde{N}$ can be chosen (for fixed choice of $N^\dag$) to balance the quality of the estimates  and the computational costs. We will always 
need to choose $\tilde N$ so that $\E(\tilde N)$ contains $\E^\dag=\E(N^\dag)$, i.e. take $\tilde N\geq N^\dag$, to ensure that the finite part of $A$ does not influence the tail estimate, see equation~\eqref{eq:Z1tail_Alambda} below.
\end{remark}
\begin{proposition}
\label{prop:Z1}
Assume that $\bar{W} \in \XX^{\div}$, $\E^\dag \subset \E(\tilde{N})$, and that $\Adag$ and $A$ are defined as in Section~\ref{s:AAdag}. Define $C = D\F(\bW)-\Adag$. Then~\eqref{def:Z_1} is satisfied with
\begin{equation*}
Z_1\bydef\max\left(\left\Vert AC^{(.,\phase)}\right\Vert_{\XX}, \max_{1 \leq m \leq 3}
\left(Z_1^{\textup{finite}}\right)^{(m)},
\max_{1 \leq m \leq 3}
\left(Z_1^{\textup{tail}}\right)^{(m)}\right), 
\end{equation*}
where
\begin{equation*}
\left(Z_1^{\textup{finite}}\right)^{(m)} = \max\limits_{n\in \tildeset^{\sol}(\tilde N)} \frac{1}{\eta^{\vert n\vert_1}}\left\Vert AC^{(.,m)}_{.,n}\right\Vert_{\XX},
\end{equation*}
and
\begin{align}
\left(Z_1^{\textup{tail}}\right)^{(m)} &= \frac{\sqrt 3}{\sqrt{\nu \tilde N}} \left\Vert \max\limits_{1 \leq p \leq 3}(M\bw)^{(p)} \right\Vert_{\ell^1_{\eta}} + \frac{1}{\tilde N}\left( \frac{3}{2} \sum_{p=1}^3 \left\Vert \bw^{(p)}\right\Vert_{\ell^1_{\eta}}  - \frac{1}{2} \left\Vert \bw^{(m)}\right\Vert_{\ell^1_{\eta}}  \right)\nonumber\\
&\qquad\quad + \frac{1}{\tilde N} \sum_{l=1}^3\left(\left\Vert D_m(M\bw)^{(l)}\right\Vert_{\ell^1_{\eta}} + \sum_{p=1}^3
\left\Vert D_p\bw^{(l)}\right\Vert_{\ell^1_{\eta}} - 
\left\Vert D_m\bw^{(l)}\right\Vert_{\ell^1_{\eta}} \right) ,
\label{e:Z1tail}
\end{align}
where $\sqrt 3$ can be replaced by $\sqrt 2$ for an (essentially) 2D solution.
\end{proposition}
\begin{proof}
We have
\begin{align*}
\left\Vert AC \right\Vert_{B(\XX,\XX)} &= \max\left[\left\Vert (AC)^{(.,\phase)} \right\Vert_{\XX},\max\limits_{1\leq m\leq 3} \sup\limits_{n\in\Z^4_*} \frac{1}{\eta^{\vert n\vert_1}}\left\Vert (AC)^{(.,m)}_{.,n}\right\Vert_{\XX}\right] \\
&= \max\left[\left\Vert AC^{(.,\phase)} \right\Vert_{\XX},\max\limits_{1\leq m\leq 3} \sup\limits_{n\in\Z^4_*} \frac{1}{\eta^{\vert n\vert_1}}\left\Vert AC^{(.,m)}_{.,n}\right\Vert_{\XX}\right],
\end{align*}
and we want to show that this quantity is bounded by $Z_1$. From the definition of $Z_1^{\textup{finite}}$ and~$Z_1^{\textup{tail}}$, and the splitting
\begin{equation*}
\sup\limits_{n\in\Z^4_*} \frac{1}{\eta^{\vert n\vert_1}}\left\Vert AC^{(.,m)}_{.,n}\right\Vert_{\XX} = \max \left(\max\limits_{n\in \tildeset^{\sol}(\tilde N)} \frac{1}{\eta^{\vert n\vert_1}}\left\Vert AC^{(.,m)}_{.,n}\right\Vert_{\XX}, \sup\limits_{n\notin \tildeset^{\sol}(\tilde N)} \frac{1}{\eta^{\vert n\vert_1}}\left\Vert AC^{(.,m)}_{.,n}\right\Vert_{\XX}\right),
\end{equation*}
we see that we only have to prove that
\begin{equation}
\label{eq:Z1tail_toprove}
\sup\limits_{n\notin \tildeset^{\sol}(\tilde N)} \frac{1}{\eta^{\vert n\vert_1}}\left\Vert AC^{(.,m)}_{.,n}\right\Vert_{\XX} \leq \left(Z_1^{\textup{tail}}\right)^{(m)}, \qquad \forall~1\leq m\leq 3.
\end{equation}
In order to do so, it is helpful to first have a look at the structure of $C$.

By definition of $\Adag$, we have that $\Pi^\dag C \Pi^\dag=0$, and that the remaining coefficients are given by $D\Psi(\bw)$, since the linear part of $D\F(\bW)$ is also present in $\Adag$, that is: if $k\notin\E^\dag$ or $n\notin\E^\dag$ then
\begin{equation*}
C^{(l,m)}_{k,n} = \left(D\Psi(\bw)\right)^{(l,m)}_{k,n},
\end{equation*}
with $\Psi$ as defined in~\eqref{e:defPsi}. Using the block-notation introduced in Section~\ref{sec:notations}, we write
\renewcommand\arraystretch{1.3}
\begin{equation*}
C=\left(\begin{array}{c|c|c|c}
0 & C^{(\phase,1)} & C^{(\phase,2)} & C^{(\phase,3)}\\
\hline 
C^{(1,\phase)} & C^{(1,1)} & C^{(1,2)} & C^{(1,3)}\\
\hline 
C^{(2,\phase)} & C^{(2,1)} & C^{(2,2)} & C^{(2,3)}\\
\hline 
C^{(3,\phase)} & C^{(3,1)} & C^{(3,2)} & C^{(3,3)}
\end{array}\right),
\end{equation*}
\renewcommand\arraystretch{1}
with, for all $1\leq l,m\leq 3$,
\begin{equation*}
C^{(m,\phase)}_{n}=\begin{cases}
 0 &\qquad n\notin\set^\sol \text{ or }n\in\set^{\sol}\cap\E^\dag, \\
i n_4\bw_n^{(m)} & \qquad n\in\set^\sol\setminus \E^\dag,
\end{cases}
\end{equation*}
and
\begin{equation*}
C^{(\phase,m)}_{n}=\begin{cases}
 0 &\qquad n\notin\set^\sol \text{ or }n\in\set^{\sol}\cap\E^\dag, \\
i n_4\left(\hat\omega_n^{(m)}\right)^* & \qquad n\in\set^\sol\setminus \E^\dag.
\end{cases}
\end{equation*}
Furthermore, using that for all $k\in\Z^4_*$
\begin{align}
\label{e:DPsi}
\left(D\Psi(\bw)\omega\right)_k &= i\left(
\left[\left(M\bw\star \tD\right) \omega\right]_k 
- \left[\left(\omega \star \tD\right) M\bw\right]_k 
\right. 
\\
& \qquad + \left. \left[\left(M\omega\star \tD\right) \bw\right]_k 
-\left[\left(\bw \star \tD\right) M\omega\right]_k
\right), 
\nonumber
\end{align}
we obtain
\begin{equation}\label{e:Cklmn}
C^{(l,m)}_{k,n} =
\left\{
\begin{aligned}
&0, \quad &&k,n\in\E^\dag,\\
&i\delta_{l,m}\sum_{p=1}^3 k_p(M\bw)^{(p)}_{k-n} - i \left(D_m(M\bw)^{(l)}\right)_{k-n}  \\
& \qquad + i\sum_{p=1}^3 M_n^{(p,m)}\left(D_p\bw^{(l)}\right)_{k-n} - i\sum_{p=1}^3 n_p M_n^{(l,m)} \bw^{(p)}_{k-n}, \quad &&k\notin\E^\dag \text{ or }n\notin\E^\dag,
\end{aligned}
\right.
\end{equation}
where $M^{(l,m)}_n$ is the coefficient on row $l$ and column $m$ in the matrix $M_n$, see~\eqref{e:defM}. Each of the four terms in $C^{(l,m)}_{k,n}$ comes from one of the four terms in $D\Psi(\bw)\omega$, and it should be noted that we used Lemma~\ref{lem:trick_derivative} on the first of those four terms to obtain this expression (see Remark~\ref{rem:derivative_out} below).

From the expression of $C$ we see that $\left\Vert AC^{(.,\phase)}\right\Vert_{\XX}$ and $\max_{1 \leq m \leq 3} \left(Z_1^{\textup{finite}}\right)^{(m)}$ can indeed be evaluated with a computer, since we only consider a finite number of columns, each having only a finite number of non-zero components. In particular, for any $n \in \tildeset^{\sol}(\tilde N)$ the coefficient $C_{k,n}^{(l,m)}$ vanishes for $k$ outside 
$\E(\tilde N) +2\set^{\sol}$.

We now focus on proving~\eqref{eq:Z1tail_toprove}. First, we recall that $\bw_{k-n}=0$ if $k-n\notin \set^{\sol}$. By definition of $\tildeset^{\sol}(\tilde N)$ this means that 
for any $1 \leq l,m \leq 3$ 
\begin{equation}\label{e:vanishingClmkn}
   C^{(l,m)}_{k,n}=0 \quad\text{and}\quad C^{(\phase,m)}_{n}=0 
   \qquad\text{for all } n\notin \tildeset^{\sol}(\tilde N)
   \text{ and all } k \in \E^\dag.
\end{equation}
In particular, provided $\E^\dag\subset \E(\tilde N)$, see Remark~\ref{r:NdagNtilde},  for all $1 \leq m \leq 3 $ and $n\notin \tildeset^{\sol}(\tilde N)$ the non-zero coefficients of the column $C^{(.,m)}_{.,n}$ only get hit by the tail part of $A$. Therefore, we obtain for all $1\leq l,m\leq 3$ and $n\notin \tildeset^{\sol}(\tilde N)$
\begin{align}
\label{eq:Z1tail_Alambda}
\frac{1}{\eta^{\vert n\vert_1}}\left\Vert A C^{(.,m)}_{.,n}\right\Vert_{\XX} &= \frac{1}{\eta^{\vert n\vert_1}}\sum_{l=1}^3 \left\Vert A^{(l,l)}C^{(l,m)}_{.,n}\right\Vert_{\ell^1_\eta} \nonumber\\
&\leq  \frac{1}{\eta^{\vert n\vert_1}}\sum_{l=1}^3 \sum_{k\notin\E^\dag} \left\vert\lambda_k\right\vert \left\vert C^{(l,m)}_{k,n}\right\vert \eta^{\vert k\vert_1}\\
&\leq \sum_{l=1}^3 \left[\delta_{l,m}\sum_{k\notin\E^\dag} \vert \lambda_k\vert \left\vert \sum_{p=1}^3 \frac{k_p}{\eta^{\left\vert n\right\vert_1}}(M\bw)^{(p)}_{k-n} \right\vert\eta^{\left\vert k\right\vert_1} \right. \nonumber\\ 
&\qquad\quad + \sum_{k\notin\E^\dag} \vert \lambda_k\vert \left\vert  \frac{1}{\eta^{\left\vert n\right\vert_1}} \left(D_m(M\bw)^{(l)}\right)_{k-n} \right\vert\eta^{\left\vert k\right\vert_1} \nonumber\\
&\qquad\quad + \sum_{k\notin\E^\dag} \vert \lambda_k\vert \left\vert  \sum_{p=1}^3 \frac{M_n^{(p,m)}}{\eta^{\left\vert n\right\vert_1}}\left(D_p\bw^{(l)}\right)_{k-n} \right\vert\eta^{\left\vert k\right\vert_1} \nonumber\\
& \left. \qquad\quad + \sum_{k\notin\E^\dag} \vert \lambda_k\vert \left\vert  \sum_{p=1}^3 \frac{n_p M_n^{(l,m)}}{\eta^{\left\vert n\right\vert_1}}  \bw^{(p)}_{k-n} \right\vert\eta^{\left\vert k\right\vert_1}\right]. \nonumber
\end{align}
We are going to bound the supremum over $n\notin \tildeset^{\sol}(\tilde N)$ of each of the four terms in the sum over~$l$ separately. For the first term, using Lemma~\ref{lem:C123_bis} we estimate
\begin{align}
\label{e:bound_term_with_derivative}
&\sup\limits_{n\notin\tildeset^{\sol}(\tilde N)} \sum_{k\notin\E^\dag} \vert \lambda_k\vert \left\vert \sum_{p=1}^3 \frac{k_p}{\eta^{\left\vert n\right\vert_1}}(M\bw)^{(p)}_{k-n} \right\vert\eta^{\left\vert k\right\vert_1} \\
&\qquad\qquad\qquad\leq  \sup\limits_{n\notin\tildeset^{\sol}(\tilde N)} \sum_{p=1}^3 \sum_{k\in \set^{\sol}} \vert\lambda_{k+n}\vert \vert k_p+n_p \vert \vert M\bw\vert ^{(p)}_{k}\eta^{\left\vert k\right\vert_1} \nonumber\\
&\qquad\qquad\qquad\leq  \sup\limits_{n\notin\tildeset^{\sol}(\tilde N)} \sum_{k\in\set^{\sol}} \left(\sum_{p=1}^3\vert\lambda_{k+n}\vert \vert k_p+n_p \vert\right) \max\limits_{p\in\{1,2,3\}}\vert M\bw\vert ^{(p)}_{k}\eta^{\left\vert k\right\vert_1} \nonumber\\
&\qquad\qquad\qquad\leq \left(\sup\limits_{k\notin\E(\tilde N) } \vert\lambda_k\vert \left(\vert k_1\vert+\vert k_2\vert+\vert k_3\vert\right)\right)  \left\Vert \max\limits_{p\in\{1,2,3\}}(M\bw)^{(p)} \right\Vert_{\ell^1_{\eta}} \nonumber\\
&\qquad\qquad\qquad\leq \frac{\sqrt 3}{\sqrt{\nu \tilde N}} \left\Vert \max\limits_{p\in\{1,2,3\}}(M\bw)^{(p)} \right\Vert_{\ell^1_{\eta}}.
\end{align}
Notice that, if we have an (essentially) $2D$ solution (see Section~\ref{s:introduction}), then one component of $M\bw$ is zero, which allows us to use the second estimate of Lemma~\ref{lem:C123_bis}, to get 
\begin{align*}
\sup\limits_{n\notin\tildeset^{\sol}(\tilde N)} \sum_{k\notin\E^\dag} \vert \lambda_k\vert \left\vert \sum_{p=1}^3 \frac{k_p}{\eta^{\left\vert n\right\vert_1}}(M\bw)^{(p)}_{k-n} \right\vert\eta^{\left\vert k\right\vert_1} \leq \frac{\sqrt 2}{\sqrt{\nu \tilde N}} \left\Vert \max\limits_{p\in\{1,2,3\}}(M\bw)^{(p)} \right\Vert_{\ell^1_{\eta}},
\end{align*}
For the second term, using Lemma~\ref{lem:C} instead of Lemma~\ref{lem:C123_bis}, we get
\begin{align*}
\sup\limits_{n\notin\tildeset^{\sol}(\tilde N)} \sum_{k\notin\E^\dag} \vert \lambda_k\vert \frac{1}{\eta^{\left\vert n\right\vert_1}} \left\vert D_m(M\bw)^{(l)}\right\vert_{k-n} \eta^{\left\vert k\right\vert_1} &\leq  \sup\limits_{n\notin\tildeset^{\sol}(\tilde N)} \sum_{k\in \set^{\sol}} \vert\lambda_{k+n}\vert  \left\vert D_m(M\bw)^{(l)}\right\vert_{k}\eta^{\left\vert k\right\vert_1} \\
&\leq \left(\sup\limits_{k\notin\E(\tilde N) } \vert\lambda_k\vert\right) \left\Vert D_m(M\bw)^{(l)} \right\Vert_{\ell^1_{\eta}} \\
&\leq  \frac{1}{\tilde N}\left\Vert D_m(M\bw)^{(l)}\right\Vert_{\ell^1_{\eta}}.
\end{align*}
For the third term, again using Lemma~\ref{lem:C}, as well as the fact that $M_n^{(p,m)}=0$ if $p=m$ and $\vert M_n^{(p,m)}\vert \leq 1$ if $p\neq m$, we infer that
\begin{align*}
\sup\limits_{n\notin\tildeset^{\sol}(\tilde N)} \sum_{k\notin\E^\dag} \vert \lambda_k\vert \left\vert  \sum_{p=1}^3 \frac{M_n^{(p,m)}}{\eta^{\left\vert n\right\vert_1}}\left(D_p\bw^{(l)}\right)_{k-n} \right\vert\eta^{\left\vert k\right\vert_1} \leq \frac{1}{\tilde N}\sum_{\substack{1\leq p\leq 3 \\ p\neq m }}\left\Vert D_p\bw^{(l)}\right\Vert_{\ell^1_{\eta}}.
\end{align*}
Finally, to bound the fourth term, we observe that 
\begin{equation*}
\left\vert n_p M_n^{(l,m)}\right\vert \leq  \chi_{\{l,m\},p} \bydef   \left\{
\begin{aligned}
&0 \quad & \text{if }l=m,\\
&\tfrac{1}{2}\quad &\text{if } l\neq m \text{ and }p\in\{l,m\},\\
&1\quad &\text{if }l\neq m \text{ and }p\notin\{l,m\}.
\end{aligned}
\right.
\end{equation*}
Hence 
\begin{align*}
\sup\limits_{n\notin\tildeset^{\sol}(\tilde N)} \sum_{k\notin\E^\dag} \vert \lambda_k\vert \left\vert  \sum_{p=1}^3 \frac{n_p M_n^{(l,m)}}{\eta^{\left\vert n\right\vert_1}}  \bw^{(p)}_{k-n} \right\vert\eta^{\left\vert k\right\vert_1} \leq \frac{1}{\tilde N}\sum_{p=1}^3 \chi_{\{l,m\},p}\left\Vert \bw^{(p)}\right\Vert_{\ell^1_{\eta}}.
\end{align*}

Putting everything together, we get that
\begin{align}
\label{eq:Z1tail}
\sup\limits_{n\notin \tildeset^{\sol}(\tilde N)}\frac{1}{\eta^{\vert n\vert_1}}\left\Vert A C^{(.,m)}_{.,n}\right\Vert_{\XX} &=\sup\limits_{n\notin\tildeset^{\sol}(\tilde N)} \sum_{l=1}^3 \frac{1}{\eta^{\vert n\vert_1}}\left\Vert A C^{(l,m)}_{.,n}\right\Vert_{\ell^1_\eta} \nonumber\\
&\leq \frac{\sqrt 3}{\sqrt{\nu \tilde N}} \left\Vert \max\limits_{1 \leq p \leq 3}(M\bw)^{(p)} \right\Vert_{\ell^1_{\eta}} + \frac{1}{\tilde N}\sum_{p=1}^3\left(1+\frac{1}{2}(1-\delta_{p,m})\right)\left\Vert \bw^{(p)}\right\Vert_{\ell^1_{\eta}} \nonumber\\
&\mbox{}\quad\quad + \frac{1}{\tilde N}\left(\sum_{l=1}^3\left(\left\Vert D_m(M\bw)^{(l)}\right\Vert_{\ell^1_{\eta}} + \sum_{p=1}^3 (1-\delta_{p,m}) \left\Vert D_p\bw^{(l)}\right\Vert_{\ell^1_{\eta}}\right)\right),
\end{align}
where $\sqrt 3$ can be replaced by $\sqrt 2$ for a 2D solution, and~\eqref{eq:Z1tail_toprove} is proven, which concludes the proof of Proposition~\ref{prop:Z1}.
\end{proof}

\begin{remark}
\label{rem:derivative_out}
To obtain the $Z_1$ bound, we have had to estimate $\left\Vert AC\right\Vert_{B(\XX,\XX)}$. Since $C$ is unbounded as an operator from $\XX$ to itself, we have to carefully handle the unbounded terms appearing in $C$, or equivalently in the derivative $D\Psi(\bw)$. Indeed, the unboundedness of $C$ is compensated by the decaying tail of $A$. Looking at~\eqref{e:DPsi}, we see that the derivatives in the second and the fourth term are not bothersome, as they are ``balanced'' by $M$ (which, very roughly speaking, acts as an antiderivative). The derivative in the third term of~\eqref{e:DPsi} is also not an issue, since it applies to $\bw$ which has only a finite number of non zero coefficients (hence its derivative still belongs to $\XX$ and its norm can be computed explicitly). However, something extra has to be done to control the derivative appearing in the first term of~\eqref{e:DPsi}. Using Lemma~\ref{lem:trick_derivative} on this specific term allows us to \emph{factor out} the derivative from the convolution product, and this derivative is then canceled by the tail part of $A$, which also acts like a kind of antiderivative (see~\eqref{e:bound_term_with_derivative} for the actual estimate).
This explains why we write $\sum_{p=1}^3 k_p(M\bw)^{(p)}_{k-n}$
rather than the equivalent (in view of $M\bw$ being divergence free) $\sum_{p=1}^3 n_p(M\bw)^{(p)}_{k-n}$ in the first term of~\eqref{e:Cklmn}. 
\end{remark}

\subsection{\boldmath${Z_2}$\unboldmath~bound for \boldmath${A\bigl(D\F(W)-D\F(\bW)\bigr)}$\unboldmath}
\label{s:Z2bound}

In this section, we present the last missing estimate for Theorem~\ref{th:radii_pol}, namely a computable bound for $\left\Vert A\left(D\F(W)-D\F(\bW)\right)\right\Vert_{B(\XX,\XX)}$ in term of $\left\Vert W-\bW \right\Vert_{\XX}$.

\begin{proposition}
\label{prop:Z2}
Assume that $A$ is defined as in Section~\ref{s:AAdag}. Then~\eqref{def:Z_2} is satisfied with
\begin{equation*}
Z_2\bydef(4+\sqrt{2})\NN{A},
\end{equation*} 
where
\begin{equation*}
\NN{A} \bydef \max\left[\max\limits_{1\leq m\leq 3}\max\limits_{n\in\E\dag} \frac{\vert n\vert_\infty}{\eta^{\left\vert n\right\vert_1}}\left\Vert A^{(.,m)}_{.,n}\right\Vert_{\XX}, \max\left(\frac{1}{\bar \Omega},\frac{1}{\sqrt{\nu N^\dag}}\right)\right].
\end{equation*}
\end{proposition}
\begin{proof}
Let $W',W''\in\XX$ and consider 
\begin{equation*}
z=\left(D\F(\bW+W')-D\F(\bW)\right)W''=D^2\F(\bW)(W',W'').
\end{equation*}
Proving the proposition amounts to showing that 
\begin{equation*}
\left\Vert Az\right\Vert_{\XX} \leq Z_2 \left\Vert W'\right\Vert_{\XX} \left\Vert W''\right\Vert_{\XX},\qquad \forall~W',W'' \in \XX.
\end{equation*}
By bi-linearity of $D^2\F(\bW)$, it is enough to assume that $\left\Vert W'\right\Vert_{\XX},\left\Vert W''\right\Vert_{\XX}\leq 1$, and show that $\left\Vert Az\right\Vert_{\XX} \leq Z_2$. We start by taking a closer look at $z$, which can be expanded as
\begin{align*}
\begin{pmatrix}
0 \\
i\bigl[\Omega'' D_4\omega' + \Omega' D_4\omega'' + \left(M\omega''\star \tD\right)\omega' + \left(M\omega'\star\tD\right)\omega'' - \left(\omega''\star\tD\right)M\omega' - \left(\omega'\star\tD\right)M\omega'' \bigr]
\end{pmatrix},
\end{align*}
where
\begin{equation*}
D_4 \omega=\begin{pmatrix}
D_4\omega^{(1)} \\ D_4\omega^{(2)} \\ D_4\omega^{(3)}
\end{pmatrix}.
\end{equation*}
Since the first component of $z$ vanishes, we have
\begin{equation}
\label{e:Az_for_Z2}
\left\Vert Az\right\Vert_{\XX} \leq \left\Vert A\right\Vert_{B(\XX_{-1,-1}^0,\XX)}\left\Vert z\right\Vert_{\XX_{-1,-1}},
\end{equation}
where $\XX_{-1,-1}^0$ is the subspace of $\XX_{-1,-1}$ defined as $\{0\} \times (\ell^1_{\eta,-1,-1}(\C))^3$ (that is, the first column $A^{(\cdot,\phase)}$ of $A$ does not play any role, since we multiply $A$ with a vector whose first component vanishes).

We now proceed to estimate both terms in the right-hand side of~\eqref{e:Az_for_Z2}. For the first one, we infer from Lemma~\ref{lem:C_inf} that
\begin{align}\label{e:normAminus}
\left\Vert  A\right\Vert_{B(\XX_{-1,-1}^0,\XX)} &\leq \max\left[\max\limits_{1\leq m\leq 3}\max\limits_{n\in\E\dag} \frac{\vert n\vert_\infty}{\eta^{\left\vert n\right\vert_1}}\left\Vert A^{(.,m)}_{.,n}\right\Vert_{\XX}, \max\left(\frac{1}{\bar \Omega},\frac{1}{\sqrt{\nu N^\dag}}\right)\right].
\end{align}

To bound the second term, we now use the splitting of $z$ introduced earlier and estimate each term separately. We start by
\begin{align*}
\left\Vert  \begin{pmatrix}
0 \\
\Omega'' D_4\omega' + \Omega' D_4\omega''
\end{pmatrix}
\right\Vert_{\XX_{-1,-1}} &\leq \vert\Omega''\vert \sum_{l=1}^3\left\Vert\omega'^{(l)}\right\Vert_{\ell^1_\eta}+\vert\Omega'\vert \sum_{l=1}^3\left\Vert\omega''^{(l)}\right\Vert_{\ell^1_\eta} \nonumber \\
&\leq \left\Vert W'\right\Vert_{\XX} \left\Vert W''\right\Vert_{\XX} \nonumber \\
&\leq 1.
\end{align*}

Next, we want to estimate
\begin{equation*}
\left\Vert \begin{pmatrix} 
0 \\ \left(M\omega''\star \tD\right)\omega'
\end{pmatrix}\right\Vert_{\XX_{-1,-1}}.
\end{equation*}
Thanks to Lemma~\ref{lem:trick_derivative}, we can rewrite\begin{equation*}
\left(M\omega''\star \tD\right)\omega' = \sum_{p=1}^3 D_p \left[(M\omega'')^{(p)} \ast \omega'\right],
\end{equation*}
where $(M\omega'')^{(p)} \ast \omega'$ must be understood as
\begin{equation*}
\begin{pmatrix}
(M\omega'')^{(p)} \ast \omega'^{(1)} \\ (M\omega'')^{(p)} \ast \omega'^{(2)} \\ (M\omega'')^{(p)} \ast \omega'^{(3)}
\end{pmatrix}.
\end{equation*}
For each $p\in\{1,2,3\}$, we estimate
\begin{align*}
\left\Vert  \begin{pmatrix} 
0 \\  D_p\left[(M\omega'')^{(p)} \ast \omega' \right]
\end{pmatrix}\right\Vert_{\XX_{-1,-1}} & \leq \left\Vert \begin{pmatrix} 
0 \\  (M\omega'')^{(p)} \ast \omega'
\end{pmatrix}\right\Vert_{\XX} \nonumber \\
& \leq \left\Vert (M\omega'')^{(p)}\right\Vert_{\ell^1_\eta}\left\Vert W'\right\Vert_{\XX} \nonumber \\
& \leq \left\Vert (M\omega'')^{(p)}\right\Vert_{\ell^1_\eta},
\end{align*}
and using that 
\begin{equation*}
\sum_{p=1}^3 \left\Vert (M\omega'')^{(p)}\right\Vert_{\ell^1_\eta} \leq  \sum_{p=1}^3 \left\Vert \omega''^{(p)}\right\Vert_{\ell^1_\eta}\leq 1,
\end{equation*}
we get
\begin{equation*}
\left\Vert \begin{pmatrix} 
0 \\ \left(M\omega''\star \tD\right)\omega'
\end{pmatrix}\right\Vert_{\XX_{-1,-1}} \leq 1.
\end{equation*}
Similarly, we get that 
\begin{equation*}
\left\Vert \begin{pmatrix} 
0 \\ \left(M\omega'\star \tD\right)\omega''
\end{pmatrix}\right\Vert_{\XX_{-1,-1}} \leq 1.
\end{equation*}

Finally, we have to estimate
\begin{align*}
\left\Vert  \begin{pmatrix} 
0 \\ \left(\omega''\star\tD\right)M\omega'
\end{pmatrix}\right\Vert_{\XX_{-1,-1}} \leq \left\Vert  \begin{pmatrix} 
0 \\ \left(\omega''\star\tD\right)M\omega'
\end{pmatrix}\right\Vert_{\XX}.
\end{align*}
Starting from
\begin{align*}
\left\Vert \begin{pmatrix} 
0 \\ \left(\omega''\star\tD\right)M\omega'
\end{pmatrix} \right\Vert_{\XX} &\leq \sum_{p=1}^3\sum_{l=1}^3 \left\Vert \omega''^{(p)}\right\Vert_{\ell^1_\eta} \left\Vert D_p(M\omega')^{(l)}\right\Vert_{\ell^1_\eta} 
\end{align*}
we estimate
\begin{align*}
\sum_{l=1}^3 \left\Vert D_p(M\omega')^{(l)}\right\Vert_{\ell^1_\eta} &\leq  \sum_{l=1}^3\sum_{n\in\Z^4_*} \vert n_p\vert \sum_{m=1}^3 \left\vert M^{(l,m)}_n \omega'^{(m)}_n\right\vert \eta^{\vert n\vert_1} \\
&\leq  \sum_{n\in\Z^4_*} \sum_{m=1}^3 \left(\sum_{l=1}^3 \vert n_p\vert  \left\vert M^{(l,m)}_n\right\vert\right) \left\vert \omega'^{(m)}_n\right\vert \eta^{\vert n\vert_1}.
\end{align*}
Up to a permutation of the indices $1$, $2$ and $3$, we have for any $l,m,p \in \{1,2,3\}$  that either
\begin{equation*}
\sum_{l=1}^3 \vert n_p\vert  \left\vert M^{(l,m)}_n\right\vert \leq \frac{\vert n_1\vert (\vert n_2\vert + \vert n_3\vert)}{\tn^2} \leq 1,
\end{equation*}
or 
\begin{equation*}
\sum_{l=1}^3 \vert n_p\vert  \left\vert M^{(l,m)}_n\right\vert \leq \frac{\vert n_1\vert (\vert n_1\vert + \vert n_2\vert)}{\tn^2} \leq \frac{1+\sqrt 2}{2},
\end{equation*}
for all $\tn \in \Z^3 \setminus \{0\}$.
Therefore
\begin{align*}
\sum_{l=1}^3 \left\Vert D_p(M\omega')^{(l)}\right\Vert_{\ell^1_\eta} &\leq  \frac{1+\sqrt 2}{2} \left\Vert W'\right\Vert_\XX,
\end{align*}
and we end up with 
\begin{align*}
\left\Vert \begin{pmatrix} 
0 \\ \left(\omega''\star\tD\right)M\omega'
\end{pmatrix}\right\Vert_{\XX} \leq \frac{1+\sqrt 2}{2}.
\end{align*}
Similarly, we have
\begin{align*}
\left\Vert \begin{pmatrix} 
0 \\ \left(\omega'\star\tD\right)M\omega''
\end{pmatrix}\right\Vert_{\XX} \leq \frac{1+\sqrt 2}{2}.
\end{align*}

Adding everything up, we end up with 
\begin{align*}
\left\Vert z\right\Vert_{\XX_{-1,-1}} \leq  4+\sqrt{2},
\end{align*}
which concludes the proof of Proposition~\ref{prop:Z2}
\end{proof}

\begin{remark}
It should be noted that the slightly unusual step~\eqref{e:Az_for_Z2} is crucial for what is to come in Section~\ref{s:symboundZ2}. Usually, one would try to postpone this step to directly ``cancel'' some of the unbounded derivative operators in $z$ with $A$, by first splitting $z$ into several terms and estimating then like
\begin{equation*}
\left\Vert A \begin{pmatrix}
0 \\
\Omega'' D_4\omega' + \Omega' D_4\omega''
\end{pmatrix} \right\Vert_{\XX} \leq \left\Vert AD_4 \right\Vert_{B(\XX,\XX)} \left\Vert \begin{pmatrix}
0 \\
\Omega'' \omega' + \Omega' \omega''
\end{pmatrix} \right\Vert_{\XX}.
\end{equation*} 
However, this splitting may not conserve some of the symmetries that $z$ has as a whole, and therefore would not be compatible with the symmetry reduced variables used later on.
\end{remark}

\section{The symmetries}
\label{s:symmetries}

In this section, we introduce the formalism that allows us to take advantage of the symmetries through the a posteriori validation procedure. In Section~\ref{s:groupaction}, we introduce some group actions that we use to explicitly describe the symmetries that we are considering, in the physical space (that is symmetries acting on the velocity $u$ or the vorticity $\omega$, seen as functions). In Section~\ref{s:equivariance}, we proceed to explain why the Navier-Stokes equations are equivariant under these group actions. Equivalent group actions in Fourier space are then defined in Section~\ref{s:symFourier}, and the equivariance of $\F$ is shown. In Section~\ref{s:redsymvar}, we introduce the subspace $\XXsym$ of~$\XX$ containing the symmetric solutions, as well as a \emph{minimal} (in term of number of Fourier mode used) subspace $\XXred$ of~$\XX$ that can be used to describe these solutions. A \emph{reduced} version $\FFred$ of $\F$ is then defined on this subspace, and sufficient conditions for the existence of zeros of $\FFred$ are then given in Theorem~\ref{th:radii_pol_sym}, which mimics Theorem~\ref{th:radii_pol}. Finally, in Section~\ref{s:symmetrybounds}, we show that the bounds obtained in Section~\ref{s:estimates} are compatible with the symmetries, and readily give us the estimates needed to apply Theorem~\ref{th:radii_pol_sym}, which is then done in Section~\ref{s:application}.

\subsection{Symmetries in physical space}
\label{s:groupaction}

We will restrict our attention to solutions $(u,p)$, $u : \R^3 \times \R \to \R^3$ and $p : \R^3 \times \R  \to \R$ which are 
$2 \pi$-periodic
in space and $2 \pi / \Omega$-periodic in time. We may thus interpret $u$ and $p$ as a function on $\T^3\times \sphere^1$, where $\T^3$ is a 3-dimensional torus and $\sphere^1 = \R / (2 \pi / \Omega)$ is a circle.

We look for solutions with additional symmetries.
In particular, we will consider periodic solutions which are invariant under a symmetry group $G$, which acts on $(u,p)$ through a \emph{right} 
action $a_g$, $g \in G$, of the form
\[
  [a_g u ] (x,t) = b_g u (c_g x, d_g t), \qquad [a_g p] (x,t) = p(c_g x, d_g t),
\]
where $b_g$ is a right action on $\R^3$, while $c_g$ is a left action on $\T^3$
and $d_g$ is a left action on $\sphere^1$. We abuse the notation slightly by using $a_g$ to denote both the action on the pair $(u,p)$ and on each of its components, but it should be clear from context which one is considered.

The symmetries under consideration in this work lead to actions of the form
\begin{equation}
\label{def:ag}
[a_g u ] (x,t) = C^{T}_g u(C_g x + 2\pi\tilde{C}_g, t + 2\pi D_g /\Omega), \qquad [a_g p ] (x,t) = p(C_g x + 2\pi\tilde{C}_g, t + 2\pi D_g /\Omega),
\end{equation}
where $C_g$ is a $3\times 3$ signed permutation matrix (it represents an element of $O_h$, the symmetry group of the cube):
there is a permutation $\tau_g$ of $\{1,2,3\}$ and a map
$\rho_g: \{1,2,3\} \to \{-1,1\}$ 
such that $(C_g)_{mm'}  = \rho_g(m') \delta_{m \tau_g(m')}$, $\tilde{C}_g$ is a vector in $\R^3$ and $D_g$ a real number. An explicit description of the symmetries satisfied by the solutions considered in this work, together with the associated $C_g$, $\tilde C_g$ and $D_g$ is given in Section~\ref{s:applic_2D}.

We now describe how these actions behave with respect to differential operators.
\begin{lemma}
\label{lem:ag_diff}
Let $u : \R^4 \to \R^3$ and $p : \R^4 \to \R$ be smooth functions, and assume $a_g$ is defined as in~\eqref{def:ag}, with $C_g$ an orthogonal matrix. Then
\\[1mm]
\hspace*{5mm}\textup{(a)}
$\nabla a_g p = a_g \nabla p$,
\\[1mm]
\hspace*{5mm}\textup{(b)}
$\nabla \cdot a_g u = a_g \nabla \cdot u$,
\\[1mm]
\hspace*{5mm}\textup{(c)}
$\Delta a_g u = a_g\Delta u$,
\\[1mm]
\hspace*{5mm}\textup{(d)}
$(a_g u\cdot \nabla) a_g u = a_g (u\cdot\nabla)u$,
\\[1mm]
\hspace*{5mm}\textup{(e)}
$\nabla \times a_g u = \det(C_g)a_g\nabla\times u$.
\end{lemma}
\begin{proof}
The first four identities follow directly from the chain rule, using that the change of variable $x\mapsto C_g x + 2\pi\tilde C_g$ transforms $\nabla$ into $C^{T}_g \nabla$. The last one is slightly less straightforward. Writing $y= C_g x + 2\pi\tilde C_g$, we have
\begin{align*}
[\nabla \times a_g u](x,t) & = [C^{T}_g\nabla \times C^{T}_g u](y,t+2\pi D_g/\Omega).
\end{align*}
For any $v\in \R^3$ we then compute
\begin{align*}
C^{T}_g v \cdot \left(C^{T}_g\nabla \times C^{T}_g u\right) &= \det(C^{T}_g v,C^{T}_g\nabla,C^{T}_g u) \\
&=\det(C^{T}_g)\det(v,\nabla,u) \\
&=\det(C_g)v\cdot (\nabla\times u).
\end{align*}
Since this holds for all $v\in \R^3$, we infer
\begin{equation*}
C^{T}_g\nabla \times C^{T}_g u = \det(C_g)C^{T}_g \nabla\times u,
\end{equation*}
which yields identity \textup{(e)}.
\end{proof}

From the action $a_g$ on $u$, we can define a corresponding action $\ta_g$ on $\omega = \nabla \times u$, given by
\begin{equation}
\label{def:tildeag}
\ta_g (\nabla \times u) = \nabla \times a_g u.
\end{equation}
From~\eqref{def:ag} and Lemma~\ref{lem:ag_diff} we infer
\begin{equation}
\label{eq:tildeag}
[\ta_g \omega ] (x,t) = \det(C_g) C^{T}_g \omega (C_g x + 2\pi\tilde{C}_g, t + 2\pi D_g/\Omega).
\end{equation}
We note that $\det(C_g) = \pm 1$,
since $C_g$ is orthonormal, hence the only difference between the action of $a_g$ and $\ta_g$ is multiplication by a factor $\pm 1$ (depending on $g$).

\subsection{Invariance and equivariance}
\label{s:equivariance}

We say that $u$ is $G$-invariant under the action $a_g$ if
\[
  a_g u = u \qquad \text{for all } g\in G.
\]  
Similarly, we say that $\omega$ is $G$-invariant under the action $\ta_g$ if 
\[
  \ta_g \omega = \omega \qquad \text{for all } g\in G.
\]
If it is clear which action is meant between $\ta_g$ and $a_g$, we only speak about $G$-invariance without mentioning the action. Notice that, by Lemma~\ref{lem:ag_diff}(b), the collection of divergence free vector fields is invariant under~$G$. Analogously, having zero spatial average, denoted by $\int_{\T^3} u=0$, is also a $G$-invariant property. Since the Laplacian is invertible on the set of vector fields with zero spatial average, and it commutes with the group action (Lemma~\ref{lem:ag_diff}(c)), it follows from Lemma~\ref{lem:ag_diff}(e) that
\begin{equation}
\label{eq:agM}
a_g M \omega = M \ta_g \omega \qquad \text{for all } g\in G.
\end{equation}

Next, we mention the equivariance properties of the Navier-Stokes equations that provide intuition for the sequel. To formalize the discussion, we define the pair $\mathcal{P} = (\mathcal{P}_1,\mathcal{P}_2)$ of maps
\begin{alignat*}{1}
  \mathcal{P}_1(u,p) &\bydef \partial_t u + (u \cdot \nabla) u - \nu \Delta u  + \nabla p -  f ,\\
  \mathcal{P}_2(u,p) &\bydef \nabla \cdot u ,
\end{alignat*}
on some set $\mathbb{X}$ of smooth vector fields $u$ and scalar functions $p$.
Assuming that $f$ is $G$-invariant under the action $a_g$, it follows from Lemma~\ref{lem:ag_diff} that $\mathcal{P}$ is $G$-equivariant:
\[
  \mathcal{P}(a_g u, a_g p) = (a_g \mathcal{P}_1 (u,p),a_g \mathcal{P}_2 (u,p)) \qquad \text{ for all } g\in G \text{ and } (u,p) \in \mathbb{X}.
\]
We note that $\Phi=\Phi(u)$ defined in~\eqref{e:defPhi} is given by $\mathcal{P}_1(u,0)$, hence $\Phi$ is also $G$-equivariant: $\Phi(a_g u)=a_g \Phi(u)$.

Next, we consider the vorticity formulation and set
\[
  \mathcal{Q}(\omega) \bydef \nabla \times \mathcal{P}_1(M\omega,0) = \nabla \times \Phi(M\omega).
\]
It then follows from~\eqref{def:tildeag}, \eqref{eq:agM} and the equivariance of $\Phi$ that $\mathcal{Q}$ is $G$-equivariant under the action $\ta_g$.
As before, we simply refer to $G$-equivariance without mentioning the action, if the latter is clear from the context.
The explicit expression of $\mathcal{Q}$ is given by
\begin{equation*}
\mathcal{Q}(\omega)= \partial_t\omega +(M\omega\cdot\nabla)\omega - (w\cdot\nabla)M\omega - M\omega(\nabla\cdot\omega) - \nu\Delta\omega -\rhs.
\end{equation*}
Notice that this expression is only equal to $F(W)$ when $\omega$ is divergence free, hence a supplemental argument is required to show that $F$ is $G$-equivariant, which we provide next.

\begin{lemma}
\label{lem:equivariance_F}
$F$ defined is~\eqref{e:defFn} is $G$-equivariant, i.e.
\begin{equation*}
F(\tilde a_g W)=\tilde a_g F(W),
\end{equation*}
where $\tilde a_g W=\tilde a_g (\Omega,\omega)$ is defined as $(\Omega,\tilde a_g \omega)$.
\end{lemma}
\begin{proof}
Looking at~\eqref{eq:F_physical_space}, the statement follows from~\eqref{eq:agM} and Lemma~\ref{lem:ag_diff}.
\end{proof}

Finally, obviously $\mathcal{P}_1(u,p)$ is $G$-invariant if both $u$ and $p$ are $G$-invariant, and similarly for $\mathcal{Q}(\omega)$ and $F(W)$.

\begin{remark}
\label{rem:invariancep}
Concerning the pressure term in the Navier-Stokes equation,
in the appendix we describe explicitly a well-defined map $\Gamma$ (see Lemma~\ref{lem:exist_grad}) such that
\[
  \left\{ 
  \begin{array}{l}
  \nabla \times \Phi =0 \\  
  \int_{\T^3} \Phi =0 \\
  p = \Gamma \Phi
  \end{array}
  \right.
  \qquad\Longleftrightarrow \qquad
  \left\{ 
  \begin{array}{l}
  \Phi = -\nabla p \\  
  \int_{\T^3} p =0.
  \end{array}
  \right.
\]
From this it follows that $[\Gamma a_g \Phi] (x,t)=  [\Gamma \Phi] (c_g x,d_g t)$.
Hence, when recovering the pressure from the vorticity, see the proof of Lemma~\ref{lem:eq_NS_vorticity}, we can infer that $G$-invariance of $\omega$ implies $G$-invariance of $\Phi(M\omega)$, from which in turn $G$-invariance of $p=\Gamma\Phi(M\omega)$ follows.
\end{remark}

\subsection{Representation of the symmetry group in Fourier space}
\label{s:symFourier}

We recall~\eqref{eq:function_coeffs_omega} that we write the Fourier transform on the $m$-th component of the vorticity as 
\begin{equation}\label{e:omegaFourier}
  \omega^{(m)}(x,t) = \sum_{n \in \Z^4_*} 
        \c^{(m)}_n  e^{i ( \tn \cdot x  + n_4 \Omega t)}, 
\end{equation}
where $\omega_0$ is assumed to be zero. We note that, as in Section~\ref{s:setup}, it should be clear from the context whether $\omega$ is to be interpreted in physical or Fourier space. From now on we will denote $\j = (n,m) \in \Z^4_* \times \{1,2,3\}$ and write
\[
  \c_\j = \c^{(m)}_n .
\]
We define the index set
\[
  \J \bydef \Z^4_* \times \{1,2,3\},
\]
and the set of Fourier coefficients by 
\[
  \X \bydef \bigl\{ \c=(\c_\j)_{\j\in\J} : \c_j \in \C \, , \,
  \left\Vert \omega \right\Vert_\X <  \infty \bigr\} ,
  \quad\text{with}\quad
  \left\Vert \omega \right\Vert_\X \bydef \sum_{j \in \J} \weight_\j |\c_\j| ,
\]
where, for convenience, we introduce the weights
\begin{equation} \label{eq:weights_over_J}
  \weight_\j \bydef \eta^{|n|_1} \qquad\text{for } j=(n,m) \in \J.
\end{equation}
\begin{remark}
Notice that the space $\XX$ defined is Section~\ref{sec:notations} is nothing but $\C\times\X$. When looking at the action $\tilde a_g$ in Fourier space, it is enough to consider $\X$ because the action does neither act on $\Omega$ nor depend on it (see Remark~\ref{rem:formula_alpha_beta}). The reason for grouping $n$ and $m$ as a duo in $j=(n,m)$ is that the group action on the Fourier coefficients (made explicit below) permutes the index duos $j \in \J$ rather than $n \in \Z^4_*$ and $m \in \{1,2,3\}$ individually, see~\eqref{e:explicitbeta}.
\end{remark}
We define the basis vectors $\e_\j \in \X$ by
\[
  (\e_\j)_{\j'} = \delta_{n n'}\delta_{m m'},
  \qquad\text{for } \j=(n,m),\j'=(n',m') \in \J.
\]

For $g \in G$, the group actions $a_g$ and $\tilde a_g$ described in Sections~\ref{s:groupaction} and \ref{s:equivariance} correspond to a \emph{right} group action $\gamma_g$ on the Fourier coefficients satisfying
\begin{equation}\label{e:symmetryequivalence}
  [\ta_g \omega]^{(m)}(x,t)  = \sum_{n\in \Z^4_*} [\gamma_g \c]^{(m)}_n 
    e^{i ( \tn \cdot x  + n_4 \Omega t)}.
\end{equation}
This action is represented component-wise by
\begin{equation}
\label{def:gamma}
  (\gamma_g \c)_\j = \alpha_g(\j) \, \c_{\beta_g(\j)}.
\end{equation}
Here $\beta_g$ is itself a \emph{left} group action on $\J$, i.e.,
\begin{equation}\label{e:betaproduct}
 \beta_{g_1 g_2}(\j)=\beta_{g_1}(\beta_{g_2}(\j)),
\end{equation}
whereas $\alpha_g(\j) \in \{ z \in \C : |z|=1 \} $ for all $\j \in \J$. 
The product structure on $\alpha$ is given by
\begin{equation}\label{e:alphaproduct}
   \alpha_{g_1 g_2}(\j) = \alpha_{g_1}(\beta_{g_2}(\j)) \alpha_{g_2}(\j),
\end{equation}
which follows directly from $\gamma_{g_1 g_2} = \gamma_{g_2} \gamma_{g_1}$.
Note that $\alpha_g$ is \emph{not} a group action.

Obviously, by~\eqref{e:symmetryequivalence} the $G$-equivariance of $F$ in terms of the \emph{physical} symmetries $\tilde a_g$ induces a $G$-equivariance in Fourier space, in term of the action $\gamma_g$.
\begin{lemma}\label{l:Fequivariant}
Let $g \in G$. Then 
	$F(\gamma_g W)  = \gamma_g F(W)$ for all  $W=(\Omega,\omega) \in \XX$, where again we extend the action notation: $\gamma_g(\Omega,\omega)=(\Omega,\gamma_g\omega)$.
\end{lemma}

\begin{remark}
\label{rem:formula_alpha_beta}
In terms of the notation introduced in~\eqref{def:ag},
the expressions for $\beta_g$ and $\alpha_g$ are, with $\j=(n,m)=((\tn,n_4),m)$,
\begin{equation}\label{e:explicitbeta}
  \beta_g((\tn,n_4),m) = (C_g \tn , n_4 , \tau_g(m)) 
\end{equation}
and
\[
  \alpha_g((\tn,n_4),m) = 
  \det(C_g) \rho_g(m)  e^{2i\pi( \tn \cdot C^T_g \tC_g + n_4 D_g)} .
\]
Explicit formulas for the symmetries under consideration in this work are given in Section~\ref{s:applic_2D}. 

Two straightforward properties that will be useful later are 
\begin{subequations}
\label{e:ntominusn}
\begin{equation}
  \text{if } (n',m')=\beta_g(n,m) \quad\text{then}\quad  \beta(-n,m)=(-n',m),
\end{equation}
and
\begin{equation}
  \alpha_g(-n,m) = (\alpha_g(n,m))^*.
\end{equation}
\end{subequations}
\end{remark}

\begin{remark}
The real-valuedness of $u$ provides another symmetry, 
which in Fourier space corresponds to invariance under the transformation $\gamma_*$ introduced in Definition~\ref{def:complex_conj}.
One way to deal with this symmetry is to consider the extended symmetry group $\hG$ generated by $G$ and~$\gamma_*$. However, one would have to interpret $\alpha_*(\j)$ as the complex conjugation operator on $S^1 \subset \C$, which is not a multiplication operator, and it is also not differentiable, which would cause problems later on. One solution is to split $\c$ into real and imaginary parts, but that is not the way we proceed here. Instead, we stick with the symmetry group $G$ and consider the action of $\gamma_*$ separately.
We note that $\gamma_*$ commutes with $\gamma_g$:
\begin{equation}\label{e:gammaggammastar}
  \gamma_* \gamma_g = \gamma_g \gamma_*
  \qquad \text{for all } g\in G,
\end{equation}
which follows either from straightforward symmetry consideration in physical space, or directly in term of $\alpha_g$ and $\beta_g$ from~\eqref{e:ntominusn}.
\end{remark}

\subsection{Reduction to symmetry variables}
\label{s:redsymvar}

The set of Fourier coefficients  that are  $G$-invariant is given by
\begin{equation*}
  \Xsym \bydef \{ \c \in \X  \,:\, \gamma_g  \c =\c \text{ for all } g \in G \}.
\end{equation*}
Again we have a correspondence with the physical space, by~\eqref{e:symmetryequivalence} elements of $\Xsym$ correspond to functions that are $G$-invariant via the action $\tilde a_g$.
\begin{lemma}\label{l:ctou}
Let $W=(\Omega,\omega) \in \C\times\Xsym$. Then $\omega(x,t)$ given by~\eqref{e:omegaFourier} satisfies $[\tilde a_g \omega](x,t)=\omega(x,t)$ for all $g\in G$.
\end{lemma}

In this section we study the properties of $\Xsym$.
Some of the arguments in this section are essentially the same as the ones in~\cite[Section~3.2]{JBJFsymmetry}. Nevertheless we repeat the main arguments and definitions here, since the setting and notation are slightly different. The short proofs of some of the lemmas are also transcribed here to keep the current paper self-contained. 

We first list some properties of $\alpha$ that will be useful in what follows.
\begin{lemma}{\cite[Lemma~3.5]{JBJFsymmetry}}\label{l:alpha}
Let $\j \in \J$ and $g,\tg\in G$. \\[1mm]
\hspace*{5mm}\textup{(a)} $\alpha_g(\j)\,\alpha_{g^{-1}}(\beta_g(\j)) =1$.\\[1mm]
\hspace*{5mm}\textup{(b)} If $\beta_g(\j)=\beta_{\tg}(\j)$ and $\alpha_{\tg^{-1}g}(\j)=1$, then $\alpha_{g}(\j)=\alpha_{\tg}(\j)$.\\[1mm]
\hspace*{5mm}\textup{(c)} If $\beta_g(\j)=\j$, then 
$\alpha_{\tg g \tg^{-1}}(\beta_{\tg}(\j))=\alpha_{g}(\j) $.
\end{lemma}
\begin{proof}
For part (a) we use~\eqref{e:alphaproduct} to infer that
\[
  1 = \alpha_e(\j) = \alpha_{g^{-1}g}(\j)
  = \alpha_{g^{-1}}(\beta_g(\j)) \, \alpha_g(\j) .
\]
For part (b) we write $g=\tg \tg^{-1}g$. From the first assumption and~\eqref{e:betaproduct} if follows that $\beta_{\tg^{-1}g}(\j)=j$.
By applying~\eqref{e:alphaproduct} to $g_1=\tg$ and $g_2=\tg^{-1}g$, we see that
the second assumption implies that
\[
  \alpha_{g}(\j)=\alpha_{\tg \tg^{-1}g}(\j)
  = \alpha_{\tg}(\beta_{\tg^{-1}g}(\j)) \, \alpha_{\tg^{-1}g}(\j) = 
  \alpha_{\tg}(\j) .
\]
For part (c) we write $\j'=\beta_{\tg}(\j)$ and 
apply~\eqref{e:alphaproduct} twice to obtain
\begin{equation*}
  \alpha_{\tg g \tg^{-1}}(\j') = 
  \alpha_{\tg}(\beta_g(\beta_{\tg^{-1}}(\j'))) \, 
  \alpha_{g}(\beta_{\tg^{-1}}(\j')) \, \alpha_{\tg^{-1}}(\j') 
    = \alpha_{\tg}(\j) \, \alpha_{g}(\j) \, \alpha_{\tg^{-1}}(\j')
	= \alpha_g(\j),
\end{equation*} 
where the final equality follows from part (a).
\end{proof}

The symmetry group $G$ may be exploited to reduce the number of independent variables of elements in $\Xsym$. 
From now on we use the notation
\[
  g \acts \j \bydef \beta_g(\j) \qquad \text{for } g \in G.
\]
For any $\j\in\J$ we define the stabilizer
\[
  G_\j \bydef \{ g \in G \,:\, g \acts \j = \j \},
\]
and the orbit
\[
  G \acts \j \bydef \{ g \acts \j \,:\, g \in G \}.
\]
\begin{remark}[Orbit-stabilizer]\label{r:orbitstabilizer}
The orbit-stabilizer theorem implies that $|G_{\j'}|=|G_\j|$ for all $\j' \in G \acts \j$, where $| \cdot |$ denotes the cardinality. 
Indeed, stabilizers of different elements in  an orbit are related by conjugacy,
and $|G|=|G_\j| \cdot |G \acts \j|$. More generally, when $q$ is a function from $\J$, or from subset thereof that is $G$-invariant under the action $\beta_g$, to some linear space, then we have 
\begin{equation}\label{e:equaltime}
  \sum_{g \in G} q(g \acts \j) = |G_\j| \sum_{\j' \in G \acts \j} q(\j')
\qquad\text{for any } \j \in \J.
\end{equation}
\end{remark}

\begin{lemma}{\cite[Lemma~3.8]{JBJFsymmetry}}\label{l:alphadichotomy}
Let $\j\in\J$ be arbitrary. We have the following dichotomy:\\
\hspace*{5mm}\textup{(a)} either $\alpha_g(\j)=1$ for all $g \in G_{\j}$;\\
\hspace*{5mm}\textup{(b)} or $\sum_{g \in G_{\j}} \alpha_g(\j)=0$.
\end{lemma}
\begin{proof}
Fix $\j\in\J$.
We see from~\eqref{e:alphaproduct} that 
$\alpha_{g_1 g_2}(\j)=\alpha_{g_1}(\j)\alpha_{g_2}(\j)$ for all $g_1,g_2 \in G_\j$.
Hence we can interpret $\alpha_g(\j)$ as a group action of the stabilizer subgroup $G_\j$, acting by multiplication on the unit circle 
$S^1 = \{ z \in \C : |z|=1 \}$.
We consider the stabilizer of $1 \in S^1$:
\[
  H_1 \bydef \{ g \in G_\j : \alpha_g(\j) =1  \},
\]
and its orbit 
\[
  O_1 \bydef \{\alpha_g(\j) : g \in G_\j \}.
\]
By the orbit-stabilizer theorem (we use Remark~\ref{r:orbitstabilizer}, but now for the $G_\j$-action $\alpha_g(\j)$)
\begin{equation}\label{e:sumalpha}
  \sum_{g \in G_\j} \alpha_g(\j) = |H_1| \sum_{z \in O_1} z.
\end{equation}
The set $O_1 \subset S^1$ is invariant under multiplication and division. In particular, if $|O_1|=N \in \mathbb{N}$, then $O_1 = \{ e^{2\pi i n/N} : n=0,1,\dots,N-1 \}$.
If $N=1$ we have $H_1=G_\j$ and alternative (a) follows, whereas
if $N>1$ then we see that $\sum_{z \in O_1} z = \sum_{n=0}^{N-1} e^{2\pi i n/N}=0$, hence we conclude from~\eqref{e:sumalpha} that alternative (b) holds.
\end{proof}

The indices for which alternative (b) in Lemma~\ref{l:alphadichotomy} applies are denoted by
\[
  \Jtriv \bydef \Bigl\{ \j\in\J \,:\, \sum_{g \in G_\j} \alpha_g(\j) = 0 \Bigr\}  .
\]
It follows from the next lemma, which is a slight generalization of Lemma~\ref{l:alphadichotomy}, that the set $\Jtriv$ is invariant under $G$.
\begin{lemma}{\cite[Lemma~3.10]{JBJFsymmetry}}\label{l:alphadichotomy2}
Let $\j\in\J$ be arbitrary. We have the following dichotomy:\\
\hspace*{5mm}\textup{(a)} either $\alpha_g(\j')=1$ for all $g \in G_{\j'}$ and all $\j' \in G \acts \j$;\\
\hspace*{5mm}\textup{(b)} or $\sum_{g \in G_{\j'}} \alpha_g(\j')=0$ for all $\j' \in G \acts \j$.
\end{lemma}
\begin{proof}
	Let $\j\in\J$ and $\j' \in G \acts \j$.
	Let $\tg \in G$ be such that $\tg \acts \j=\j'$.
	Then a conjugacy between $G_\j$ and $G_{\j'}$ is given by $g \to \tg g \tg^{-1}$. 
	It follows from Lemma~\ref{l:alpha}(c) that 
\[
  \sum_{g \in G_{\j'}} \alpha_g(\j') = 
  \sum_{g \in G_\j} \alpha_{\tg g \tg^{-1}} (\j') =
  \sum_{g \in G_\j} \alpha_g(\j).
\]
The assertion now follows from Lemma~\ref{l:alphadichotomy}.
\end{proof}	
	
When considering $\c \in \Xsym$,
the indices $\j$ in $\Jtriv$ are the ones for which $\c_\j$ necessarily vanishes.
\begin{lemma}{\cite[Lemma~3.11]{JBJFsymmetry}}\label{l:zeroindices}
  Let $\c \in \Xsym$. Then $\c_\j=0$ for all $\j \in \Jtriv$.
\end{lemma}
\begin{proof}
Fix $\j\in\J$.
For any $\c \in \Xsym$ we have in particular $[\gamma_g \c]_\j = \c_\j$
for all $g \in G_\j$. By summing over $g \in G_\j$ and using that $g \acts \j=\j$ for 
$g \in G_\j$, we obtain
\[
 |G_\j| \, \c_\j =  \sum_{g \in G_\j} \c_\j = 
 \sum_{g \in G_\j} (\gamma_g \c)_\j  =
 \sum_{g \in G_\j} \alpha_g(\j) \c_\j =
 \c_\j  \sum_{g \in G_\j} \alpha_g(\j) .
\]
If $\j \in \Jtriv$ then the right hand side vanishes, hence
$\c_\j=0$.
\end{proof}

Lemma~\ref{l:zeroindices}
implies that 
\[
  \Xsym \subset \{ \c \in \X \,:\, \c_\j=0 \text{ for all } \j \in \Jtriv \}.
\]
In other words, we may restrict attention to the Fourier coefficients corresponding to indices in the complement 
\[
  \Jsym \bydef \J \setminus \Jtriv .
\]

\begin{remark}\label{r:Zcode}
On $S^1$ we have $z^{-1}=z^*$. In particular, Lemma~\ref{l:alphadichotomy} implies that
\begin{equation}\label{e:alphainverse}
  \sum_{g \in G_{\j}} \alpha^{-1}_g(\j) = \left(\sum_{g \in G_{\j}} \alpha_g(\j)\right)^* =  
  \begin{cases} 
	0   &  \text{for } \j \in \Jtriv , 	 \\
  |G_\j| &  \text{for } \j \in \Jsym . 	  
  \end{cases}
\end{equation}
This identity is used, through Lemma~\ref{l:averageunit}, to extract both the set $\Jtriv$ and the values $|G_\j|$ for $\j \in \Jsym$ from what is computed in the code, see Remark~\ref{r:coding}.
\end{remark}

For $\c \in \Xsym$ the coefficients $\c_\j$ with $\j \in \Jsym$ are not all independent. 
To take advantage of this, we choose a fundamental domain $\Jdom$ 
of $G$ in $\J$,
i.e., $\Jdom$ contains precisely one element of each group orbit.
The arguments below are independent of which fundamental domain one chooses. 
We now define the set of \emph{symmetry reduced indices} as
 \[
   \Jred \bydef  \Jdom \cap \Jsym,
 \]
 and the space of \emph{symmetry reduced variables} as
 \[
    \Xred \bydef \bigl\{ \phi=(\b_\j)_{\j \in \Jred} \,:\, \b_\j \in \C \,,\, \left\Vert \phi \right\Vert_{\Xred}  < \infty \bigr\},
	\quad\text{with}\quad
	\left\Vert \phi \right\Vert_{\Xred}  \bydef 
	\sum_{\j \in \Jred}   |G \acts \j| \, \weight_\j  |\b_\j| \, .
 \]
In slight abuse of notation we will also interpret $\e_\j$ with $\j \in \Jred$
as elements of~$\Xred$. We may then interpret $\Xred$ as a subspace of $\X$ by writing $\b = \sum_{\j \in \Jred} \b_j \e_j$ for $\phi \in \Xred$.

The action of $\gamma_g$ on basis vectors is 
\[
  ( \gamma_g \e_\j )_{\j'} = \alpha_g(\j') \delta_{\j \beta_g(\j')}
    = \alpha_g(\j') \delta_{\beta_{g^{-1}}(\j) \j'} =
	\alpha_g(\j') (\e_{g^{-1} \acts \j})_{\j'} \, ,
\]
hence
\begin{equation}\label{e:actiononbasis}
  \gamma_g \e_\j = \alpha_g (g^{-1} \acts \j) \e_{g^{-1} \acts \j} \, \qquad\text{for all }\j \in\J.
\end{equation}

Before we specify the dependency of the coefficients $\{\c_\j\}_{\j \in  \Jsym}$ on the symmetry reduced variables $\{\c_\j\}_{\j \in \Jred}$ for $\c \in \Xsym$,
we derive some additional properties of $\alpha_g(\j)$ for $\j \in \Jsym$.
\begin{lemma}{\cite[Lemma 3.13]{JBJFsymmetry}}\label{l:talpha} 
Let $g_1,g_2 \in G$ and $\j \in \Jsym$.
If $g_1 \acts \j = g_2 \acts \j$ then \mbox{$\alpha_{g_1}(\j)=\alpha_{g_2}(\j)$}.
\end{lemma}
\begin{proof}
	Since $g_2^{-1} g_1 \acts \j = \j$ and $\j \in \Jsym$
	we have $\alpha_{g_2^{-1}g_1} (\j) =1$ by Lemma~\ref{l:alphadichotomy}(a).
	An application of Lemma~\ref{l:alpha}(b) concludes the proof.
\end{proof}

\begin{definition}\label{def:tgtalpha}
Let $\j$ by any element of $\Jsym$ and $\j'$ any element in its orbit 
$ G \acts \j$. 
We can choose an $\tg = \tg(\j,\j') \in G$ such that $\tg \acts \j=\j'$.
For such $\j$ and $\j'$ we define 
\begin{equation}\label{e:deftalpha}
  \talpha(\j,\j') \bydef \alpha^{-1}_{\tg(\j,\j')}(\j) 
  \qquad\text{for } \j \in \Jsym \text{ and } \j' \in G \acts \j.
\end{equation}
This is independent of the choice of $\tg$ by Lemma~\ref{l:talpha},
and 
\begin{equation}\label{e:talphaidentity}
  \alpha^{-1}_g(\j) = \talpha(\j,g \acts \j) 
  \qquad\text{for all } \j \in \Jsym \text{ and } g \in G.
\end{equation}
\end{definition}

This $\talpha$ will be used to define the natural map $\Sigma$ from $\Xred$ to $\Xsym$, defined below in~\eqref{e:sigma}. It also appears in the description of the action of $\sum_{g\in G} \gamma_g$ on unit elements.
\begin{lemma}\label{l:averageunit}
We have
\[
\sum_{g\in G} \gamma_g \e_j = \begin{cases}
0 & \text{for } j \in \Jtriv, \\
|G_{j}| \sum_{j' \in G \acts j}  \talpha(j,j') \, \e_{j'} 
& \text{for } j \in \Jsym .
\end{cases} 
\]
\end{lemma}
\begin{proof}
It follows from~\eqref{e:actiononbasis} and Lemma~\ref{l:alpha}(a) that
\begin{equation}\label{e:identity1}
\sum_{g\in G} \gamma_g \e_j = 
\sum_{g \in G} \alpha_g(g^{-1} \acts \j) \, \e_{g^{-1} \acts \j}
=
   \sum_{g \in G} \alpha_{g^{-1}}(g \acts \j) \, \e_{g \acts \j} 
=
   \sum_{g \in G} \alpha^{-1}_{g}(\j) \, \e_{g \acts \j} .
\end{equation}
Based on the orbit-stabilizer theorem, we use the notation $\tg(\j,\j')$ from Definition~\ref{def:tgtalpha} to write the right-hand side of~\eqref{e:identity1} as
\begin{alignat}{1}
  \sum_{g \in G} \alpha^{-1}_{g}(\j) \, \e_{g \acts \j} 
  &=   \sum_{\hat{g} \in G_\j} \sum_{\j' \in G \acts \j}
   \alpha^{-1}_{\tg(\j,\j') \hat{g}}(\j) \, \e_{\tg(\j,\j') \hat{g} \acts \j} 
   \nonumber \\
   &=   \sum_{\hat{g} \in G_\j} \sum_{\j' \in G \acts \j}
    \alpha^{-1}_{\tg(\j,\j')}(\j) \alpha^{-1}_{\hat{g}}(j)\, \e_{\tg(\j,\j') \acts \j}  \nonumber \\
   &=   \sum_{\hat{g} \in G_\j}  \alpha^{-1}_{\hat{g}}(j) 
   \sum_{\j' \in G \acts \j}
    \alpha^{-1}_{\tg(\j,\j')}(\j) \, \e_{\j'} ,   \label{e:identity2}
\end{alignat}
where we have used~\eqref{e:alphaproduct} in the second equality.
By combining~\eqref{e:identity1} and~\eqref{e:identity2} with~\eqref{e:alphainverse} and~\eqref{e:deftalpha} the assertion follows.
\end{proof}

We define the projection $\Pi : \X \to \Xred$ by
\[
  \Pi \c = \sum_{\j \in \Jred} \c_\j \e_\j.
\]
Clearly $\Pi^2= \Pi$.
Next, the group average $\Av : \X \to \Xsym$
is
\[
  \Av \omega = \frac{1}{|G|} \sum_{g \in G} \gamma_g \omega.
\]
We note that $\Av^2=\Av$.
To relate elements of $\Xred$ to elements of $\Xsym$
it turns out that it is useful to introduce the rescaling $\RR : \X \to \X$
\[
  \RR \c = \sum_{\j \in \J} |G\acts j| \c_j \e_\j.
\]
Finally, using $\talpha$ introduced in Definition~\ref{def:tgtalpha},
we define the linear map $\Sigma : \Xred \to \X$ by
\begin{equation}\label{e:sigma}
  \Sigma \b  \bydef \sum_{\j \in \Jred} \b_\j  \sum_{\j' \in G \acts \j} \talpha(\j,\j') \e_{\j'}.
\end{equation}

\begin{lemma}\label{l:bijective}\mbox{ }\\
\hspace*{5mm}\textup{(a)} 
For $\b \in \Xred$ we have $\Sigma \b = \Av \RR \b$, hence $\Sigma \b  \in \Xsym$.\\[1mm]
\hspace*{5mm}\textup{(b)}
$\Pi \Sigma$ is the identity on $\Xred$.	 \\[1mm]
\hspace*{5mm}\textup{(c)}
$\Sigma \Pi$ is the identity on $\Xsym$. \\[1mm]
\hspace*{5mm}\textup{(d)}
$\Sigma$ is bijective as a map from $\Xred$ to $\Xsym$ with inverse $\Pi : \Xsym \to \Xred$.	
\end{lemma}
\begin{proof}
We start with part (a). Let $\b \in \Xred$. Then it follows from the definitions of $\Av$ and $\RR$,
as well as Lemma~\ref{l:averageunit} that
\begin{equation*}
	\Av \RR \b  =
	\sum_{\j \in \Jred} \b_\j \frac{1}{|G_\j|} \sum_{g \in G} \gamma_g \e_\j 
    =  \sum_{\j \in \Jred} \b_\j 
	    \sum_{\j' \in G \acts \j} \talpha(\j,\j') \, \e_{\j'}	
	=		\Sigma \b .
\end{equation*}

Part (b) is an immediate consequence of~\eqref{e:sigma} and $\talpha(\j,\j)=1$.
	
To prove part (c), let $\c \in \Xsym$ be arbitrary and set $\tc = \c -\Sigma \Pi \c$. Then from part (b) we get
\[ 
  \Pi \tc = (\text{Id}_{\Xred} - \Pi \Sigma) \Pi \c =0,
\]
hence $\tc_{\j} =0$ for all $\j \in \Jred$. 
Consider any fixed $\j' \in \Jsym$, then there exists a $\j \in \Jred \cap G \acts \j'$.
Let $\tg \in G$ be such that $\tg \acts \j =\j'$.
Since $\tc \in \Xsym$ we have
\[
  \alpha_{\tg}(\j) \tc_{\j'} = (\gamma_{\tg} \tc)_{\j} = \tc_{\j} =0.
\] 
Since $\alpha_{\tg}(\j) \in S^1$ this implies that $\tc_{\j'}=0$, and since $\j' \in \Jsym$ was arbitrary, we conclude that $\tc_{\j'}=0$ for all $\j' \in \Jsym$. 
It then follows from  Lemma~\ref{l:zeroindices} that
\[
  \tc = \sum_{\j\in\J} \tc_\j \e_\j = \sum_{\j \in \Jsym} \tc_\j \e_\j =0.
\]
Hence, since $\c \in \Xsym$ was arbitrary, $\Sigma \Pi$ is the identity on $\Xsym$.

Part (d) follows immediately from (b) and (c).
\end{proof}

We will need to determine $\talpha$, particularly for computing the derivative of $\F$, see~\eqref{e:Credjj}. The next remark explains how this is accomplished in the code.
\begin{remark}\label{r:coding}

In the code we compute on a finite set of indices
$
  \Jdag = \{ j=(n,m) : n \in \E^\dagger \}.
$
In particular we determine
\[
  \Scode = \sum_{g \in G} \gamma_g \sum_{j \in \Jdom \cap \Jdag} \e_j .
\]
Since each orbit contains precisely one element of $\Jdom$,
it follows from Lemma~\ref{l:averageunit} that
\[
  \Scode_{j'} = \begin{cases}
  0 & \text{for } j' \in \Jdag \setminus \Jsym \\
  |G_{j'}| \, \talpha(j,j') & \text{for } j' \in \Jdag \cap \Jsym, \text{ where } j \in \Jred \cap G \acts j' .\\
  \end{cases}
\]
Since $\talpha(j,j') \in S^1 \subset \C$, it is straightforward to infer the index set $\Jdag \cap \Jsym$ from $\Scode$. Furthermore, for $j' \in \Jdag \cap \Jsym$ the value of $\Scode_{j'} \in \C$ gives us direct access to both the order of the stabilizer~$|G_{j'}|$ and the value $\talpha(j,j')$ for the unique $j \in \Jred \cap G \acts j'$. 
\end{remark}

We finish this subsection by extending the complex conjugate symmetry $\gamma_*$ to the reduced space. We define a corresponding action $\tgamma_*$ on $\Xred$ by 
\[
   \tgamma_* \b \bydef \Pi  \gamma_* \Sigma \b 
   \qquad\text{for all } \b \in \Xred.
\]
Clearly $\gamma_*$ leaves $\Xsym$ invariant (see~\eqref{e:gammaggammastar}).
It then follows from Lemma~\ref{l:bijective}(d) that
\begin{equation}\label{e:sigmagammastar}
  \Sigma \tgamma_* = \gamma_* \Sigma    .
\end{equation}
As one may expect, $\tgamma_*$ is an involution. 
In particular, this is useful in order to symmetrize $\b$ by 
$\frac{1}{2}(\b + \tgamma_* \b)$, which is invariant under $\tgamma_*$.
\begin{lemma}
We have $\tgamma_*^2 = \text{\textup{Id}}_{\Xred}$.
\end{lemma}
\begin{proof}
Let $\b \in \Xred$. Since $\gamma_*$ leaves $\Xsym$ invariant it follows that
\begin{alignat*}{1}
  \tgamma^2_* \b = \Pi  \gamma_* \Sigma \Pi  \gamma_* \Sigma \b 
  = \Pi  \gamma^2_*  \Sigma \b  = \Pi  \Sigma \b = \b,
\end{alignat*}
where in the third equality we have used~\ref{l:bijective}(c),
and in the final equality we have used~\ref{l:bijective}(b).
\end{proof}

\subsection{Functional analytic setup in symmetry reduced variables}
\label{s:Fred}

We now introduce the framework (with symmetry-reduced variables), where we define the fixed point operator that is actually used in practice to validate the numerical solution.

First, we introduce the space
\[
  \XXred \bydef \C \times \Xred.
\]
Elements in $\XXred$ are denoted by
\[
  \varphi = (\Omega,\phi) = \left(\varphi_\phase,(\varphi_j)_{j\in\Jred} \right),
\]
and the operators $\Pi$ and $\Sigma$ naturally extend to $\XX$ and $\XXred$ respectively via
\[
  \Pi(\Omega,\omega)=(\Omega,\Pi\omega) \qquad\text{and}\qquad
  \Sigma(\Omega,\phi)=(\Omega,\Sigma\phi).
\]
Similarly, we extend the definition of $\tgamma_*$ to $\XXred$ by 
\begin{equation*}
\tgamma_* (\Omega,\b) = (\Omega^*,\tgamma_*\b).
\end{equation*}
For later use, we also define the space $\XXsym \bydef \C \times \Xsym$.

Introducing the space $\XXred_{-2,-1}$, which is obtained from $\XX_{-2,-1}$ the same way $\XXred$ is obtained from $\XX$, we define the map $\FFred:\XXred\to \XXred_{-2,-1}$ by
\begin{equation*}
\FFred=\begin{pmatrix}
\Fred_{\phase} \\
\left(\Fred_j\right)_{j\in\Jred}
\end{pmatrix},
\end{equation*}
where
\begin{equation}\label{e:defFF}
   \Fred(\Omega, \b) \bydef  \Pi  F(\Omega,\Sigma \b ),
\end{equation}
and the phase condition in the symmetrized setting is given by
\begin{equation}
\label{e:phasesymmetry}
\Fred_{\phase}(\phi) \bydef i \sum_{j} |G \acts j| n_4 \phi_j (\hat{\phi}_j)^*,
\end{equation}
where $j=(\tilde{n},n_4,m)$, and $\hat\phi$ is some fixed element of $\Xred$. Notice that, if we choose $\hat{\omega} = \Sigma \hat{\phi}$ in~\eqref{eq:phase_condition}, then $\Fred_{\phase}(\phi) = F_{\phase} (\Sigma \phi)$ and therefore
\begin{equation}
\label{eq:FredVSF}
\FFred(\varphi) = \Pi \F (\Sigma \varphi) .
\end{equation}

Recalling \eqref{eq:weights_over_J}, we also introduce the weights
\begin{equation}\label{e:defxis}
   \symweight_\j \bydef  \weight_\j \, |G \acts \j| .
\end{equation}
and (see Remark~\ref{r:unitweight})
\[
  \symweight_\phase = 1,
\]
together with the norm
\[
\rednorm{\varphi} \bydef \symweight_\phase |\Omega| + \left\Vert \phi \right\Vert_{\Xred} = \sum_{j\in \JJred} \symweight_j |\varphi_j|,
\]
with index set
\[ 
  \JJred \bydef \phase \cup \Jred.
\]
We recall that $|G \acts \j | = |G| / |G_\j|$ by Remark~\ref{r:orbitstabilizer}.
This observation is combined with Remark~\ref{r:Zcode} to determine the weights  $\symweight_\j$ in the code.
The next lemma shows that the norm $\rednorm{\cdot}$ is compatible with the symmetrization, as well as complex conjugation.
\begin{lemma}\label{l:symmetrycompatible}\mbox{ }\\
\hspace*{5mm}\textup{(a)} 
For all $\varphi \in \XXred$ we have 
$ \rednorm{\varphi} = \norm{\Sigma\varphi} $.\\[1mm]
\hspace*{5mm}\textup{(b)} 
For all $\varphi \in \XXred$ we have 
$ \rednorm{\tgamma_* \varphi} = \rednorm{\varphi} $.
\end{lemma}
\begin{proof}
We first observe that $\weight_{\j}$ is invariant under $G$. Namely, when we denote $(n',m')= g \acts (n,m)$, then it follows from~\eqref{e:explicitbeta}
 that $|n'|_1 = |n|_1$ for any $g \in G$, since $C_g$ is a signed permutation matrix. 
 
From the definition~\eqref{e:sigma}, we then obtain, 
\begin{alignat*}{1}
\norm{\Sigma\varphi}  &= \vert \Omega\vert +
    \sum_{\j \in \Jred} \sum_{\j' \in G \acts \j} \weight_{\j'} |\b_\j \, \talpha(\j,\j')| \\
	&= \vert \Omega\vert + \sum_{\j \in \Jred} |\b_\j| \, \weight_{\j} \sum_{\j' \in G \acts \j} | \talpha(\j,\j')| \\
	 &= \vert \Omega\vert + \sum_{\j \in \Jred} |\b_\j| \, \weight_{\j}  \, |G \acts \j| \\
	 &= \rednorm{\varphi} \, ,
\end{alignat*}
since $\talpha \in S^1$. This proves part (a).

To prove part (b), we first note that, for all $W\in\XX$, $\norm{\gamma_* W} = \norm{W}$
by Definition~\ref{def:complex_conj} and the observation $\weight_{\j}$ is invariant under $\gamma_*$, in the sense that $|{-n}|_1=|n|_1$. It then follows from part (a) and~\eqref{e:sigmagammastar} that 
\[
  \rednorm{\tgamma_* \varphi} = \norm{\Sigma \tgamma_* \varphi} = \norm{\gamma_* \Sigma \varphi} = \norm{\Sigma \varphi} = \rednorm{\varphi}. \qedhere
\]
\end{proof}

To solve the zero finding problem $\FFred=0$ on $\XXred$,
we analyse a fixed point operator as in Section~\ref{s:radpol}.
The role of $\XX$ is taken over by $\XXred$, 
and the norm $\| \cdot\|_\XX$ is replaced by the weighted norm $\| \cdot\|_{\XXred}$,
which is symmetry compatible in the sense of Lemma~\ref{l:symmetrycompatible}.

Notice that $\FFred$ inherits the equivariance property of $\F$ with respect to the complex conjugacy $\gamma_*$.
\begin{lemma}
\label{lem:Fredequivariance}
Assume that $\hat{\omega}$ used in~\eqref{eq:phase_condition} is such that, $\hat{\omega}= \Sigma \hat{\phi}$ for some $\hat\phi\in \Xred$ such that $\tgamma_*\hat\phi=\hat\phi$. Then, for all $\varphi\in \XXred$, 
\begin{equation*}
\FFred ( \tgamma_* \varphi) = \tgamma_* \FFred (\varphi).
\end{equation*}
\end{lemma}
\begin{proof}
Using successively~\eqref{eq:FredVSF},~\eqref{e:sigmagammastar} and Lemma~\ref{lem:FequivariantS}, we obtain
\begin{equation*}
\FFred ( \tgamma_* \varphi) = \Pi \F(\Sigma \tgamma_* \varphi) = \Pi \F( \gamma_* \Sigma \varphi) = \Pi \gamma_*\F( \Sigma \varphi) = \tgamma_* \FFred (\varphi). \qedhere
\end{equation*}
\end{proof}

Having chosen $N^\dagger$ and $\tilde N$,
we define the index sets
\begin{alignat*}{1}
  \E^\dagger_{\text{red}} (N^\dagger) & \bydef \{ j=(n,m) \in \Jred : n \in \E^\dagger (N^\dagger) \} ,\\
  \tildeset^{\text{sol}}_{\text{red}}(\tilde N) & \bydef \{ j=(n,m) \in \Jred : n \in \tildeset^{\text{sol}} (\tilde N) \} .
\end{alignat*}
The size of the Galerkin projection is the number of elements 
in $ \E^\dagger_{\text{red}} (N^\dagger) $,
which is substantially smaller than the number of elements in $\E^\dagger (N^\dagger)$, since we restrict to symmetry reduced variables. Indeed, the number of independent variables is reduced by roughly a factor $|G|$. 

The construction of linear operators $\Ared$ and $\Areddag$
is completely analogous to Section~\ref{s:AAdag}. To incorporate the incompressibility, we introduce
\begin{equation*}
\XXred_\div \bydef \Pi \left\{W\in \XXsym : \sum_{m=1}^{3} D_m\omega^{(m)} = 0 \right\}.
\end{equation*}
\begin{theorem}
\label{th:radii_pol_sym}
Let $\eta>1$.
Assume there exist non-negative constants $Y_0^\ered$, $Z_0^\ered$, $Z_1^\ered$ and~$Z_2^\ered$ such that
\begin{align}
\rednorm{\Ared\FFred(\bb)} &\leq Y_0^\ered \label{def:Y_0_sym}\\ 
\Vert I-\Ared\Areddag \Vert_{B(\XXred,\XXred)} &\leq Z_0^\ered \label{def:Z_0_sym}\\
\Vert \Ared\left(D\FFred(\bb)-\Areddag\right) \Vert_{B(\XXred,\XXred)} &\leq Z_1^\ered \label{def:Z_1_sym}\\
\Vert \Ared(D\FFred(\varphi)-D\FFred(\bb)) \Vert_{B(\XXred,\XXred)} &\leq Z_2^\ered \rednorm{\varphi-\bb}, \quad \text{for all }\varphi\in\XXred. \label{def:Z_2_sym}
\end{align}
Assume also that
\begin{itemize}
\item the forcing term $f$ is time independent, $G$-invariant, and has average zero,
\item $\bb = (\bO,\bar\phi)$ is in $\in\XXred_\div$,
\item $\tgamma_* \bar\varphi=\bar\varphi$ and $\hat\phi$ (used to define the phase condition~\eqref{e:phasesymmetry}) is such that $\tgamma_* \hat\phi=\hat\phi$.
\end{itemize}
If 
\begin{equation}
\label{hyp:cond_pol_sym}
Z_0^\ered+Z_1^\ered<1 \qquad \text{and}\qquad 2Y_0^\ered Z_2^\ered<\left(1-(Z_0^\ered+Z_1^\ered)\right)^2,
\end{equation}
then, for all $r\in[r_{\min},r_{\max})$ there exists a unique $\tilde\varphi=(\tilde\Omega,\tilde\phi)\in \B_{\XXred}(\bb,r)$ such that $\FFred(\tilde\varphi)=0$, where $\B_{\XXred}(\bb,r)$ is the closed ball in ${\XXred}$, centered at $\bb$ and of radius $r$, and
\begin{alignat}{1}
r_{\min} &\bydef \frac{1-(Z_0^\ered+Z_1^\ered)-\sqrt{\left(1-(Z_0^\ered+Z_1^\ered)\right)^2-2Y_0^\ered Z_2^\ered}}{Z_2^\ered}, \label{e:rmin}\\
r_{\max}&\bydef\frac{1-(Z_0^\ered+Z_1^\ered)}{Z_2^\ered}  \label{e:rmax}.
\end{alignat}
Besides, this unique $\tilde\varphi$ also lies in $\in\XXred_{\div}$. Finally, defining $u=M\Sigma\tilde\phi$, there exists a pressure function $p$ such that $(u,p)$ is a $\frac{2\pi}{\tilde\Omega}$-periodic, real valued, analytic and $G$-invariant solution of Navier-Stokes equations~\eqref{eq:NS}.
\end{theorem}
\begin{proof}
The structure of the proof is the same as the one of Theorem~\ref{th:radii_pol}. Therefore, we only outline the main steps. First, we consider the operator
\[
  \Tred \bydef I - D \FFred(\bb) \FFred ,
\]
and use~\eqref{hyp:cond_pol_sym} to infer that it is a contraction on $\B_{\XXred}(\bb,r)$ for all $r\in[r_{\min},r_{\max})$, yielding the existence of a unique zero $\tilde\varphi$ of $\FFred$ in $\B_{\XXred}(\bb,r)$. 
Next, we use that $\bb\in\XXred_\div$ and that $\Tred\left(\XXred_\div\right)\subset \XXred_\div$, which follows from arguments entirely analogous to the ones in the proof of Theorem~\ref{th:radii_pol}, to infer that $\tilde\varphi$ also belongs to $\XXred_\div$. Then, we use Lemmas~\ref{l:symmetrycompatible} and~\ref{lem:Fredequivariance}, once again as in the proof of Theorem~\ref{th:radii_pol}, to conclude that $\tgamma_*\tilde\varphi=\tilde\varphi$.

Defining $\tilde W=(\tilde \Omega,\tilde \omega)=\Sigma\tilde\varphi$, we then have $\F(\tilde W)=0$ (with $\hat\omega=\Sigma\hat\phi$ in~\eqref{eq:phase_condition}), $\tilde\omega\in\XX_\div$ and $\gamma_*\tilde W=\tilde W$. By Lemma~\ref{lem:eq_NS_vorticity}, there exists $p$ such that $(u,p)$ is a $2\pi/\tilde\Omega$-periodic, real valued, analytic solution of Navier-Stokes equations~\eqref{eq:NS}.

Finally, by construction $\tilde\omega\in\Xsym$, and by Lemma~\ref{l:ctou} $\tilde a_g\tilde\omega=\tilde\omega$ for all $g\in G$. From~\eqref{eq:agM} we infer that $u=M\tilde\omega$ satisfies $a_g u=u$, and that $p=\Gamma \Phi(u)$ satisfies $a_g p=p$ (see Remark~\ref{rem:invariancep}), i.e., $(u,p)$ is $G$-invariant.
\end{proof}

\begin{remark}\label{r:makesymmetric}
To make sure that $\bb \in \XXred_\div$ in practice we replace a candidate $\bar\phi$ by $\Pi \Av \Pi_{\div} \Sigma \bar\phi$,
where $\Pi_{\div}$ is a projection of $\X$ onto 
\[
  \X_\div \bydef
  \left\{ \omega \in X : \sum_{m=1}^{3} D_m\omega^{(m)} = 0 \right\},
\] 
which is $G$-invariant and hence invariant under the averaging operator $\Av$.

Similarly, to make sure that $\tgamma_* \bb=\bb$  we make sure 
$\bar{\Omega}$ is real-valued and we 
replace a candidate $\bar\phi$ by $\frac{1}{2}(\bar\phi + \Pi \gamma_* \Sigma \bar\phi)$, and analogously $\tgamma_* \hat\phi=\hat\phi$. Since $\XXred_\div$ is invariant under $\tgamma_*$, this new $(\bar{\Omega},\frac{1}{2}(\bar\phi + \Pi \gamma_* \Sigma \bar\phi))$ is still in $\XXred_\div$ if $\bar\varphi$ is.

The resulting $\bb$ and $\hat{\phi}$ may then be represented by intervals in the code, but that is no impediment whatsoever.
\end{remark}

In section~\ref{s:symmetrybounds} we explain how explicit expressions for the bounds $Y_0^\ered$, $Z_0^\ered$, $Z_1^\ered$ and $Z_2^\ered$ can be obtained from the bounds $Y_0$, $Z_0$, $Z_1$ and $Z_2$ of Section~\ref{s:estimates}.

\subsection{Bounds in the symmetric setting}
\label{s:symmetrybounds}

As in Section~\ref{s:estimates}, we 
assume throughout that the approximate solution $\bb=(\bar\Omega,\bar\phi)$ only has a finite number of non-zero modes, i.e., $\bar\phi_n=0$ for all $n\notin \set^\sol$ for some finite set $\set^\sol$.
Similarly, $\hat{\phi}_n=0$ for all $n\notin \set^\sol$. 

\subsubsection{bound \boldmath$Y_0^\ered$\unboldmath}
\label{s:symboundY}

There are essentially no changes compared to Section~\ref{s:Y0bound} in the computation of the bound on the residue, except that we need to take into account the symmetry respecting norm:
\begin{equation} \label{eq:Y0red}
   Y_0^\ered \bydef  \symweight_\phase \bigl|[\Ared \FFred(\bb)]_\phase \bigr| 
   + \sum_{\substack{ \j=(n,m) \in \Jred \\ n \in  \set^{\sol}+\set^{\sol} }} \symweight_j \, \bigl|[\Ared \FFred(\bb)]_\j\bigr| .
\end{equation}

\subsubsection{bound \boldmath$Z_0^\ered$\unboldmath}
\label{s:symboundZ0}

The bound $Z_0^\ered$ is completely analogous to the one in Section~\ref{s:Z0bound}.
We just need to use the operator norm
\begin{equation} \label{eq:Z0red}
Z_0^\ered \bydef  \|\Bred\|_{B(\XXred,\XXred)} = \sup_{j \in \JJred} \frac{1}{\symweight_j}
  \rednorm{B^{(.,j)}},
\end{equation}
applied to $\Bred=I-\Ared \Areddag$, which again has finitely many nonzero components only.

\subsubsection{bound \boldmath$Z_1^\ered$\unboldmath}
\label{s:symboundZ1}



We introduce 
\[
  \Cred \bydef D\FFred(\bb) - \Areddag.
\]
The finite part of the $Z_1^\ered$ estimate is analogous to section~\ref{s:Z1bound}:
\begin{equation} \label{eq:Z1red_finite}
  \left(Z_1^\ered\right)^{\textup{finite}} = \max_{j \in \phase \cup  \tildeset^{\text{sol}}_{\text{red}}(\tilde N) } 
  \frac{1}{\symweight_j} \rednorm{\Ared \Cred_{\cdot, j}} .
\end{equation}
For the tail part, we will take advantage of the estimates in section~\ref{s:Z1bound}.
We need to estimate
\[
  \sup_{j \in \Jred \setminus \set^{\text{sol}}_{\text{red}}(\tilde N) } 
  \frac{1}{\symweight_j} \rednorm{\Ared \Cred_{\cdot, j}} .
\]
We introduce, cf.~\eqref{e:defFF},
\begin{equation}\label{e:psired}
  \Psired(\b )  \bydef \Pi \Psi (\Sigma \b )  ,
\end{equation}
where $\Psi$ is defined in~\eqref{e:defPsi}.
Furthermore, we set $\lambda_j \bydef \lambda_n$ for $j=(n,m)$, see~\eqref{e:deflambda}. It follows from the definition of $\FFred$ and $\Areddag$ that
\begin{equation*}
	\Cred_{j', j}= \begin{cases}
	0 & \text{for } j,j' \in \phase \cup \E^\dagger_{\textup{red}}, \\
	(D\Psired(\bar\phi))_{j',j} & \text{for } 
	j \in \Jred \setminus \E^\dagger_{\textup{red}}
	\text{ or } 
	j' \in \Jred \setminus \E^\dagger_{\textup{red}}.
	\end{cases}
\end{equation*}
In particular, analogous to~\eqref{e:vanishingClmkn},
\begin{equation*}
\Cred_{j', j}= 0\qquad \text{for all } j \in \Jred \setminus \tildeset^{\text{sol}}_{\text{red}}(\tilde N)
\text{ and } j'\in \phase \cup \E^\dagger_{\textup{red}}.
\end{equation*}
Therefore, the tail estimate reduces to
\begin{equation}\label{e:tailAC}
  \sup_{j \in \Jred \setminus \set^{\text{sol}}_{\text{red}}(\tilde N) } 
  \frac{1}{\symweight_j} \rednorm{\Ared \Cred_{\cdot, j}} = \sup_{j \in \Jred \setminus \tildeset^{\text{sol}}_{\text{red}}(\tilde N) } 
  \frac{1}{\symweight_j} \sum_{j'\in \Jred} \vert \lambda_{j'}\vert \vert \Cred_{j', j} \vert \symweight_{\j'} .
\end{equation}
It follows from~\eqref{e:psired}
and the formula~\eqref{e:sigma} for the symmetrization $\Sigma$ that 
\begin{equation}\label{e:Credjj}
  \Cred_{j', j} 
  =  D\Psi_{\j'}(\Sigma\bar\phi)\Sigma e_{\j} = \sum_{\j'' \in G \acts \j} \talpha(\j,\j'')  D \Psi_{\j'}(\Sigma \bar\phi ) \e_{\j''} \,  , 
	\qquad \text{for } \j,\j' \in \Jred .
\end{equation}
For any  $j \in \Jred \setminus \tildeset^{\text{sol}}_{\text{red}}(\tilde N) $ 
we then use the triangle inequality to estimate
\begin{alignat*}{1}
  \frac{1}{\symweight_j} \sum_{j'\in \Jred} \vert \lambda_{j'}\vert \vert \Cred_{j', j} \vert \symweight_{\j'}
  & \leq
  \frac{1}{\symweight_j} \sum_{j'\in \Jred} \vert \lambda_{j'} \vert \symweight_{\j'} 
   \sum_{\j'' \in G \acts \j} \vert 
  D \Psi_{\j'}(\Sigma\bar\phi ) \e_{\j''}  \vert .
\end{alignat*}
Writing $\mathcal{C}_{j ,j'} \bydef  D \Psi_{\j}(\Sigma\bar\phi ) e_{\j'}$ for convenience,
we use Remark~\ref{r:orbitstabilizer}, Equation~\eqref{e:defxis} and Lemma~\ref{l:equivariantindices} below to obtain
\begin{alignat}{1}
\frac{1}{\symweight_{j}}\sum_{j'\in \Jred} \vert \lambda_{j'}\vert \symweight_{j'} \sum_{j''\in G.j} \vert \mathcal{C}_{j',j''} \vert & = \frac{1}{\symweight_{j}}\sum_{j'\in \Jred} \vert \lambda_{j'}\vert \symweight_{j'} \frac{1}{\vert G_{j}\vert }\sum_{g\in G} \vert \mathcal{C}_{j',g.j} \vert \nonumber \\
& = \frac{1}{\symweight_{j}\vert G_{j}\vert}\sum_{j'\in \Jred} \vert \lambda_{j'}\vert \symweight_{j'} \sum_{g\in G} \vert \mathcal{C}_{g.j',j} \vert \nonumber \\
& = \frac{1}{\weight_{j}\vert G\vert}\sum_{j'\in \Jred} \vert \lambda_{j'}\vert \symweight_{j'} \vert G_{j'}\vert \sum_{j''\in G.j'} \vert \mathcal{C}_{g.j',j} \vert \nonumber \\
& = \frac{1}{\weight_{j}}\sum_{j'\in J} \vert \lambda_{j'}\vert \vert \mathcal{C}_{j',j} \vert \weight_{j'} . \label{e:expressionwithoutsymmetry}
\end{alignat}
The right-hand side of~\eqref{e:expressionwithoutsymmetry} is exactly the one estimated in Section~\ref{s:Z1bound} with $\bw = \Sigma \bar\phi$.
Hence~\eqref{e:tailAC} is bounded by 
\[
\left(Z_1^\ered\right)^{\textup{tail}}  \bydef \max_{m=1,2,3} \left(Z_1^{\textup{tail}}\right)^{(m)}
\]
where the elements in the righthand side are given in~\eqref{e:Z1tail} with $\bar\omega= \Sigma \bar\phi$.

\begin{lemma}\label{l:equivariantindices}
	Let $\omega \in \Xsym$.
Writing $D\Psi(\omega)=\left(\mathcal{C}_{j',j}\right)_{j',j\in J}$, we have
\begin{equation*}
\left\vert \mathcal{C}_{j',j} \right\vert = \left\vert \mathcal{C}_{g^{-1}.j',g^{-1}.j} \right\vert \qquad \text{for all } j',j\in J
\text{ and } g \in G.
\end{equation*}
\end{lemma}
\begin{proof}
Analogous to Lemma~\ref{lem:equivariance_F},
we observe that $\Psi$ is a $G$-equivariant, which implies that $ \Psi (\gamma_g \omega) = \gamma_g  \Psi (\omega)  $ for all $\omega \in \X$ and $g\in G$.
For $\omega \in \Xsym$ we then obtain 
\begin{equation*}
\gamma_{g^{-1}}D\Psi(\omega)\gamma_g=D\Psi(\omega)
\qquad \text{for all } g \in G.
\end{equation*}
Using~\eqref{e:actiononbasis} twice, we get for all $j',j\in J$, 
\begin{align*}
\mathcal{C}_{j',j} &= \left(D\Psi(\omega)\e_{j}\right)_{j'} \\
&= \left(\gamma_{g^{-1}}D\Psi(\omega)\gamma_g \e_{j}\right)_{j'} \\
&= \alpha_g(g^{-1}.j)\left(\gamma_{g^{-1}}D\Psi(\omega)\e_{g^{-1}.j}\right)_{j'} \\
&= \alpha_g(g^{-1}.j)\left(\gamma_{g^{-1}}\sum_{i\in J} \mathcal{C}_{i,g^{-1}.j}\e_i\right)_{j'} \\
&= \alpha_g(g^{-1}.j)\left(\sum_{i\in J} \mathcal{C}_{i,g^{-1}.j}\alpha_{g^{-1}}(g.i)\e_{g.i}\right)_{j'} \\
&= \alpha_g(g^{-1}.j)\alpha_{g^{-1}}(j') \mathcal{C}_{g^{-1}.j',g^{-1}.j},
\end{align*}
which yields the assertion, since $\vert \alpha_g(j')\vert =1$ for all $g\in G$ and all $j'\in J$.
\end{proof}

Finally, we set
\begin{equation} \label{eq:Z1red}
  Z_1^\ered \bydef \max \left\{ \left(Z_1^\ered\right)^{\textup{finite}}  , \left(Z_1^\ered\right)^{\textup{tail}}  \right\} .
\end{equation}

\subsubsection{bound \boldmath$Z_2^\ered$\unboldmath}
\label{s:symboundZ2}

For any $\varphi, \varphi' \in \XXred$ with $\rednorm{\varphi}, \rednorm{\varphi'} \leq 1$
we need to estimate
\begin{equation}
  \rednorm{\Ared D^2 \FFred (\bb)(\varphi,\varphi')} 
  =  \rednorm{\Ared \Pi D^2 \F (\bW)(W,W')} ,
\end{equation}
where $\bW=\Sigma\bb$, $W=\Sigma \varphi$, $W' = \Sigma \varphi' $, and by Lemma~\ref{l:symmetrycompatible} $\|W\|_{\XX},\|W'\|_{\XX} \leq 1$. 

As in Section~\ref{s:Z2bound}, we start by splitting 
\begin{equation*}
\rednorm{\Ared \Pi D^2 \F (\bW)(W,W')} \leq \left\Vert \Ared \right\Vert_{B(\XXred_{-1,-1},\XXred)} \left\Vert \Pi D^2 \F (\bW)(W,W') \right\Vert_{\XXred_{-1,-1}}.
\end{equation*}
But since $W,W' \in\XXsym$ and $\F$ is $G$-equivariant, we have that $D^2\F (\bW)(W,W')$ is $G$-invariant, i.e. belongs to $\XXsym$, and thus by Lemma~\ref{l:symmetrycompatible}
\begin{equation*}
\left\Vert \Pi D^2 \F (\bW)(W,W') \right\Vert_{\XXred_{-1,-1}} = \left\Vert D^2 \F (\bW)(W,W') \right\Vert_{\XX_{-1,-1}}.
\end{equation*}
Therefore, we can directly use the estimates of Section~\ref{s:Z2bound} to obtain
\begin{equation} \label{eq:Z2red}
  Z_2^\ered \bydef (4+\sqrt 2)\NNred{A} ,
\end{equation}
where
\begin{equation*}
\NNred{A} \bydef \max \left\{ \max_{j \in \E^\dagger_{\text{red}}} \frac{1}{\tilde \xi^s_j} \rednorm{(\Ared)_{\cdot,j}}  : 
    \frac{1}{\bO},\frac{1}{\sqrt{\nu N^\dagger}}\right\} ,
\end{equation*}
with 
\begin{equation*}
\tilde \xi^s_j = \frac{\xi^s_j}{\vert n\vert_\infty},\quad j=(n,m)\in \E^\dagger_{\text{red}}.
\end{equation*}

\section{Application to Taylor-Green flow}
\label{s:application}

In this section, we present the results obtained with our computer-assisted approach for the Navier-Stokes equations~\eqref{eq:NS} with forcing~\eqref{eq:TG_intro}.
%
%
In Section~\ref{sec:applic_background}, we set the stage by recalling some know analytic results in this context, complemented by numerical simulations. In Section~\ref{s:applic_2D}, we then show how the estimates of Section~\ref{s:symmetrybounds}, combined with Theorem~\ref{th:radii_pol_sym}, can be used to obtain rigorous existence results and error bounds about periodic solutions.

\subsection{Analytical and numerical background}
\label{sec:applic_background}

We consider the forced Navier-Stokes equations for an incompressible, homogeneous fluid
\[
\left\{ 
\begin{array}{l}
  \partial_t u + (u\cdot \nabla)u - \nu_0\Delta u + \dfrac{1}{\rho_0}\nabla P  = f_0 ,  \\[1mm]
  \nabla \cdot u = 0,
  \end{array} \right.
\]
on a cubical domain of dimension $L$ with periodic boundary conditions.
Here $\rho_0$ is the density of the fluid and $\nu_0$ is its viscosity,
while the forcing is chosen to be of the simple form
\[
f_0 (x) = 
\frac{\gamma_0}{2} \begin{pmatrix} \hfill \sin\frac{2\pi x_1}{L} \, \cos\frac{2\pi x_2}{L} \\[1mm] 
-\cos \frac{2\pi x_1}{L} \,\sin\frac{2 \pi x_2}{L} \\[1mm] 0\end{pmatrix},
\]
where  $\gamma_0$ parametrizes its intensity.
This corresponds to the planar case of a family of flows introduced by Taylor and Green \cite{Taylor1937} to study the interaction of motions on different spatial scales. Non-dimensionalization (without introducing new notation for the new variables)
leads to
\begin{equation}\label{eq:NSrepeated}
\left\{ 
\begin{array}{l}
  \partial_t u + (u\cdot \nabla)u - \nu \Delta u + \nabla p = f ,  \\[1mm]
  \nabla \cdot u = 0,
  \end{array} \right.
\end{equation}
on a cube of dimension $2 \pi$, where 
$\nu=\sqrt{\frac{32\pi^3}{\gamma_0 L^3}} \, \nu_0$ is a dimensionless parameter and the dimensionless forcing is
\begin{equation}\label{TG}
f=f(x)=\begin{pmatrix} \hfill 2 \sin x_1 \cos x_2  \\ -2\cos  x_1 \sin x_2 \\0\end{pmatrix}.
\end{equation}
To be able to compare with the literature
we use the geometric Reynolds number $Re=\frac{\sqrt{8\pi}}{\nu}$ in the discussion of the bifurcation diagram below, cf.~Figure~\ref{bif_diagram}.

The Navier-Stokes equations~\eqref{eq:NSrepeated} under the forcing~\eqref{TG} admit an equilibrium solution for which we have the analytic expression
\begin{alignat}{2}\label{visc_eq}
u^{*}(x)&=\frac{1}{2\nu} f(x) &\qquad
 p^{*}(x)&=\frac{1}{4\nu^2}\left(\cos  2x_1 +\cos 2x_2 \right).
\end{alignat}
We will refer to this solution as the {\em viscous equilibrium}. 
\begin{figure}
\begin{center}
\hspace{-1cm}
\includegraphics[width=.6\textwidth]{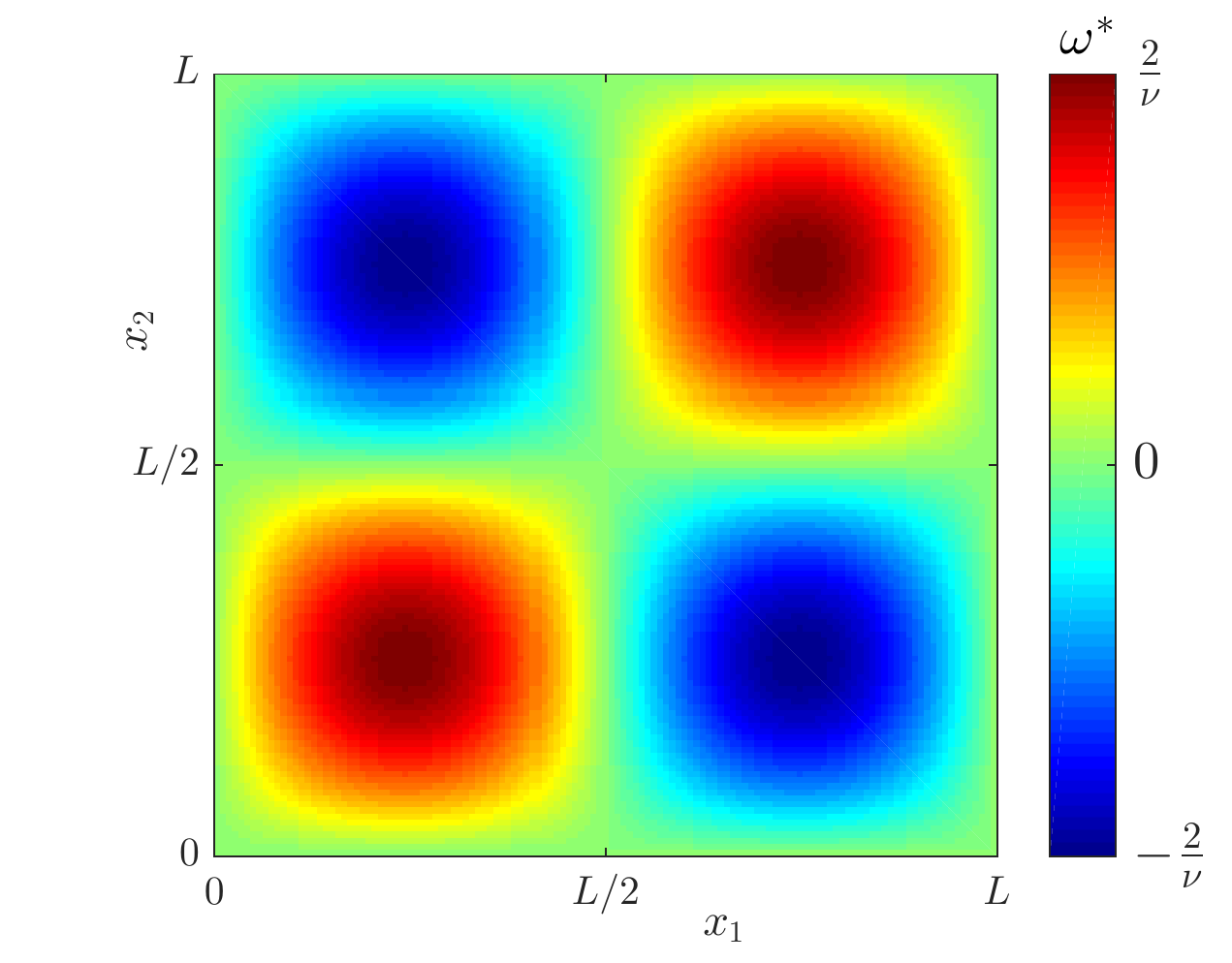}
\end{center}
\vspace{-.6cm}
\caption{Density plot the vertical vorticity of the viscous equilibrium state, $\omega^{*}=f^{\omega}/(2\nu)$. The maxima (red) correspond to vortices with a counter clockwise rotation and the minima (blue) to vortices with a clockwise rotation.}
\label{VE}
\end{figure}
The associated forcing for the vorticity equation is given by
\begin{equation}
\label{eq:rhs_visc}
\rhs(x) = \begin{pmatrix} 0 \\ 0 \\ 4\sin x_1  \sin x_2 \end{pmatrix}.
\end{equation}
The viscous equilibrium  consists of four counter-rotating vortices. Its vertical vorticity 
is shown in a plane of constant height in Figure~\ref{VE}. It is straightforward to verify that this solution is invariant under the following symmetry operations:
\begin{itemize}
\item Translation over any distance $d$ in the vertical direction, $T_{d}$.
\item Reflection in the $x_1$-direction, $S_{x_1}$.
\item Reflection in the $x_2$-direction, $S_{x_2}$.
\item Reflection in the $x_3$-direction, $S_{x_3}$.
\item Rotation about the axis $x_1=x_2=0$ over $\pi/2$ followed by a shift over $L/2$ in the $x_1$-direction, $R$.
\item A shift over $L/2$ in both the $x_1$ and $x_2$ directions, $D$.
\end{itemize}
These operations generate the group of spatial symmetries of the Navier-Stokes equation with planar Taylor-Green forcing. In addition, the system is equivariant under translations in time.
One may reduce this continuous symmetry to a discrete subgroup
by considering, for a periodic orbit of period $\tau$,  the shift by $\frac{\tau}{k}$ in time, $P_k$, for some $k \in \N$ ($k=4$ for the solutions studied below). 

The linear instability of the viscous equilibrium at high Reynolds number (that is at low viscosity) has been investigated at length in the literature, for instance by Sipp and Jacquin~\cite{SippJacquin}. In those studies, the emphasis is on the rapid formation of small-scale structures. In the current context, we are interested in the first instabilities that occur when increasing the Reynolds number from zero. Due to a classical result by Serrin \cite{Serrin1959a} the viscous equilibrium is guaranteed to be the unique limit state of the flow for any viscosity greater than
\[
\nu_{\rm s}=\sqrt[4]{\frac{8}{3+\sqrt{13}}}\approx 1.049 \qquad \left(Re_{\rm s} \approx 4.78\right).
\]

Below that value, linear instabilities occur, giving rise to branches of solutions for which some of the symmetries are broken. Such bifurcating branches can be numerically approximated by standard methods \cite{Sanchez2016}. Using these methods, we found that the first instability is a Hopf bifurcation that gives rise to a branch of two-dimensional periodic solutions. Subsequently, this branch appears to turn unstable at a point where at least one family of three-dimensional periodic solutions branches off. A partial bifurcation diagram is shown in Figure~\ref{bif_diagram}. 
In order to differentiate between the solutions, we compute the
deviation from reflection symmetry by computing the maximum of $\Enorm{u-S_{x_1}u}$, where 
\begin{equation*}
\Enorm{u} \bydef \frac{1}{2} \int u\!\cdot\!u \,dx
\end{equation*}
is the energy. We normalize this measure by the maximum it can attain, and display
\begin{equation}\label{e:symmetrybreakingmeasure}
  \max_{0 \leq t \leq \tau} \max_{0 \leq x_1 \leq 2\pi} \frac{\Enorm{u(t)-S_{x_1}u(t)}}{4\Enorm{u(t)}},
\end{equation}
on the vertical axis of the bifurcation diagram. 
\begin{figure}[t!]
\begin{center}
\includegraphics[width=.7\textwidth]{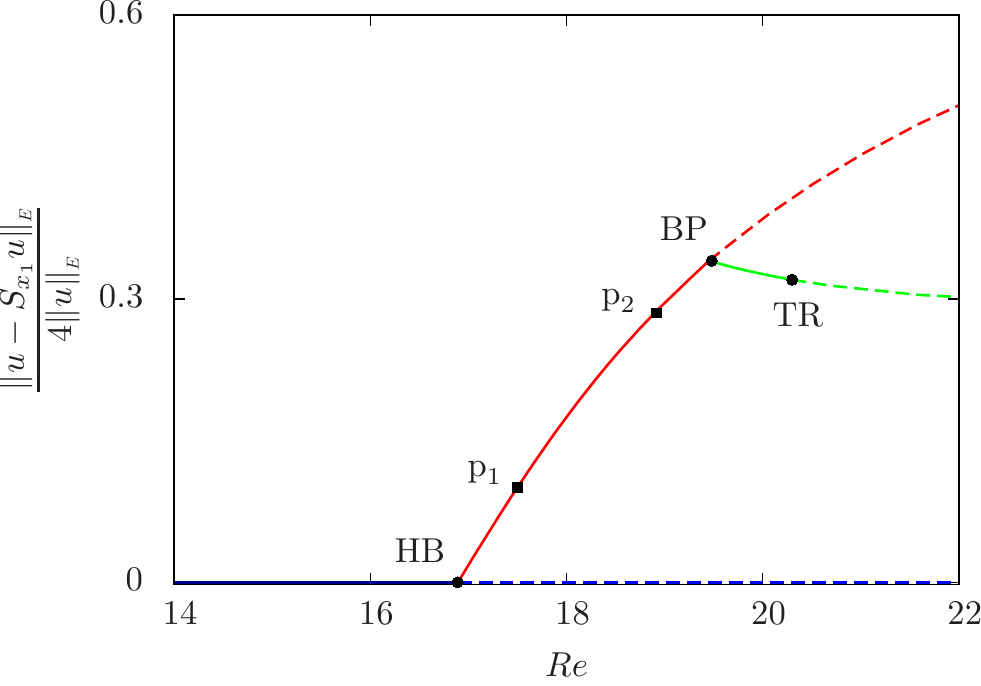}
\end{center}
\vspace{-.5cm}
\caption{ 
Partial numerical bifurcation diagram of the planar Taylor-Green flow,
where we use the geometric Reynolds number 
as the bifurcation parameter. Solutions that appear to be stable are shown with solid lines and unstable solutions with dashed lines. Shown on the vertical axis is the deviation from symmetry under the reflection $S_{x_1}$ as defined~\eqref{e:symmetrybreakingmeasure}. For periodic solutions, the maximal value over one period is shown. For increasing Reynolds number, the first instability of the viscous equilibrium, represented by the blue line, is a Hopf bifurcation (HB) to a branch of two-dimensional periodic orbits, shown in red. At the branch point labeled BP, the two-dimensional periodic orbit turns unstable and a family of three-dimensional periodic solutions branches off. The latter is stable near the branch point and turns unstable at a torus bifurcation point (TR). The solid squares correspond to the solutions proven to exist, as laid out in Theorems~\ref{thm:result1} and~\ref{thm:result2}.}
\label{bif_diagram}
\end{figure}
\subsection{New results for 2D periodic orbits}
\label{s:applic_2D}

Using a computer program in MATLAB, we computed a numerical approximation $\bW$ of a periodic orbit and applied Theorem~\ref{th:radii_pol_sym}, together with the bounds of Section~\ref{s:symmetrybounds}, to validate this solution  with explicit error bounds
, see Theorems~\ref{thm:result1} and~\ref{thm:result2} below. In the appendix, we describe how to recover errors bounds for the associated velocity $u$ and pressure $p$ that solve the Navier-Stokes equations. 

This validation was done for two separate orbits, represented by $\mathrm{p}_1$ and $\mathrm{p}_2$ in the bifurcation diagram of Figure~\ref{bif_diagram}. These two orbits are trivially invariant under $T_d$ and $S_{x_3}$. In addition, they are invariant under a symmetry group $G$ of order $16$, generated by the following three symmetries:
\begin{equation}
\label{def:sym}
g_1=S_{x_1}S_{x_2},\quad g_2=DS_{x_1},\quad g_3=P_4S_{x_1}R.
\end{equation}
The associated actions $a_{g_i}$ (see~\eqref{def:ag}) can be represented by the following matrices and vectors:
\begin{alignat*}{3}
C_{g_1}&=\begin{pmatrix}
-1 & 0 & 0 \\
0 & -1 & 0 \\
0 & 0 & 1
\end{pmatrix},&\quad \tilde C_{g_1}&=
\begin{pmatrix}
0 \\
0 \\
0
\end{pmatrix}
,&\quad D_{g_1}&=0,\\
C_{g_2}&=\begin{pmatrix}
-1 & 0 & 0 \\
0 & 1 & 0 \\
0 & 0 & 1
\end{pmatrix},&\quad \tilde C_{g_2}&=\frac{1}{2}\begin{pmatrix}
1 \\
1 \\
0
\end{pmatrix},&\quad D_{g_2}&=0,\\
C_{g_3}&=\begin{pmatrix}
0 & 1 & 0 \\
1 & 0 & 0 \\
0 & 0 & 1
\end{pmatrix},&\quad \tilde C_{g_3}&=\frac{1}{2}\begin{pmatrix}
1 \\
0 \\
0
\end{pmatrix},&\quad D_{g_3}&=\frac{1}{4}.
\end{alignat*}
These group actions can then be represented at the level of the Fourier coefficients of the vorticity, via $\alpha_g$ and $\beta_g$ (see~\eqref{e:symmetryequivalence}-\eqref{def:gamma} and Remark~\ref{rem:formula_alpha_beta}) given by
\begin{alignat*}{3}
\alpha_{g_1}(n,m) &= \left\{\begin{aligned}
&-1  &&m=1,2 \\
&1 &&m=3
\end{aligned}\right.
&&\qquad\text{and}\qquad&
\beta_{g_1}(n,m) &= ((-n_1,-n_2,n_3,n_4),m), \\
\alpha_{g_2}(n,m) & = \left\{\begin{aligned}
&(-1)^{n_1+n_2} &&m=1 \\
&-(-1)^{n_1+n_2} &&m=2,3
\end{aligned}\right.
&&\qquad\text{and}\qquad&
\beta_{g_2}(n,m) &= ((-n_1,n_2,n_3,n_4),m),\\
\alpha_{g_3}(n,m) & = -(-1)^{n_2}i^{n_4}
&&\qquad\text{and}\qquad&
\beta_{g_3}(n,m) &= \left\{\begin{aligned}
&((n_2,n_1,n_3,n_4),2) && m=1 \\
&((n_2,n_1,n_3,n_4),1) && m=2 \\
&((n_2,n_1,n_3,n_4),3) && m=3. \\
\end{aligned}\right.
\end{alignat*}

The symmetries are used to define the subspace $\Xred$ (as in
Section~\ref{s:redsymvar}), which is then used in
Theorem~\ref{th:radii_pol_sym}. In practice, these symmetries are of course
first found by inspecting the numerical solution. However, once a true solution
is shown to exist via Theorem~\ref{th:radii_pol_sym}, we automatically also get
a proof that this solution satisfies these symmetries.

The value of the other parameters ($\eta$, $N^\dag$, $\tilde N$, etc) that are
needed to defined $\Ared$ or used to compute the estimates of
Section~\ref{s:symmetrybounds}, are listed in Table~\ref{table:comput_para}. We choose a rectangular box of Fourier coefficients to compute the numerical approximate solutions:
\begin{equation}\label{e:rectangularSsol}
\set^{\sol}=\{n\in \Z^4_* : |n_1|\leq N_{x_1},\,|n_2|\leq N_{x_2},\,|n_3|\leq N_{x_3},\, |n_4|\leq N_t\}.
\end{equation}
See Remark~\ref{rem:choice_para} below for a brief discussion about how the
parameter values were chosen. 
\begin{theorem}
\label{thm:result1}
Let $\nu = 0.286$ and let $\rhs$ be given by~\eqref{eq:rhs_visc}. Let $G$ be the symmetry group generated by $g_1$, $g_2$, and $g_3$ given in~\eqref{def:sym}. Let $\bar{\omega}$ and $\bar{\Omega}$ be the Fourier data
time frequency provided in the file \verb+dataorbit1.mat+ (which can be downloaded at \cite{navierstokescode}). Let $r_{\sol}^{\omega}=2.6314 \cdot 10^{-5}$.
Then there exists a unique pair $(\tilde{\Omega},\tilde{\omega})$ satisfying the error estimate
\begin{equation}\label{e:thmbound}
  |\tilde{\Omega}-\bar{\Omega}| 
  + \sum_{m=1}^3 \| \tilde{\omega}^{(m)} - \bar{\omega}^{(m)} \|_{\ell^1_1} \leq r_{\sol}^{\omega} \, ,
\end{equation}
such that 
$\tilde{\omega}$ is an analytic, divergence free, $G$-invariant,  $2\pi/\tilde{\Omega}$-periodic solution of the vorticity equation~\eqref{eq:F_physical_space}.
\end{theorem}
\begin{proof}
Let $\bar{\phi}=\Pi \bar{\omega}$ and $\bb=(\bar{\Omega},\bar{\phi})$.
Let $\hat{\phi}=\frac{1}{2}(\bar{\phi}+\Pi\gamma_*\Sigma\bar{\phi})$ be used to define $\Fred_\phase$.	

We start by using Remark~\ref{r:makesymmetric} to replace $\bar{\phi}$ by a divergence free variant $\bar{\phi}_0 \in \Xred_\div$ such that $\tgamma_* \bar{\phi}_0=\bar{\phi}_0$.
This new $\bar{\phi}_0$ still has a finite number of nonzero modes: $(\bar{\phi}_0)_n=0$ for $n \notin \set^{\sol}$.
We take $\eta=1$ and perform an interval arithmetic computation to determine bounds
\begin{equation*}
  \sum_{m=1}^3 \| \bar\omega^{(m)}- (\Sigma \bar{\phi}_0)^{(m)} \|_{\ell^1_1}
  \leq r_0^\omega .
\end{equation*}
This bounds absorbs rounding errors introduced in various symmetrization steps.

Let $\Fred$ and $\XXred$ (with $\eta=1$) be as defined in Section~\ref{s:Fred}. 
We now use $\bb_0=(\bar\Omega,\bar{\phi_0})$ instead of $\bb$ as the numerical approximate solution
and apply the 2D estimates to obtain the bounds $Y_0^\ered$, $Z_0^\ered$, $Z_1^\ered$ and $Z_2^\ered$ defined in Section~\ref{s:symmetrybounds}.
Using the script \verb+runprooforbit1+, available at~\cite{navierstokescode}, with $\eta=1$ and the other parameters chosen as in Table~\ref{table:comput_para}, we checked that the inequalities~\eqref{hyp:cond_pol_sym} are satisfied. In order to make this verification rigorous, all the bounds are evaluated using the interval arithmetic package INTLAB~\cite{Intlab}. 

Recalling from \eqref{eq:weights_over_J} that $\weight_\j  =\eta^{|n|_1}$, one has that the symmetric weights $\symweight_\j
= \symweight_\j (\eta)$ given in \eqref{e:defxis} are continuous in the parameter $\eta$. Moreover, $\symweight_\j (\eta)$ is increasing in $\eta$.
Hence, the bounds $Y_0^\ered$, $Z_0^\ered$, $Z_1^\ered$ and $Z_2^\ered$ (given respectively in \eqref{eq:Y0red}, \eqref{eq:Z0red}, \eqref{eq:Z1red} and \eqref{eq:Z2red}) are also continuous functions of $\eta$. By continuity of the bounds in $\eta$, since the strict inequalities in \eqref{hyp:cond_pol_sym} are satisfied for $\eta=1$, they still hold for all $\eta \in [1,\tilde{\eta}]$ for some $\tilde{\eta}>1$. In particular, we can use Theorem~\ref{th:radii_pol_sym} for any $\eta \in (1,\tilde{\eta}]$, which yields the existence of a unique zero $\tilde{\varphi}=(\tilde\Omega,\tilde\phi)$
of $\FFred$, 
so that, by monotonicity of $\symweight_\j (\eta)$,
\begin{equation*}
|\tilde{\Omega}-\bar{\Omega}| 
  + \sum_{j \in \Jred}  \symweight_\j (1)   | \tilde{\phi}_j - (\bar{\phi_0})_j | \leq
|\tilde{\Omega}-\bar{\Omega}| 
  + \sum_{j \in \Jred} \symweight_\j (\eta)  \, | \tilde{\phi}_j - (\bar{\phi_0})_j |  \leq  r,
\end{equation*} 
for each $r \in [r_{\min}(\eta),r_{\max}(\eta))$, with
$r_{\min}(\eta)$ and $r_{\max}(\eta)$ given in~\eqref{e:rmin}
and~\eqref{e:rmax}, respectively. By continuity we find
\begin{equation*}
	|\tilde{\Omega}-\bar{\Omega}| 
	  + \sum_{j\in\Jred}  |G \acts j| \, | \tilde{\phi}_j - (\bar{\phi_0})_j |
	  \leq  r_{\min},
\end{equation*}
where $r_{\min} = r_{\min}(1)$, and have that the solution $\tilde\varphi$ is unique in the ball of radius $r$ around 
$\bar{\varphi_0}$ (with $\eta=1$) for any $r < r_{\max} =  r_{\max}(1)$.

To translate this existence and uniqueness result to the vorticity,
let $\tilde{\omega}=\Sigma\tilde\phi\in (\ell^1_{\tilde{\eta}})^3$.
The asserted properties of $\tilde{\omega}$ follow
from the arguments given towards the end of the proof of Theorem~\ref{th:radii_pol_sym},
and by applying the triangle inequality after checking that $r_{\min} + r_0^\omega < r_{\sol}^\omega < r_{\max}-r_0^\omega$.
\end{proof}

The solution in Theorem~\ref{thm:result1} represents the point $\mathrm{p}_1$ in Figure~\ref{bif_diagram}.
Four snap shots of the third component (the nontrivial one for a 2D solution) of the vorticity field~$\bar\omega$ from Theorem~\ref{thm:result1} are shown in Figure~\ref{2DPO}. In exactly the same way, but with the script \verb+runprooforbit2+ we prove
\begin{theorem}
\label{thm:result2}
Let $\nu=0.265$ 
and $r_{\sol}^\omega=2.2491 \cdot 10^{-6}$. 
Let $\bar{\omega}$ and $\bar{\Omega}$ be the Fourier data
time frequency provided in the file \verb+dataorbit2.mat+. 
Then there exists a unique pair $(\tilde{\Omega},\tilde{\omega})$ satisfying the error estimate~\eqref{e:thmbound},
such that 
$\tilde{\omega}$ is an analytic, divergence free, $G$-invariant,  $2\pi/\tilde{\Omega}$-periodic solution of the vorticity equation~\eqref{eq:F_physical_space}.
\end{theorem}

\begin{remark}
 As a direct corollary of Theorem~\ref{thm:result2}, we get a validated  velocity $u$ and pressure~$p$ that solves the Navier-Stokes equations via  Lemma~\ref{lem:eq_NS_vorticity}, with explicit error bounds presented in the  appendix (Lemma~\ref{lem:error_estimates}). In particular,
 Theorem~\ref{thm:NS_result} holds.
\end{remark}

The solution from Theorem~\ref{thm:result2}, corresponding to the point $\mathrm{p}_2$ in Figure~\ref{bif_diagram}, is depicted in Figure~\ref{2DPO_intro}.
The script \verb+mimicproofsfloats+ may be used to perform faster but non-rigorous versions of the computations using floating point calculations rather than interval arithmetic. The computational cost for Theorem~\ref{thm:result2} is considerably higher than for Theorem~\ref{thm:result1}, as more Fourier modes need to be considered to make the proof work.

\begin{figure}[t]
\begin{center}
\hspace{2cm}
\includegraphics[width=0.8\textwidth]{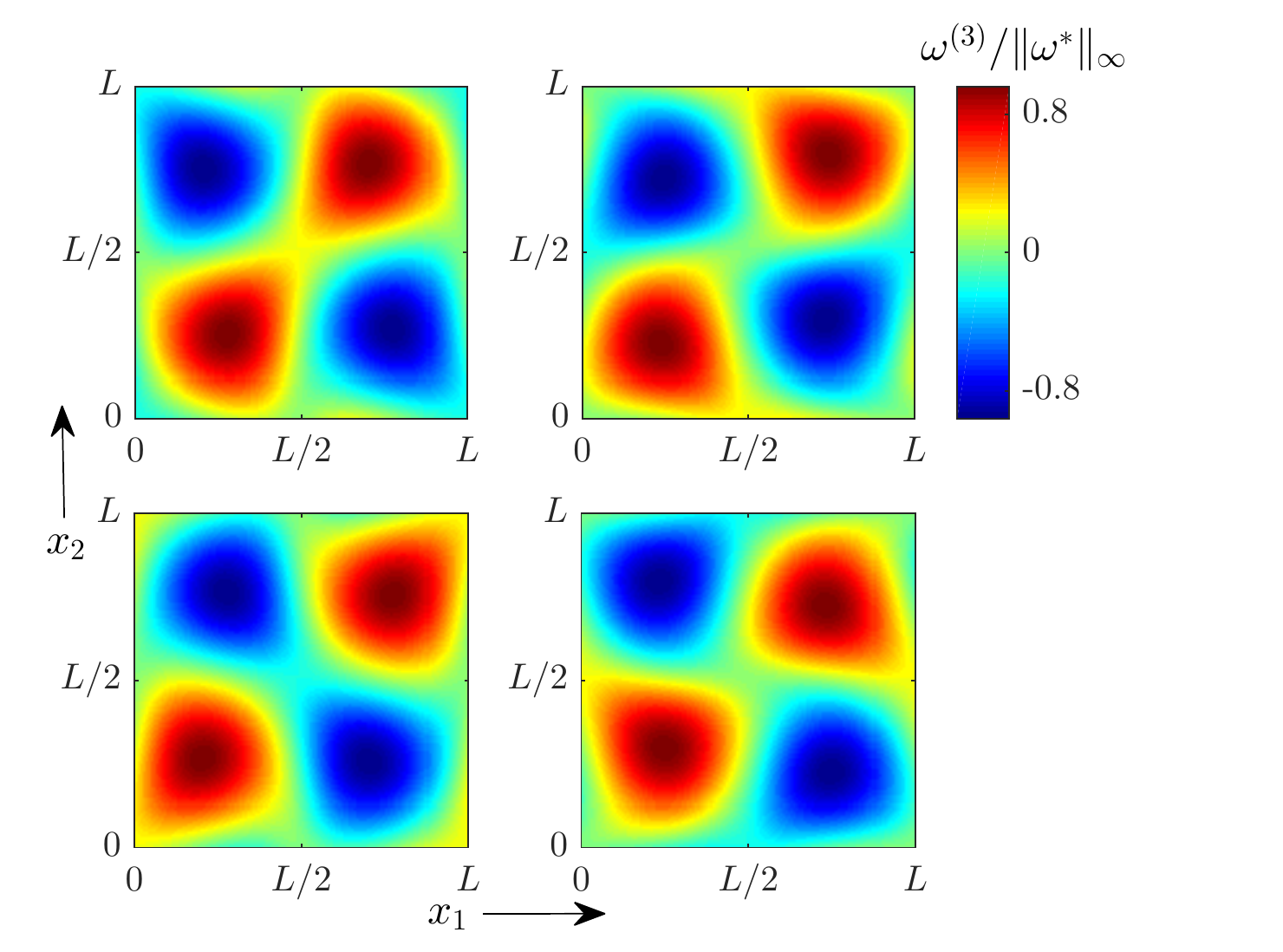}
\end{center}
\vspace{-.5cm}
\caption{
Physical space portrait of a two-dimensional periodic orbit at $\nu=0.286$ ($Re \approx 17.5$),
corresponding to the solution labeled $\mathrm{p}_{1}$ in the bifurcation diagram in Figure~\ref{bif_diagram}. Shown is the value of $\omega^{(3)}$ at four time instances
a quarter of the period apart, normalized by the amplitude of the viscous
equilibrium at the same viscosity.}
\label{2DPO}
\end{figure}

\begin{table}[t]
\begin{center}
\begin{tabular}{|c|c|c|c|c|c|c|c|c|c|c|}\hline
              & $\eta$ & $N_{x_1}$ & $N_{x_2}$ & $N_{x_3}$ & $N_t$ & $N^{\dagger}$ & $\widetilde{N}$ & RAM (GB) & CPU days \\\hline
$\mathrm{p}_1$  & $1$ &  17  &   17  &   0  &  11   & 130       &  265   & 10   & 6 \\\hline 
$\mathrm{p}_2$  & $1$ & 21  &   21  &   0  &  16   & 210       &  425   & 110   & 95 \\\hline
\end{tabular}
\caption{Parameters for the two rigorously computed solutions of the Navier-Stokes equations with Taylor-Green forcing~\eqref{eq:TG_intro}. The number $N_{x_1}$, $N_{x_2}$, $N_{x_3}$ and $N_t$ define the set $\set^{\sol}$ in~\eqref{e:rectangularSsol}. The solutions are indicated by labels $\mathrm{p}_1$ and $\mathrm{p}_2$ in bifurcation diagram \ref{bif_diagram} and are computed for $\nu=0.286$ and $\nu=0.265$ respectively. The computations were timed on an Intel Xeon E5-1620v2 with a 3.7GHz clock speed.}
\label{table:comput_para}
\end{center}
\end{table}
\begin{remark}
Since the size of the estimates deteriorates for larger values of $\eta$, we choose $\eta =1$.
By continuity of the bounds in $\eta$, as explained in the proof of Theorem~\ref{thm:result1}, this is enough to prove that we obtain an analytic solution, although we do not get an explicit decay rate of the Fourier coefficients. If obtaining such decay rate was needed, for instance to obtain error estimates on derivatives of the solution, or to obtain an explicit domain of analyticity, this could be achieved by using some explicit $\eta>1$ to compute the bounds $Y_0^\ered$, $Z_0^\ered$, $Z_1^\ered$ and $Z_2^\ered$.
\end{remark}
	
\begin{remark}
\label{rem:choice_para}
We briefly discuss the choices made for the other computational parameters.
The $Z_2^{\text{\textup{red}}}$ bound is relatively insensitive to the computational parameters (as it roughly measures the size of the second derivative at the numerical approximation). We then consider a set $\set^\sol$ of rectangular shape (as explained in Table~\ref{table:comput_para}), with sufficiently many modes $N^{\sol}_{x_1}=N^{\sol}_{x_2}$ and $N^{\sol}_t$ (for the 2D solution $N^{\sol}_{x_3}=0$) to make the bound $Y_0^{\text{\textup{red}}}$ on the residual so small that $2 Y_0^{\text{\textup{red}}} Z_2^{\text{\textup{red}}}$ is roughly of size $10^{-2}$. 
The choice of this threshold size is based on two criteria. Primarily, we choose it small, so that at the next step (see below) $Z_1^{\text{\textup{red}}}$ is allowed to be very close to $1$, as decreasing~$Z_1^{\text{\textup{red}}}$ is computationally very costly. On the other hand, the threshold should not be too small, as that would require many modes in $\set^\sol$ to decrease the residual sufficiently.

Subsequently, one may determine from a relatively cheap calculation how large
$\tilde{N}$ needs to be for the tail bound to satisfy
$1-(Z_1^{\text{\textup{red}}})^{\text{\textup{tail}}} > (2
Y_0^{\text{\textup{red}}} Z_2^{\text{\textup{red}}})^{1/2}$. We note that $Z_0^{\text{\textup{red}}}$ is negligible in practice. Finally, it
requires some experimentation to determine how large one needs to choose~$N^\dagger$ in order to also satisfy
$1-(Z_1^{\text{\textup{red}}})^{\text{\textup{finite}}} > (2
Y_0^{\text{\textup{red}}} Z_2^{\text{\textup{red}}})^{1/2}$. 
\end{remark}

\section{Appendix}
\label{s:appendix}

The estimates obtained in Section~\ref{s:estimates} and
Section~\ref{s:symmetrybounds} are used as input for 
Theorem~\ref{th:radii_pol_sym}, which allows us to validate symmetric periodic
solutions $\omega$ of the vorticity equation, with explicit error bounds, as
illustrated in Section~\ref{s:application}. In this appendix, we describe how to
recover errors bounds for the associated velocity $u$ and pressure
$p$ that solve the Navier-Stokes equations.

We start with a variation, adapted to our framework, of the classical result stating that a curl-free vector field can be written as a gradient, which was used already in the proof of Lemma~\ref{lem:eq_NS_vorticity}.

\begin{lemma}
\label{lem:exist_grad}
Let $\Phi\in \left(\C^3\right)^{\Z^4}$ satisfy
\begin{equation*}
 \left\{
 \begin{aligned}
 &\nabla\times\Phi=0 \\
 &\Phi_{n}=0,\qquad\text{for all } \tn=0.
 \end{aligned}
 \right.
 \end{equation*} 
Then the map $\Gamma:\left(\C^3\right)^{\Z^4} \to \C^{\Z^4}$ constructed component-wise as
\[
  (\Gamma \Phi)_n = 
  \begin{cases}
  -i\Phi^{(k)}_n/n_k & \text{if } n_k \neq 0 \text{ for any } k=1,2,3 \\
  0 & \text{if } \tn=0,
  \end{cases}
\]
is well defined, and $p=\Gamma\Phi$ satisfies $\Phi=-\nabla p$.
\end{lemma}
\begin{proof}
To ensure that $\Gamma$ is well defined, it suffices to show that, for all $n\in\Z^4$ and $l,m \in \{1,2,3\}$
\begin{equation}
\label{eq:N_well_defined}
\text{if } n_l,n_m \neq 0 \quad\text{ then }\quad \frac{\Phi_n^{(l)}}{n_l}=\frac{\Phi_n^{(m)}}{n_m}.
\end{equation}
Indeed, since $\nabla\times\Phi=0$, we have that for all $n\in\Z^4$ and all 
$l,m\in\{1,2,3\}$,
\begin{equation}
\label{eq:rot_free_coord}
n_l\Phi_n^{(m)}=n_m\Phi_n^{(l)},
\end{equation}
which immediately yields~\eqref{eq:N_well_defined}. Therefore $p=\Gamma\Phi$ is well defined, and we are left to check that $\Phi=-\nabla p$. If $\tn=0$ then we have 
\begin{equation*}
 \Phi_n=-\left(\nabla p\right)_n,
\end{equation*} 
because we assumed $\Phi_{n}=0$ for all $\tn=0$. If $\tn\neq 0$, for any $l\in\{1,2,3\}$ we distinguish between two cases. If $n_l\neq 0$, then
\begin{equation*}
-\left(\nabla p\right)_n^{(l)} = -in_l p_n 
=-in_l\frac{-i \Phi_n^{(l)}}{n_l} 
= \Phi_n^{(l)}.
\end{equation*}
If $n_l= 0$, then $-\left(\nabla p\right)_n^{(l)}=0$, but there exists an $m\neq l$ such that $n_m\neq 0$ and thus by~\eqref{eq:rot_free_coord} we find $\Phi_n^{(l)}=0$, i.e. $-\left(\nabla p\right)_n^{(l)}= \Phi_n^{(l)}$ also holds.
\end{proof}

The above lemma can be used in the context of Navier-Stokes equations, to recover the pressure from the velocity (we recall that the velocity itself is recovered from the vorticity via $u=M\omega$). We point out that an alternative (arguably more classical) approach is to define $p$ as the solution of the Poisson equation
\begin{equation}
\label{eq:Poisson_p}
-\Delta p= \nabla \cdot\left((u\cdot\nabla)u\right) - \nabla\cdot f,
\end{equation}
satisfying
\begin{equation*}
\int_{\T^3} p(x,t) dx = 0.
\end{equation*}
Indeed, the latter approach is going to be useful in the sequel, as~\eqref{eq:Poisson_p} allows to recover sharper error bounds for the pressure
(compared to using Lemma~\ref{lem:exist_grad} only).

Our aim is to derive error estimates for the velocity and the pressure, that can be applied as soon as we have validated a divergence-free solution $W$ of the vorticity equation via Theorem~\ref{th:radii_pol} or Theorem~\ref{th:radii_pol_sym}.

\begin{lemma}
\label{lem:error_estimates}
Assume that for some $\bW=(\bO,\bw)\in\XX^\div$, $\eta>1$, we have proven the existence of $r>0$ and of $W=(\Omega,\omega)\in\B_{\XX^\div}(\bar W,r)$ such that $\F(W)=0$. Define 
\begin{equation*}
u=M\omega,\quad \bu=M\bw, \quad  p=\Gamma\Phi,
\end{equation*}
where $\Phi$ is defined in~\eqref{e:defPhi} and $\Gamma$ is defined as in Lemma~\ref{lem:exist_grad}. We also consider the sequence $\bar p\in \C^{\Z^4}$  defined as 
\begin{equation*}
\bar p_n = 
\left\{\begin{aligned}
& 0 \quad & \text{if } n=(\tn,n_4)\in\Z^3\times\Z,\ \tn= 0, \\
& -\frac{1}{\tn^2}\sum_{l=1}^3 n_l \left(\left[\left(\bu \star \tD\right) \bu\right]^{(l)}_n +i f^{(l)}_n\right) \quad & 
\text{if } n=(\tn,n_4)\in\Z^3\times\Z,\ \tn\neq 0.
\end{aligned}\right.
\end{equation*}
Then, we have the following error estimates for the velocity and the pressure:
\begin{equation*}
\left\Vert u - \bar u \right\Vert_\X \leq r \quad \text{and}\quad \left\Vert p - \bar p \right\Vert_{\ell^1_\eta} \leq \left(2\left\Vert \bar u \right\Vert_\X + r\right)r.
\end{equation*}
\end{lemma}
\begin{remark}
As explained in Remark~\ref{rem:comparison_of_norms}, these weighted $\ell^1$-norms control the $\CC^0$-norms of the errors (explicitly). Notice also that, even though we used $\eta=1$ to validate the vorticity in Theorem~\ref{thm:result1} and Theorem~\ref{thm:result2}, by continuity of the estimates with respect to $\eta$ we get \emph{for free} a validation for some $\tilde\eta>1$ (see the proof of Theorem~\ref{thm:result1}), and thus Lemma~\ref{lem:error_estimates} is directly applicable.
\end{remark}
\begin{proof}
The error estimate for the velocity simply follows from the definition of $M$:
\begin{align*}
\left\Vert u - \bar u \right\Vert_\XX &= \sum_{m=1}^3 \left\Vert \left(M(\omega-\bw)\right)^{(m)}\right\Vert_{\ell^1_\eta} \\
&\leq  \sum_{m=1}^3 \left\Vert \left(\omega-\bw\right)^{(m)}\right\Vert_{\ell^1_\eta} \\
& = \left\Vert W - \bar W \right\Vert_\XX \\
&\leq r.
\end{align*}
To obtain the error estimate for the pressure, we use the fact that $(u,p)$ are smooth solutions of Navier-Stokes equations (see Lemma~\ref{lem:eq_NS_vorticity}), and thus~\eqref{eq:Poisson_p} holds. In Fourier space, this reduces to
\begin{equation*}
p_n = 
\left\{\begin{aligned}
& 0 \quad & \text{if } n=(\tn,n_4)\in\Z^3\times\Z,\ \tn= 0, \\
& -\frac{1}{\tn^2}\sum_{l=1}^3  n_l \left(\left[\left(u \star \tD\right) u\right]^{(l)}_n +i f^{(l)}_n\right) \quad &  \text{if }  n=(\tn,n_4)\in\Z^3\times\Z,\ \tn\neq 0.
\end{aligned}\right.
\end{equation*}
Since both $u$ and $\bu$ are divergence-free, by Lemma~\ref{lem:trick_derivative} we can write, for all $\tn\neq 0$,
\begin{equation*}
p_n = -\frac{1}{\tn^2}\sum_{l=1}^3\sum_{m=1}^3 n_l n_m \left[u^{(m)} \ast u^{(l)}\right]_n \quad\text{and}\quad \bar p_n = -\frac{1}{\tn^2}\sum_{l=1}^3\sum_{m=1}^3 n_l n_m \left[\bu^{(m)} \ast \bu^{(l)}\right]_n.
\end{equation*}
We then estimate
\begin{align*}
\left\Vert p - \bar p \right\Vert_{\ell^1_\eta} &\leq \sum_{l=1}^3\sum_{m=1}^3 \sum_{\tn\neq 0} \frac{\vert n_l\vert \vert n_m\vert}{\tn^2} \left(\left[\vert (u-\bu)^{(m)} \vert  \ast \vert u^{(l)}\vert \right]_n + \left[\vert\bu^{(m)}\vert \ast \vert(u-\bu)^{(l)}\vert \right]_n\right)  \eta^{\left\vert n\right\vert_1} \\
&\leq  \sum_{l=1}^3\sum_{m=1}^3 \sum_{\tn\neq 0} \left(\left[\vert (u-\bu)^{(m)} \vert  \ast \vert u^{(l)}\vert \right]_n + \left[\vert\bu^{(m)}\vert \ast \vert(u-\bu)^{(l)}\vert \right]_n\right)  \eta^{\left\vert n\right\vert_1} \\
&\leq  \sum_{l=1}^3\sum_{m=1}^3  \left(\left\Vert (u-\bu)^{(m)} \right\Vert_{\ell^1_\eta}  \left\Vert u^{(l)}\right\Vert_{\ell^1_\eta} + \left\Vert \bu^{(m)}\right\Vert_{\ell^1_\eta} \left\Vert(u-\bu)^{(l)}\right\Vert_{\ell^1_\eta}\right),
\end{align*}
where $\left\vert \cdot\right\vert$ applied to a sequence must be understood component-wise, i.e. $\vert u^{(l)}\vert$ is the sequence whose $n$-th element is equal to $\vert u^{(l)}_n\vert$. Finally, we obtain
\begin{align*}
\left\Vert p - \bar p \right\Vert_{\ell^1_\eta} &\leq \sum_{m=1}^3 \left\Vert (u-\bu)^{(m)} \right\Vert_{\ell^1_\eta} \left(\sum_{l=1}^3 \left\Vert u^{(l)}\right\Vert_{\ell^1_\eta} + \sum_{l=1}^3 \left\Vert \bu^{(l)}\right\Vert_{\ell^1_\eta} \right) \\
&\leq \sum_{m=1}^3 \left\Vert (u-\bu)^{(m)} \right\Vert_{\ell^1_\eta} \left(\sum_{l=1}^3 2\left\Vert \bu^{(l)}\right\Vert_{\ell^1_\eta} + \sum_{l=1}^3 \left\Vert (u-\bu)^{(l)}\right\Vert_{\ell^1_\eta} \right) \\
&\leq \sum_{m=1}^3 \left\Vert (u-\bu)^{(m)} \right\Vert_{\ell^1_\eta} \left(2\left\Vert \bar u\right\Vert_{\X} + r\right) \\
&\leq \left(2\left\Vert \bar u \right\Vert_\X + r\right)r. \qedhere
\end{align*}
\end{proof}

\bibliographystyle{abbrv}
\bibliography{NS_2D}

\end{document}